\documentclass[11pt]{amsart}
\usepackage{amscd}
\usepackage{pifont}
\usepackage{epsfig}
\usepackage{amsmath,amsfonts}
\numberwithin{equation}{section}
\usepackage{graphicx, color}
\usepackage[backend=biber,style=numeric,
            maxbibnames=99,   
            maxcitenames=99   
           ]{biblatex}
\usepackage[colorlinks=true,linkcolor=blue,citecolor=blue,urlcolor=blue]{hyperref}

\usepackage{cleveref}
\newtheorem{theorem}{Theorem}
\newtheorem{lemma}{Lemma}[section]
\newtheorem{proposition}{Proposition}[section]
\newtheorem{corollary}{Corollary}[section]
\theoremstyle{definition}
\newtheorem{definition}{Definition}[section]

\newtheorem{remark}{Remark}[section]
\usepackage{enumerate}
\def\dis{\displaystyle}
	\newcommand{\dint}{\dis\int}
	
\usepackage{multirow}
\usepackage{hhline}

\newcommand\be{\begin{equation}}
\newcommand\ee{\end{equation}}
\newcommand\bea{\begin{eqnarray}}
\newcommand\eea{\end{eqnarray}}
\newcommand\beaa{\begin{eqnarray*}}
\newcommand\eeaa{\end{eqnarray*}}
\newcommand\bay{\begin{array}}
\newcommand\eay{\end{array}}
\newcommand\ba{\begin{align}}
\newcommand\ea{\end{align}}
\newcommand{\R}{\mathbb{R}}

\numberwithin{equation}{section}

\title[nonlocal diffusion systems]{Harnack-type inequalities and traveling waves for non-cooperative nonlocal diffusion systems}
\author[B.-S. Chen]{Bo-Sheng Chen}
\address{Department of Applied Mathematics, National Yang Ming Chiao Tung University, Hsinchu 300, Taiwan.}
\email{boshengchen.sc14@nycu.edu.tw}
\author[C.-H. Wu]{Chang-Hong Wu}
\address{Department of Applied Mathematics, National Yang Ming Chiao Tung University, Hsinchu 300, Taiwan.}
\email{changhong@math.nctu.edu.tw}

\date{\today}
\begin{document}


\setlength{\baselineskip}{14pt}
\begin{abstract}
In this paper, we establish a system-wide $\ell^1$-norm Harnack-type inequality for positive solutions of weakly coupled nonlocal diffusion systems on $\mathbb{R}$, without assuming cooperativity or irreducibility of the coupling matrices. As a primary application, we integrate this analytical tool with a rescaling argument to establish the existence of the minimal wave speed for traveling waves for a class of non-cooperative nonlocal diffusion systems with network structures. Furthermore, we derive the precise asymptotic behavior of wave profiles at negative infinity. This is achieved by combining Harnack-type inequalities and Ikehara's theorem with Riesz projections to analyze singularities of vector-valued Laplace transforms. 
\end{abstract}

\maketitle
\section{Introduction}

Positive solutions of elliptic and parabolic weakly coupled systems often satisfy componentwise Harnack inequalities (Chen and Zhao \cite{chen1997harnack}, Arapostathis et al. \cite{arapostathis1999harnack}, F{\"o}ldes and Pol{\'a}cik \cite{foldes2009cooperative}); stronger forms relating different components typically require irreducibility of the cooperative coupling. Such inequalities provide local comparisons for each
individual component at nearby spatial points.
Beyond these local comparisons, Harnack inequalities act as a quantitative nondegeneracy principle. 

In wave propagation problems, these Harnack estimates help propagate positivity across the coupled system, establish persistence estimates, and prevent partial extinction.
They also play a central role in compactness arguments: after translation and normalization, Harnack estimates allow one to extract nontrivial positive entire limits. This compactness mechanism is a key step in connecting the existence of positive entire solutions of the linearized system to positive exponential modes whose exponents satisfy the associated characteristic equation  \cite{girardin2017non,lam2018traveling}.

In wave propagation problems involving nonlocal diffusion, related Harnack-type inequalities for scalar equations or two/three-component systems with particular
coupling have been established in \cite{zhang2012spreading,dong2017asymptotic}. 
Extending such estimates to general fully coupled $n$-component systems, however, is far from straightforward. In view of these difficulties, it is natural to ask whether the full strength of a componentwise Harnack inequality is needed in the study of wave propagation in nonlocal diffusion systems. More specifically, one may ask whether a weaker quantitative nondegeneracy estimate, available under less restrictive structural assumptions, would still provide sufficient control for the relevant normalization and limiting
arguments. This question motivates us to seek aggregate control of the entire vector solution rather than separate estimates for its individual components.

To this end, we establish a system-wide $\ell^1$-norm Harnack-type inequality. More precisely, our inequality provides a local comparison of the total size of a positive vector solution at nearby spatial points. Notably, it requires neither cooperativity nor irreducibility of the coefficient matrices. Although weaker than a full componentwise Harnack inequality, it retains precisely the quantitative control needed for the normalization, persistence, and limiting arguments arising in the study of nonlocal wave propagation.
The estimate may also be of independent interest in the linear theory of nonlocal diffusion systems.

To precisely state this inequality, we first formulate the basic assumption on the dispersal kernel.
Throughout this paper, we always assume that the dispersal kernel $J:\mathbb{R}\to\mathbb{R}$ satisfies (J0):
\begin{enumerate}
    \item[(J0)]$J(s)\geq 0$ for all $s\in\mathbb{R}$; $J(0)>0$; $J$ is continuous at $0$; $J\in L^1(\mathbb{R})$ with $\int_{\mathbb{R}}J(s)ds=1.$
\end{enumerate}
Our $\ell^1$-norm Harnack-type inequality is stated as follows.

\begin{theorem}[$\ell^1$-norm Harnack-type inequality]\label{generalizedHarnack}
   Assume that $J$ satisfies {\rm{(J0)}} and $c$ is a nonzero real constant. 
Let $\mathcal{D}(s)=\mathrm{diag}(d_1(s),\dots,d_n(s))$, in which $d_i(\cdot)\in\mathcal{C}(\mathbb{R})$ for each $i=1,\dots,n$. 
   Suppose that there exist $d_{\max}\geq d_{\min}>0$ such that $d_{\min}\leq  d_i(s)\leq d_{\max}$ for each $i=1,\dots,n$ and each $s\in\mathbb{R}$.
   Let $P_1$ and $P_2$ be any two $n\times n$ real matrices
  and let $p=(p_1,\ldots,p_n)
    \in\mathcal{C}^1\big(\mathbb{R};(0,\infty)^n\big)$
satisfy, componentwise,
    \begin{align}\label{linearinequality}
        P_1p(s)\leq cp'(s)-\mathcal{D}(s)\bigg(\int_{\mathbb{R}}J(s-s')p(s')ds'-p(s)\bigg)\leq P_2p(s),\quad s\in\mathbb{R}.
    \end{align}
Then there exists a constant $M>0$, depending only on
$J$, $c$, $d_{\max}$, $d_{\min}$, $P_1$, and $P_2$, such that the following statements hold.
    \begin{enumerate}[(a)]
        \item 
        For all $s\in\mathbb{R}$,
        \begin{align*}
            \frac{1}{|p(s)|_1}\left|\sum_{i=1}^np_i'(s)\right|\leq M.
        \end{align*}
        where $|p(s)|_1:=\sum_{i=1}^n|p_i(s)|.$
        \item For any $r>0,$ 
        \begin{align*}
            e^{-M r}\leq\frac{|p(\xi+s)|_1}{|p(\xi)|_1}\leq e^{M r}\mbox{ for all }\xi\in\mathbb{R}\mbox{ and }s\in[-r,r].
        \end{align*}
    \end{enumerate}    
\end{theorem}

Note that the constant $M$ is independent of the choices of $p(s)$ and of $r$.
Note also that these estimates hold for arbitrary $n\times n$ matrices $P_1$ and $P_2$, meaning that no structural assumptions, such as the irreducibility or cooperativity, are imposed on the coupling.
Furthermore, these estimates do not require the dispersal kernel $J$ to have compact support. They also do not need any specific tail decay, such as exponential decay, to establish the Harnack-type bounds.
With these estimates at hand, we can effectively address traveling-wave problems in non-cooperative systems with nonlocal dispersal. 

Wave propagation in reaction-diffusion systems with nonlocal dispersal has been extensively investigated for scalar equations
\cite{schumacher1980travelling,bates1997traveling,chen1997existence,carr2004uniqueness,coville2005propagation,coville2006uniqueness,coville2007non,coville2008existence,zhang2012spreading,shen2012traveling,Rawal2015Spreading,berestycki2016non,alfaro2023quantifying},
and monotone systems where comparison principles hold
\cite{liang2007asymptotic,fang2015bistable,zhang2018traveling,zhang2020propagation}. 
However, in many biological contexts--most notably multi-group disease transmission and predator-prey dynamics--the underlying systems are inherently non-cooperative
\cite{sherratt2016invasion, shu2019traveling,ducrot2019spreading,yang2020wave, zhou2020traveling,hao2021traveling,ai2023traveling,jiang2023wave,ducrot2025travelling,zhang2025spreading}, where the loss of monotonicity may significantly complicate the analysis. 
Specifically, the absence of a standard comparison principle for the full system invalidates classical monotone-iteration techniques.
More crucially, the nonexistence of a traveling wave below the critical speed cannot be established through the standard construction of lower barriers. 

To address these difficulties,
we consider a class of non-cooperative nonlocal diffusion systems with a network-coupled structure, motivated by multi-group disease transmission models.
The model consists of $n+1$ compartments, including the susceptible class $u$ and the $i$-th carrier (host) class $v_i$ ($i=1,\dots,n$), and is given by
\begin{align}\label{epidemicmodel}
\begin{cases}
    \dis\frac{\partial u}{\partial t}(x,t)=d_0\int_{\mathbb{R}}J(x-y)\big(u(y,t)-u(x,t)\big)dy+f(u)-g_0(u,v),\quad x\in\mathbb{R},\quad t>0\\[2ex]
    \dis\frac{\partial v_i}{\partial t}(x,t)=d_i\int_{\mathbb{R}}J(x-y)\big(v_i(y,t)-v_i(x,t)\big)dy+g_i(u,v),\quad x\in\mathbb{R},\quad t>0,\quad i=1,\dots,n,
\end{cases}
\end{align}
where $J$ is the dispersal kernel, $v=(v_1,\dots,v_n),$ and $f, g_0,g_i$ $(i=1,\dots,n)$ are given by
\begin{align}
&f(u)=\delta(K-u),\quad g_0(u,v)=\sum_{i=1}^ng_i^0(u,v),\quad g_i(u,v)=g_i^0(u,v)+\sum_{j=1}^nm_{ij}v_j,
\label{f}
\end{align}
with the incidence terms
\[
g_i^0(u,v)=u\sum_{j=1}^ng_{ij}^0(v_j).
\]
The function $g_i^0$ stands for the incidence term generating new infections in class $i$. Each $g_{ij}^0$ reflects that a susceptible host $u$ contacts an infected host $v_j$ of type $j$ and subsequently becomes an infected host $v_i$ of type $i$.
The coefficients $d_j~(j=0,1,\dots,n),\delta$ and $K$ are positive constants.
The transfer matrix $\mathbb{M}=\begin{bmatrix}
            m_{ij}\end{bmatrix}$ satisfies 
            $m_{ij}\geq 0$ for $i\neq j$ and $m_{ii}\leq 0$ for all $i=1,\dots,n.$ 
            Thus, $\mathbb{M}$ is an essentially nonnegative matrix.

In this paper, we assume that for each $i,j\in\{1,\dots,n\},$ the function $g_{ij}^0\in\mathcal{C}^2(\mathbb{R})$ satisfies (C1)--(C4):
\begin{enumerate}
    \item[(C1)]$g_{ij}^0(0)=0$ and $g_{ij}^0(\xi)\ge0$ for each $\xi\ge0$;
    \item[(C2)]$(g_{ij}^0)'(\xi)\ge 0$ for all $\xi\ge0$;
    \item[(C3)] $g_{ij}^0(\xi)> 0$ for all $\xi>0$ whenever $(g_{ij}^0)'(0)>0$;
    \item[(C4)] $(g_{ij}^0)''(\xi)\leq 0$ for all $\xi\ge0$. 
\end{enumerate}
The assumption (C1) reflects basic biological feasibility: 
no new infection is generated in the absence of infected hosts, and the 
incidence rate remains nonnegative; 
(C2) indicates that the infection risk monotonically increases with the infected population; 
(C3) ensures that whenever a specific infected group can spread the disease at zero infected density, it will always generate new cases whenever it is present; 
(C4) means that the marginal transmission contribution is nonincreasing as the infected population grows.

A typical example of $g_{ij}^0$ is
\begin{align}\label{g0ij}
    g_{ij}^0(v_j)=\frac{\beta_{ij}v_j}{1+\gamma_{ij}v_j}.
\end{align}
where $\beta_{ij},\gamma_{ij}$ are nonnegative constants.
In particular, by appropriately choosing the dimension $n$, the incidence functions $g_{ij}^0$, and the transfer matrix $\mathbb{M}$, the framework of \eqref{epidemicmodel} encompasses a large class of disease-transmission models, such as the endemic SIR model ($n=1$), the SEIR model ($n=2$), and the two-strain model with mutations ($n=2$), where $\mathbb{M}$ captures the latency rates or the mutation rates
(see, e.g., \cite{martcheva2015introduction}).



System \eqref{epidemicmodel} may be viewed as a nonlocal-diffusion extension of the network epidemic model studied in \cite{lam2018traveling}. Indeed, when the nonlocal dispersal operators are replaced by standard diffusion operators and $g_{ij}^0$ is chosen as in \eqref{g0ij}, system \eqref{epidemicmodel} reduces to the model considered there. However, this extension is not trivial since the nonlocal dispersal
operators do not provide the local elliptic regularity or the classical componentwise Harnack inequalities available for random
diffusion.

A feature of invasion systems is the emergence of a minimal wave speed that separates propagation from non-propagation. In the present setting, the linearization at the disease-free equilibrium yields a characteristic equation that determines a threshold speed $c^*$. Our next main result shows that this threshold is sharp.
For the existence of waves for supercritical speeds $c>c^*$, despite the non-cooperative nature of the full system \eqref{epidemicmodel}, we can adopt a truncated-problem approach to establish the existence of solutions on bounded intervals via Schauder's fixed point theorem, followed by a compactness argument to construct a traveling wave on the entire real line (see also \cite{berestycki2005quenching,berestycki2009non,ducrot2012qualitative,fu2015wave,girardin2017non, lam2018traveling,zhang2019propagation}).

However, establishing the nonexistence of traveling waves 
and the existence of a critical wave presents significant challenges due to the lack of monotonicity.
Under the random-diffusion setting, an important technique called the rescaling method \cite{berestycki2005analysis,ducrot2012qualitative,girardin2017non,lam2018traveling}, which is based on Harnack's inequalities, has been widely used.
For our coupled nonlocal diffusion system \eqref{epidemicmodel}, our newly established system-wide $\ell^1$-norm Harnack-type inequality enables this rescaling argument, allowing us to prove the nonexistence of traveling waves.
Furthermore, this estimate helps establish the persistence-extinction dichotomy (see Proposition \ref{lma:Persistence-Extinction}), which is crucial for proving the existence of a traveling wave at the critical speed.

Beyond existence, characterizing the asymptotic behavior at the leading edge is crucial for understanding the qualitative profile of these waves.
Typically, the decay rate at the leading edge of scalar nonlocal diffusion KPP equations can be obtained using the method of bilateral Laplace transforms, developed by Diekmann and Kaper 
\cite{diekmann1978bounded}, and Carr and Chmaj \cite{carr2004uniqueness} (see also \cite{coville2008nonlocal,guo2012traveling,girardin2018non}).
However, this approach faces significant difficulties for coupled systems, since the Laplace transform becomes vector-valued, and its analytic continuation and singularity analysis are considerably more involved. 
Thus, precise edge-asymptotic results for nonlocal coupled systems remain limited in the literature. 

Available rigorous asymptotic analyses on coupled systems in the random-diffusion setting were established in \cite{girardin2018non}, which is based on Ikehara's theorem and ODE arguments using componentwise Harnack's inequality (see \cite{arapostathis1999harnack}).
In our nonlocal diffusion model \eqref{epidemicmodel},  we first utilize our $\ell^1$-norm Harnack-type inequality to provide a concise proof of the 
exponential boundedness of waves.
Building on this foundation, 
we integrate Ikehara’s theorem with Riesz projections to perform the matrix-valued singularity analysis.
This approach (i) converts the singularity analysis into a scalar residue computation, and (ii) 
overcomes the inherent limitation of nonlocal operators, namely the inability to employ a componentwise Harnack's inequality.

The system-wide Harnack-type estimate and the associated rescaling method used here replies on the weakly coupled network structure rather than the particular form of the epidemic nonlinearity. Therefore, 
we 
may expect
that the analytical arguments developed in this work 
can be applied to a broader class of non-cooperative systems with 
network-coupled structures, 
including
nonlocal-diffusion counterparts of the random diffusion models investigated in \cite{girardin2017non, girardin2018non}.

\subsection{Notations}
\begin{enumerate}[1.]
    \item $\hat{\imath}:=\sqrt{-1},$ the imaginary unit (avoiding confusion with indices).
    \item $[n]:=\{1,2,\dots,n\},$ where $n\in\mathbb{N}.$
    \item $0_n:=$ the $n$-dimensional zero vector.
    \item $1_n:=$ the $n$-dimensional vector of all $1$'s.
    \item $I_n:=$ the $n\times n$ identity matrix.
    \item $P^\top:=$ the transpose of a matrix $P$.
    \item $P_{ij}:=$ the $(i,j)$-entry of a matrix $P$.
    \item For a square matrix $P$,
    $\Lambda_1(P):=\max\{ {\rm Re}~\lambda :\, \lambda\mbox{ is an eigenvalue of }P \}$ (the spectral bound).
   \item $v_i:=$ the $i$-th component of the vector $v=(v_1,\dots,v_n).$
    \item $|v|_1:=\sum_{i=1}^n|v_i|$, the $\ell^1$-norm of the vector $v=(v_1,\dots,v_n).$
    \item The left (right, resp.) positive Perron eigenvector of an essentially nonnegative irreducible matrix $P$ is the left (right, resp.) eigenvector of $P$ corresponding to the simple eigenvalue $\Lambda_1(P)$ whose $\ell^1$-norm is $1.$
    \item For two (row or column) vectors $v=(v_1,\dots,v_n)$ and $w=(w_1,\dots,w_n)$ in $\mathbb{R}^n,$ define $v\leq w$ ($v<w$, resp.) if $v_i\leq w_i$ ($v_i<w_i$, resp.) for each $i\in[n];$ $v$ is nonnegative (positive, resp.) if $v_i\geq 0$ ($v_i>0,$ resp.) for each $i\in[n].$
    \item A (row or column) vector function $v(s)=(v_1(s),\dots,v_n(s))$ is nonnegative (positive, resp.) on an interval $I$ if $v_i\geq 0$ ($v_i>0$, resp.) on $I$ for each $i\in[n].$
\item For $u\in\mathbb{R},v=(v_1,\dots,v_n)\in\mathbb{R}^n$ and functions $g_i(u,v),~i\in[n]$, denote
$g=(g_1,\dots,g_n)$ and its Jacobian matrix with respect to $v$ as $D_vg(u,v):=\begin{bmatrix}
         \frac{\partial g_i}{\partial v_j}(u,v)
     \end{bmatrix}.$
\end{enumerate}

\subsection{Existence and asymptotic behaviors of traveling waves}\label{sec:mainresult}
We formulate the traveling wave solutions of \eqref{epidemicmodel} as follows.

\begin{definition}
    A traveling wave solution of \eqref{epidemicmodel} is a bounded positive solution of the form
\begin{align*}
    (u,v)(x,t)=(U,V)(s)\mbox{ with }s=x+ct,
\end{align*}
where $c\in\mathbb{R}$ is the wave speed, satisfying the system
\begin{align}\label{epidemictws}
\begin{cases}
    cU'(s)=d_0\dint_{\mathbb{R}}J(s-s')(U(s')-U(s))ds'+f(U(s))-g_0(U(s),V(s)),\quad s\in\mathbb{R}, \\[2ex]
    cV_i'(s)=d_i\dint_{\mathbb{R}}J(s-s')(V_i(s')-V_i(s))ds'+g_i(U(s),V(s)),\quad s\in\mathbb{R},\quad i\in[n],
\end{cases}
\end{align}
together with the boundary conditions
\begin{align*}
    (U,V)(-\infty)=E_0,\quad\liminf\limits_{s\to\infty}U(s)>0,\quad\liminf\limits_{s\to\infty}V_i(s)>0,\quad i\in[n],
\end{align*}
where $V=(V_1,\dots,V_n)$, and $E_0=(K,0_n)$ is the disease-free equilibrium.
\end{definition}

To construct traveling waves, we further assume that the kernel is thin-tailed and symmetric:
\begin{enumerate}
\item[(J1)]$\int_{\mathbb{R}}J(s)e^{\lambda |s|}ds<\infty\mbox{ for all }\lambda\in[0,\infty)$; 
 \item[(J2)]$J(s)=J(-s)$ for all $s\in\mathbb{R}.$
\end{enumerate}
On the other hand, in some parts, we shall assume that the incidence terms satisfy
\begin{enumerate}
\item[(G1)] for each $i,j\in[n]$, $\sup\limits_{\xi\in[0,\infty)}g_{ij}^0(\xi)$ is finite;
    \item[(G2)] there is a pair $i_0,j_0\in[n]$ such that $(g_{i_0j_0}^0)'(0)>0$.
\end{enumerate}
The assumption the assumption (G1) means that if transmission is present, the incidence rate incorporates a saturation effect, while (G2) ensures that the disease is actively spreading through at least one host group.

Hereafter, we denote $D=\mathrm{diag}(d_1,\dots,d_n)~(d_1,\dots,d_n>0)$. For $\lambda,c\in\mathbb{R},$ denote
\begin{align}\label{defofmlambda}
    m(\lambda)=\int_{\mathbb{R}}J(s)e^{-\lambda s}ds-1,\quad H_{\lambda,c}=m(\lambda)D-c\lambda I_n.
\end{align}
Once (J0) and (J1) hold, we have $m(\lambda)$ is finite for all $\lambda\in\mathbb{R}$ and thus we can define 
\begin{align}\label{thm1c*}
            c^*:=\inf_{\lambda\in(0,\infty)}\frac{\Lambda_1\big(m(\lambda)D+D_vg(E_0)\big)}{\lambda},
        \end{align}
        where $E_0=(K,0_n)$ and $D_vg(E_0)=D_vg(K,0_n)=\begin{bmatrix}
            K(g_{ij}^0)'(0)+m_{ij}
        \end{bmatrix}$.
As we will see  (Lemma \ref{lmalambda1}), $c^*$ is a positive real number provided  $\Lambda_1(D_vg(E_0))>0.$

We first present the existence and nonexistence results for traveling waves of \eqref{epidemicmodel} as follows.

\begin{theorem}\label{thm1}
Let the assumptions {\rm (C1)--(C4)} hold.
    Assume that $J$ satisfies {\rm (J0)--(J2)}, and that $D_vg(E_0)$ is irreducible. 
    \begin{enumerate}[(a)]
    \item If $\Lambda_1(D_vg(E_0))<0,$ then \eqref{epidemicmodel} admits no traveling wave solution for any wave speed $c\in\mathbb{R}$.
    \item If $\Lambda_1(D_vg(E_0))>0$, then \eqref{epidemicmodel} admits no traveling wave solution for any wave speed $c<c^*$.
    \item If $\Lambda_1(D_vg(E_0))>0$, {\rm(G1)} holds, and $\Lambda_1(\mathbb{M})<0$, then \eqref{epidemicmodel} admits a traveling wave solution for any wave speed $c\geq c^*$, where $\mathbb{M}=\begin{bmatrix}
        m_{ij}
    \end{bmatrix}$.
    \end{enumerate}
\end{theorem}

This $c^*$ is referred to as the minimal wave speed for \eqref{epidemicmodel}. 
In part (c), 
the assumption $\Lambda_1(\mathbb{M})<0$
means that the disease will go extinct in the absence of susceptibles ($u\equiv 0$) \cite{lam2018traveling}. 
This assumption, together with (G1), avoids unbounded growth in infected populations, even in the presence of possible long-distance dispersal (i.e., $J$ needs not be compactly supported).
Furthermore, the assumptions $\Lambda_1(D_vg(E_0))>0$ and $\Lambda_1(\mathbb{M})<0$ together imply (G2).
Under {\rm (C1)}, {\rm (C2)}, and
{\rm (C4)}, failure of {\rm (G1)} may imply
$g_{ij}^0\equiv0$ for all $i,j\in[n]$. In that case, no new infections are generated, and a propagation wave of the disease cannot be expected.

Next, we investigate the asymptotic behavior of $(U,V)(s)$ as $s\to-\infty$. 
As we will see (Lemma \ref{lmalambda1}), $\Lambda_1(H_{\lambda,c}+D_vg(E_0))=0$ has a smallest positive root $\lambda=\lambda_c$ whenever $c\geq c^*$. 
On the other hand, it is easy to see that
\begin{align*}
\Delta_c(\lambda):=c\lambda-d_0m(\lambda)+\delta=0
\end{align*}
has a unique positive simple root $\lambda=\nu_c$ for any $c\in\mathbb{R}$, where $\delta>0$ is defined in \eqref{f}.

\begin{theorem}\label{thm2}
Let the assumptions {\rm (C1)--(C4)} and {\rm (G2)} hold.
    Assume that  $J$ satisfies {\rm (J0)--(J2)}, and that
    $D_vg(E_0)$ is irreducible with $\Lambda_1(D_vg(E_0))>0$. 
    Let $c\geq c^*$ and set
    \begin{align*}
    p_c:=\begin{cases}
        0&\quad\mbox{ if }c>c^*;\\
        1&\quad\mbox{ if }c=c^*,
    \end{cases}\quad\mbox{and}\quad
        \mu_c:=\begin{cases}
            p_c&\quad \mbox{ if }\nu_c>\lambda_c;\\
            p_c+1&\quad \mbox{ if }\nu_c=\lambda_c;\\
            0&\quad \mbox{ if }\nu_c<\lambda_c.
        \end{cases}
    \end{align*}
   Suppose that  $(U,V)$ is a bounded 
   positive solution of \eqref{epidemictws} on $\mathbb{R}$
   satisfying $(U,V)(-\infty)=E_0.$ 
   Then the following asymptotic behaviors of $(U,V)$ at $-\infty$ hold.
   \begin{enumerate}[(a)]
       \item There exists a constant $A_c>0$ such that 
   \begin{align*}
       \lim_{s\to-\infty}\frac{V(s)}{|s|^{p_c}e^{\lambda_c s}}=A_c\psi_c,
   \end{align*}
   where 
   $\psi_c$ is the right positive Perron eigenvector of $H_{\lambda_c,c}+D_vg(E_0)$ associated with zero eigenvalue.
   \item There exists a constant 
   $B_c>0$ such that 
   \begin{align*}
       \lim_{s\to-\infty}\frac{K-U(s)}{|s|^{\mu_c}e^{\min\{\lambda_c,\nu_c\}s}}=B_c.
   \end{align*}
   \end{enumerate}
\end{theorem}

\medskip

The rest of this paper is organized as follows.  
In Section \ref{Preparatory lemmas}, we establish linear theory for coupled nonlocal diffusion systems, including maximum principles for cooperative systems and $\ell^1$-norm Harnack-type estimates. In particular, we prove Theorem \ref{generalizedHarnack}.
As an application, we prove Theorem \ref{thm1} in Section \ref{sec:existence and non}.
 The convergence rate at the edge is studied in Section \ref{sectionconvergencerate}, and hence Theorem \ref{thm2} is proved. 

\section{Linear theory and system-wide Harnack estimates}\label{Preparatory lemmas}

In this section, we establish four important results in linear nonlocal diffusion systems, namely the maximum principles for cooperative systems, the $\ell^1$-norm Harnack-type inequality, the existence of exponential solutions, and the characterization of the minimal wave speed. The results in this section are independent of other sections.

\subsection{Maximum principles for cooperative systems}
In this subsection, fix $N\in\mathbb{N}$. We assume that $J:\mathbb{R}^N\to\mathbb{R}$ is a nonnegative Lebesgue measurable function such that
$$\int_{\mathbb{R}^N}J(x)dx=1.$$
Here we do not require (J0)--(J2).
Let $\Omega$ be a domain in $\mathbb{R}^N$, possibly bounded or unbounded.
Fix a $T>0$ and denote $$\Omega_T:=\Omega\times(0,T]\mbox{ and }\overline{\Omega_T}:=\overline{\Omega}\times[0,T],$$
where $\overline{\Omega}$ is the closure of $\Omega$ in $\mathbb{R}^N.$
    Assume that 
   \begin{itemize}
        \item[(M)] $d_i,c_{ij}\in\mathcal{C}(\overline{\Omega_T})\cap L^\infty(\overline{\Omega_T})$ 
        with $d_i\geq0\mbox{ and }c_{ij}\geq0$ for $i\neq j$.
    \end{itemize}

\begin{proposition}[Weak Maximum Principle]\label{maxprinciple}
    Assume {\rm(M)}. Suppose that  for each $i\in[n],$
    $u_i\in \mathcal{C}([0,T];L^\infty(\Omega))$, 
      $u_i(\cdot,t)\in \mathcal{C}(\Omega)$ for each $t\in[0,T]$
      and $u_i(x,\cdot)\in \mathcal{C}^1((0,T])$ for each $x\in\Omega$, 
      and satisfies
    \begin{align*}
         u_i(x,0)&\geq 0\mbox{ for } x\in\Omega;\notag\\
         \frac{\partial u_i}{\partial t}(x,t)&\geq d_i(x,t)\int_{\Omega}J(x-y)u_i(y,t)dy-d_i(x,t)u_i(x,t)+\sum_{j=1}^nc_{ij}(x,t)u_j(x,t)\mbox{ for }(x,t)\in\Omega_T.
    \end{align*}
    Then $u_i\geq 0$ on $\Omega_T$ for each $i\in[n].$
\end{proposition}
\begin{proof}
First of all, since $J\in L^1(\mathbb{R}^N)$ and $u_i(\cdot,t)\in L^\infty(\Omega)$ for each $i\in[n]$ and each $t\in[0,T],$ the integral $\int_{\Omega}J(x-y)u_i(y,t)dy$ is well-defined for each $t\in[0,T].$

Let 
\begin{align}\label{K0}
K_0:=\max\left\{\sum_{j=1}^n\sup_{\overline{\Omega_T}}|c_{ij}|:i\in[n]\right\}+1.
\end{align}
    Also,  
    for a fixed $\varepsilon>0,$ let
    $$v_i(x,t)=u_i(x,t)+\varepsilon e^{K_0t}\mbox{ for }i\in[n]\mbox{ and }(x,t)\in\overline{\Omega_T}.$$
     Then for each $i\in[n]$ and $(x,t)\in\Omega_T,$ we have
    \begin{align}\label{ineqofv}
        &d_i(x,t)\int_{\Omega}J(x-y)v_i(y,t)dy-d_i(x,t)v_i(x,t)+\sum_{j=1}^nc_{ij}(x,t)v_j(x,t)\notag\\
        =&d_i(x,t)\int_{\Omega}J(x-y)u_i(y,t)dy-d_i(x,t)u_i(x,t)+\sum_{j=1}^nc_{ij}(x,t)u_j(x,t)\notag\\
        &+d_i(x,t)\int_{\Omega}J(x-y)\varepsilon e^{K_0t}dy-d_i(x,t)\varepsilon e^{K_0t}+\sum_{j=1}^nc_{ij}(x,t)\varepsilon e^{K_0t}\notag\\
        <&\frac{\partial u_i}{\partial t}(x,t)+K_0\varepsilon e^{K_0t}=\frac{\partial v_i}{\partial t}(x,t),
    \end{align}
    where the inequality holds since
    we used $\int_{\Omega}J(x-y)dy\leq\int_{\mathbb{R}^N}J(x-y)dy=\int_{\mathbb{R}^N}J(y)dy=1$ and \eqref{K0}. 
    Also, 
    \begin{align}\label{vkgeqeps}
        v_i(x,0)\geq \varepsilon>0\mbox{ for all }i\in[n]\mbox{ and }x\in\Omega.
    \end{align}
    To reach a contradiction, we suppose that $$\inf\{v_i(x,t):i\in[n],(x,t)\in\Omega_T\}<0.$$ 
    Define
    \[
m(t):=\min_{i\in[n]}\inf_{x\in\Omega}v_i(x,t),\quad t\in[0,T].
\]
Since $v_i\in \mathcal{C}([0,T];L^\infty(\Omega))$ for each $i\in[n]$, $m$ is continuous on $[0,T]$. Moreover, by \eqref{vkgeqeps},
$m(0)\geq\varepsilon>0$. Hence, there exist $t_0\in(0,T]$ and $i_0\in[n]$ such that
$v_i(x,t)\geq0$ for all $i\in[n]$, $x\in\Omega$ and $t\in[0,t_0]$,
and 
\begin{align}\label{inf>0}
   \inf_{x\in\Omega}v_{i_0}(x,t_0)=0.
\end{align}

    Thanks to (M), \eqref{ineqofv} gives
    \begin{align*}
        \frac{\partial v_i}{\partial t}(x,t)+(d_i(x,t)-c_{ii}(x,t))v_i(x,t)>0\mbox{ for all }t\in(0,t_0],\, x\in\Omega, \mbox{ and }i\in[n].
    \end{align*}
    This implies that the function 
    $$t \mapsto v_i(x,t)\exp\left(\int_0^t(d_i(x,\tau)-c_{ii}(x,\tau))d\tau\right)$$ is  
    strictly increasing
    in $t\in[0,t_0]$ for each $i\in[n].$ 
    Then, using $v_i(x,0)\geq \varepsilon$,
    \begin{align*}
    v_i(x,t_0)>v_i(x,0)\exp\left(-\int_0^{t_0}(d_i(x,\tau)-c_{ii}(x,\tau))d\tau\right)\geq \varepsilon e^{-C_0T}>0\quad \text{ for all } i\in[n]\mbox{ and all }x\in\Omega,
    \end{align*}
    where $C_0:=\max_{i\in[n]}\|d_i-c_{ii}\|_{L^{\infty}(\Omega_T)}$.
    This contradicts \eqref{inf>0}.
    Therefore, $v_i\geq0$ on $\Omega_T$ for each $i\in[n]$, and thus
    $$u_i(x,t)\geq-\varepsilon e^{K_0t}\mbox{ for each }(x,t)\in\Omega_T\mbox{ and each }i\in[n].$$
    Since $\varepsilon>0$ is arbitrary, we can conclude the proof.
\end{proof}

We next establish a strong maximum principle. The following lemma will be used in our proof.

\begin{lemma}\label{graphlemma}
 Let $P=\begin{bmatrix}
     P_{ij}
 \end{bmatrix}$ be an essentially nonnegative irreducible matrix and $w=(w_1,\dots,w_n)^\top$ be a nonnegative vector with the property
 \begin{align}\label{pro1}
     \mbox{whenever $w_i=0$ for some $i\in[n]$, we have $\sum_{j=1}^nP_{ij}w_j=0$}.
 \end{align}
If $w_i=0$ for some $i\in[n]$, then $w$ is a zero vector.
\end{lemma}
\begin{proof}
    Since $P$ is irreducible, 
    the associated directed graph is strongly connected. Therefore,
    for any $i\in[n],$ every vertex $\ell\in[n]$ can be reached from $i$ by a sequence of edges
    $$i\longrightarrow i_1\longrightarrow i_2\longrightarrow\cdots\longrightarrow i_k\longrightarrow\ell,$$
    which means that the corresponding entries $P_{ii_1},P_{i_1i_2},\dots,P_{i_k\ell}$ are all positive.
    If $w_i=0,$ then by \eqref{pro1},
    \[
    \sum_{j\in[n]\setminus\{i\}} P_{ij} w_j =\sum_{j=1}^n P_{ij} w_j =0,
    \]
    which implies that $w_i=w_{i_1}=\cdots=w_{i_k}=w_{\ell}=0.$ Since $\ell$ is arbitrary, we conclude that $w=0.$
\end{proof}

\begin{proposition}[Strong Maximum Principle]\label{SMP}
    Assume {\rm{(M)}}. Suppose that for each $i\in[n],$ $u_i\geq 0$ on $\Omega\times[0,T]$,
     $u_i(\cdot,t)\in \mathcal{C}(\Omega)$ for each $t\in[0,T]$, 
$u_i(x,\cdot)\in \mathcal{C}([0,T])\cap \mathcal{C}^1((0,T])$ for each $x\in\Omega$,
     and satisfies the inequality
    \begin{align*}
         \frac{\partial u_i}{\partial t}(x,t)&\geq d_i(x,t)\int_{\Omega}J(x-y)u_i(y,t)dy-d_i(x,t)u_i(x,t)+\sum_{j=1}^nc_{ij}(x,t)u_j(x,t)\mbox{ for }(x,t)\in\Omega_T.
    \end{align*}
    Furthermore, suppose that $J$ is continuous at $0$ and $J(0)>0.$
    If there is an $(x_0,t_0)\in\Omega_T$ such that 
    \begin{enumerate}[(i)]
        \item $d_{i}(x,t_0)>0$ for each $i\in[n]$ and each $x\in\Omega,$
        \item $\begin{bmatrix}
        c_{ij}(x_0,t_0)
    \end{bmatrix}$ is irreducible, and
    \item $u_{i_0}(x_0,t_0)=0$ for some $i_0\in[n],$
    \end{enumerate}
    then $u_i\equiv 0$ on $\Omega\times[0,t_0]$ for each $i\in[n].$
\end{proposition}
\begin{proof}
First of all, since $J$ and $u_i(\cdot,t)$ are nonnegative on $\Omega$ for each $i\in[n]$ and each $t\in[0,T],$ the integral $\int_{\Omega}J(x-y)u_i(y,t)dy$ is a nonnegative extended real number. Furthermore, the given inequality and (i) together guarantee that $\int_{\Omega}J(x-y)u_i(y,t_0)dy$ is finite for each $i\in[n]$ and each $x\in\Omega$.
    The remainder of the proof is divided into three steps.

   {\bf Step 1}. {\it  $u_i(x_0,t_0)=0$ for each $i\in[n].$}

   Let $k\in[n]$ be any index such that $u_k(x_0,t_0)=0$.
    Since $u_{k}$ is nonnegative and $u_{k}(x_0,t_0)=0,$ it must be a minimum and so
    \begin{align*}
        0\ge\frac{\partial u_{k}}{\partial t}(x_0,t_0)&\geq d_{k}(x_0,t_0)\int_{\Omega}J(x_0-y)u_{k}(y,t_0)dy-d_{k}(x_0,t_0)u_{k}(x_0,t_0)+\sum_{j=1}^nc_{kj}(x_0,t_0)u_j(x_0,t_0)\\
        &\geq\sum_{j=1}^nc_{kj}(x_0,t_0)u_j(x_0,t_0)=\sum_{j\in[n]\setminus\{k\}}c_{kj}(x_0,t_0)u_j(x_0,t_0)\geq 0,
    \end{align*}
    which implies $\sum_{j=1}^nc_{kj}(x_0,t_0)u_j(x_0,t_0)=0$. 
    Therefore, $w=(w_1,\ldots,w_n)^\top
:=(u_1(x_0,t_0),\ldots,u_n(x_0,t_0))^\top$ satisfies \eqref{pro1} in Lemma~\ref{graphlemma}.
    Since 
    $\begin{bmatrix}
        c_{ij}(x_0,t_0)
    \end{bmatrix}$ is essentially nonnegative and irreducible, and $w_k=0$,
    we have $w=0_n$ by Lemma \ref{graphlemma}. Therefore, Step 1 is complete. 

    {\bf Step 2}. {\it  $u_i(x,t_0)=0$ for any $i\in[n]$ and any $x\in\Omega$.}\\
    Since $J$ is continuous at $0$ and $J(0)>0,$ there is a $\delta>0$ such that $J>0$ on $B_\delta(0).$
    Let
    $$Z_0=\{x\in\Omega:u_i(x,t_0)=0\mbox{ for each }i\in[n]\}.$$
    By Step 1, $Z_0$ is nonempty as $x_0\in Z_0$. Now, let $x_1\in Z_0$ and choose a $\delta_1\in(0,\delta)$ such that $B_{\delta_1}(x_1)\subseteq\Omega$. Then for each $i\in[n],$
    \begin{align*}
        0\ge\frac{\partial u_i}{\partial t}(x_1,t_0)\geq d_{i}(x_1,t_0)\int_{\Omega}J(x_1-y)u_i(y,t_0)dy\geq d_{i}(x_1,t_0)\int_{B_{\delta_1}(x_1)}J(x_1-y)u_i(y,t_0)dy\ge 0.
    \end{align*}
    Then $\int_{B_{\delta_1}(x_1)}J(x_1-y)u_i(y,t_0)dy=0$. Note that $J(x_1-y)>0$ for all $y\in B_{\delta_1}(x_1)$. Then we have $u_{i}(y,t_0)=0$ for all $y\in B_{\delta_1}(x_1),$ since $u_i$ is nonnegative and is continuous in $y$.
     In other words, $B_{\delta_1}(x_1)\subseteq Z_0$. Hence, $Z_0$ is an open set in $\Omega.$ However, $u_i(\cdot,t_0)$ is continuous on $\Omega,$ which means that $Z_0$ is also a closed set in $\Omega.$ As $\Omega$ is a connected set, we conclude that $Z_0=\Omega,$ i.e., the proof of Step 2 is complete.

    {\bf Step 3}. {\it  $u_i\equiv 0$ on $\Omega\times[0,t_0]$ for each $i\in[n].$}\\
    Let 
    $$\mu=\sup\left\{d_i(x,t)+|c_{ii}(x,t)|:i\in[n],(x,t)\in\overline{\Omega_T}\right\}.$$
    Then $$\frac{\partial u_i}{\partial t}(x,t)\geq -\mu u_i(x,t)\mbox{ for each }(x,t)\in\Omega_T\mbox{ and each }i\in[n].$$
    For each $i\in[n],$ since $u_i(x,t_0)=0$ for any $x\in\Omega,$ we see that $u_i(x,t)=0$ for any $x\in\Omega$ and $t\in[0,t_0],$ as desired.
    \end{proof}
\medskip

\subsection{A Harnack-type inequality}
Throughout this subsection, we let $$\mathcal{D}(s)=\mathrm{diag}(d_1(s),\dots,d_n(s)),$$
in which $d_i(\cdot)\in\mathcal{C}(\mathbb{R})$ for each $i=1,\dots,n$. 
Also, we always assume that there exist $d_{\max}\geq d_{\min}>0$ such that $$d_{\min}\leq  d_i(s)\leq d_{\max}\mbox{ for each }i\in[n]\mbox{ and each }s\in\mathbb{R}.$$

In this subsection, we always assume that $J:\mathbb{R}\to\mathbb{R}$ satisfies (J0), while (J1) or (J2) is not required.
Under this assumption, we prove the central mathematical tool of our analysis--the $\ell^1$-norm Harnack-type inequality (Theorem \ref{generalizedHarnack}).

\begin{proof}[Proof of Theorem \ref{generalizedHarnack}]
    Let $\Lambda_1$ be the minimum column sum of $P_1$ and $\Lambda_2$ the maximum column sum of $P_2.$ Then
    $$\Lambda_11_n^\top\leq1_n^\top P_1\mbox{ and }1_n^\top P_2\leq\Lambda_21_n^\top.$$
    Thus, 
    \begin{align}\label{ineq-p}
       \Lambda_11_n^\top p(s)\leq 1_n^\top P_1p(s)&\leq c1_n^\top p'(s)-1_n^\top \mathcal{D}(s)\bigg(\int_{\mathbb{R}}J(s-s')p(s')ds'- p(s)\bigg)\\&\leq 1_n^\top P_2p(s)\leq\Lambda_21_n^\top p(s),\quad s\in\mathbb{R}.\notag
    \end{align}
    Let 
    \begin{align*}
    z(s)=|p(s)|_1,\quad  \varphi(s)=\frac{z'(s)}{z(s)}.
    \end{align*}
   Since $p(s)>0_n$ for each $s\in\mathbb{R}$, we have $z(s)=1_n^\top p(s)$
 and thus $z'(s)=1_n^\top p'(s)$. Also, using
 \begin{align*}
    d_{\min}z(s) \leq 1_n^\top \mathcal{D}(s)p(s) \leq d_{\max}z(s) \quad \text{for all } s \in \mathbb{R},
\end{align*}    
    from \eqref{ineq-p} we have
    \begin{align}\label{key-est}
        (\Lambda_1-d_{\max})+d_{\min}F(s)\leq c\varphi(s)\le (\Lambda_2-d_{\min})+d_{\max}F(s),\quad s\in\mathbb{R},
    \end{align}
    where 
    $F(s):=\int_{\mathbb{R}}J(s')\exp\left(\int_s^{s-s'}\varphi(\tau)d\tau\right)ds'.$

   Next, we introduce
   \begin{align}\label{defofw}
       w(s)=\exp\left(ms+\int_0^s\varphi(\tau)d\tau\right),
   \end{align}  
 where $m=-\frac{\Lambda_1-d_{\max}}{c}.$
To prove part (a), we divide our discussion into two cases: $c>0$ and $c<0.$ 
   We first assume that $c>0.$
   Then 
   \begin{align}\label{w-inc}
       w'(s)=w(s)\Big(\varphi(s)-\frac{\Lambda_1-d_{\max}}{c}\Big)\geq \frac{d_{\min}}{c}w(s)F(s)=\frac{d_{\min}}{c}\int_{\mathbb{R}}J(s')e^{ms'}w(s-s')ds'>0
   \end{align}
   and so $w$ is increasing on $\mathbb{R}$. Since $J$ is continuous at $0$ and $J(0)>0$, there exists $\delta>0$ such that $J(s)>0$ for all $s\in(-3\delta,3\delta)$. Together with \eqref{w-inc}, we have
   \begin{align*}
       w(s)&\geq w(s)-w(s-\delta)=\int_{s-\delta}^s w'(\tau)d\tau
       \geq \frac{d_{\min}}{c}\int_{\mathbb{R}}J(s')e^{ms'}\int_{s-\delta}^sw(\tau-s')d\tau ds'\\
       &\geq \frac{d_{\min}\delta}{c}\int_{-3\delta}^{-2\delta}J(s')e^{ms'}w(s-\delta-s')ds'
       \geq \frac{d_{\min}\delta}{c} w(s+\delta)\int_{-3\delta}^{-2\delta}J(s')e^{ms'}ds'.
   \end{align*}
   It follows that
   $$\frac{w(s+\delta)}{w(s)}\leq \frac{c}{d_{\min}\delta A_\delta}\mbox{ for all } s\in\mathbb{R},$$
   where $A_\delta=\int_{-3\delta}^{-2\delta}J(s')e^{ms'}ds'>0$.
   
   Similarly, since $w$ is increasing, 
   \begin{align*}
       w(s+\delta)&\geq w(s+\delta)-w(s)\geq  \frac{d_{\min}}{c}\int_{\mathbb{R}}J(s')e^{ms'}\int^{s+\delta}_sw(\tau-s')d\tau ds'\\&\geq  \frac{d_{\min}\delta}{c}\int_{\mathbb{R}}J(s')e^{ms'}w(s-s')ds'
   \end{align*}
   and so 
   \begin{align*}
       \frac{w(s+\delta)}{w(s)}\geq  \frac{d_{\min}\delta}{c}\int_{\mathbb{R}}J(s')e^{ms'}\frac{w(s-s')}{w(s)}ds'= \frac{d_{\min}\delta}{c} F(s),
   \end{align*}
   where we used the fact that 
   \begin{align*}
   \frac{w(s-s')}{w(s)} = e^{-ms'}\exp\left(\int_s^{s-s'}\varphi(\tau)d\tau\right).
   \end{align*}
   This implies
   $$F(s)\leq\frac{c}{d_{\min}\delta}\frac{w(s+\delta)}{w(s)}\leq \frac{c^2}{(d_{\min}\delta)^2A_\delta}.$$
   Substituting this upper bound into \eqref{key-est} and noting that $F>0$, we have
   $$\Lambda_1-d_{\max}\leq c\varphi(s)\leq \frac{d_{\max}c^2}{(d_{\min}\delta)^2A_\delta}+(\Lambda_2-d_{\min})<\infty\mbox{ for all }s\in\mathbb{R}.$$
   This completes the proof of part (a) when $c>0$.

   The case $c<0$ can be treated using a symmetric argument.
Following similar steps as in the previous case, and choosing a suitable constant $B_\delta:=\int_{2\delta}^{3\delta}J(s')e^{ms'}ds',$ we can show that
   $$\Lambda_1-d_{\max}\leq c\varphi(s)\leq \frac{d_{\max}c^2}{(d_{\min}\delta)^2B_\delta}+(\Lambda_2-d_{\min})<\infty\mbox{ for all }s\in\mathbb{R}.$$
    This proves that $\varphi$ is bounded on $\mathbb{R}$ and the bounds are independent of $p(s).$ Then part (a) is asserted.

       For part (b), let $y(s)=\ln z(s)$ for $s\in\mathbb{R},$ where $z(s)=|p(s)|_1 =1_n^\top p(s)>0$. Recall the definition of $\varphi(s)$ in the beginning of the proof. 
       Then $y'(s)=\varphi(s)$ and thus
    $|y'(s)|=|\varphi(s)|\leq M$ for all  $s\in\mathbb{R}$, where $M>0$ is chosen in part (a). By the mean value theorem, for each $\xi\in\mathbb{R}$ and $s\in[-r,r]$, we have
    \begin{align*}
    |y(\xi+s)-y(\xi)|\leq M|s|\leq Mr.
    \end{align*}
    It follows that 
$$e^{-Mr} \leq e^{y(\xi+s)-y(\xi)} \leq e^{Mr},\quad \xi\in\mathbb{R},\ s\in[-r,r].$$
Since $e^{y(s)} = z(s) = |p(s)|_1$ for all $s\in\mathbb{R}$, 
we have
\[
e^{-Mr}\le \frac{z(\xi+s)}{z(\xi)}=\frac{1_n^\top p(\xi+s)}{1_n^\top p(\xi)}\le e^{Mr},\quad \xi\in\mathbb{R},\ s\in[-r,r].
\]
    Then part (b) is asserted. This completes the proof.
\end{proof}

Thanks to Theorem \ref{generalizedHarnack}, we immediately obtain the following result.

\begin{corollary}\label{lmalinear}
Assume that $J$ satisfies {\rm{(J0)}}. Let $c$ be a nonzero real constant, and $P$ be any
$n \times n$ real matrix.
Let $p(s)=(p_1(s),p_2(s),\dots,p_n(s))$ be a positive $\mathcal{C}^1$ function satisfying
the linear system
    \begin{align*}
        cp'(s)=\mathcal{D}(s)\bigg(\int_{\mathbb{R}}J(s-s')p(s')ds'-p(s)\bigg)+Pp(s)\mbox{ for }s\in\mathbb{R}.
    \end{align*}
    Then the same conclusions (a) and (b) in Theorem \ref{generalizedHarnack} hold.
\end{corollary}
\begin{proof}
    This corollary follows directly from Theorem \ref{generalizedHarnack} with $P_1=P_2=P.$
\end{proof}
\begin{remark}\label{rmk:uniforminc}
    Although the constant $M$ in Theorem \ref{generalizedHarnack} (and Corollary \ref{lmalinear}) depends on the wave speed $c$, the estimate remains uniform for $c$ within any bounded interval away from zero. This specific form of uniformity is precisely what is required for the rescaling arguments in Lemma \ref{limsuppositive}.
\end{remark}

\medskip

\subsection{Exponential solutions and the characteristic equation}
Throughout the remainder of this section, we assume that $$\mathcal{D}=\mathrm{diag}(d_1,\dots,d_n)$$ is a constant diagonal matrix with positive diagonal entries.

The result in this subsection needs assumptions (J0) and (J1) (where (J1) is defined in Section \ref{sec:mainresult}).

We first establish the strong maximum principle for the linearized system, which is a corollary of Proposition \ref{SMP}.

\begin{lemma}\label{SMPlinearized}
 Assume that $J$ satisfies {\rm{(J0)}}. Let $c\in\mathbb{R}$.
 Also, let $P_{ij}(\cdot)\in\mathcal{C}(\mathbb{R})\cap L^\infty(\mathbb{R})$ for $i,j\in[n]$ and assume that $P(s)=\begin{bmatrix}
     P_{ij}(s)
 \end{bmatrix}$ is an essentially nonnegative irreducible matrix for each $s\in\mathbb{R}$.
    If $p=(p_1,\dots,p_n)$ is a nonnegative $\mathcal{C}^1$
    function 
    satisfying the inequality
    \begin{align*}
    cp'(s)\geq\mathcal{D}\bigg(\int_{\mathbb{R}}J(s-s')p(s')ds'-p(s)\bigg)+P(s)p(s)
       \mbox{ for each }s\in\mathbb{R},
    \end{align*}
    then $p(s)$ is either positive on $\mathbb{R},$ or identically zero on $\mathbb{R}.$
\end{lemma}
\begin{proof}
This is an immediate result from the strong maximum principle (Proposition \ref{SMP}) by setting $u_i(x,t)=p_i(x+ct)$ for $i\in[n],x\in\mathbb{R}$ and $t\geq 0.$
\end{proof}

Due to Corollary \ref{lmalinear}, we have the following result. 
A key step in our proof involves a rescaling argument, where the validity lies in the use of our Harnack-type inequality.

\begin{proposition}\label{linearexp}
 Assume that $J$ satisfies {\rm{(J0)}} and {\rm (J1)}. Let $c$ be a nonzero real constant and let $P$ be an essentially nonnegative irreducible matrix.
 If the linear system
 \begin{align}\label{linearsystem1}
        cp'(s)=\mathcal{D}\bigg(\int_{\mathbb{R}}J(s-s')p(s')ds'-p(s)\bigg)+Pp(s)\mbox{ for }s\in\mathbb{R}.
    \end{align}
has a positive $\mathcal{C}^1$ solution on $\mathbb{R}$,
then there exist $\lambda\in\R$ and a positive vector $\psi\in\R^n$ such that $p^{\infty}(s):=e^{\lambda s}\psi$ is a positive exponential solution of \eqref{linearsystem1}  on $\mathbb{R}$,
where $\lambda$ satisfies the characteristic equation
$$\Lambda_1(\mathcal{H}_{\lambda,c}+P)=0,$$
and $\psi$ is the corresponding right positive Perron eigenvector, and
$$\mathcal{H}_{\lambda,c}:=\left(\int_{\mathbb{R}}J(s)e^{-\lambda s}ds-1\right)\mathcal{D}-c\lambda I_n.$$
\end{proposition}
\begin{proof}
Let $p=(p_1,\dots,p_n)$ be a positive solution of \eqref{linearsystem1}. Then
the function
    $q=(q_1,\dots,q_n)$ defined by $$q_i(s):=\dfrac{p_i'(s)}{p_i(s)}\quad \mbox{ for }i\in[n],$$ is well-defined and is continuous on $\mathbb{R}$. We first assume that $c>0.$ The proof is cut into three steps.

    {\bf Step 1}. {\it  The quantity
    $$\lambda:=\min_{i\in[n]}\inf_{s\in\mathbb{R}}q_i(s)$$
    is finite.}

    Clearly, $\lambda<\infty$.
    Since $P$ is essentially nonnegative,
    we see that for each $i\in[n]$, 
    \begin{align*}
    cp_i'(s)=d_i\int_{\R}J(s-s')p_i(s')ds'-d_ip_i(s)+\sum_{j=1}^nP_{ij}p_j(s)\geq -d_i p_i(s)+P_{ii} p_i(s),\quad s\in\R,  
    \end{align*}
    which together with $c>0$ implies that
    $p_i'\geq-\mu p_i$ on $\mathbb{R}$,
    where $\mu=\max\limits_{i\in[n]}\left(\frac{d_i+|P_{ii}|}{c}\right)$. This implies $q_i\geq -\mu$ on $\mathbb{R}$ for each $i\in[n]$,  
    and thus $\lambda>-\infty$. This completes Step 1.
   
   {\bf Step 2}. {\it  The linear system \eqref{linearsystem1} has a positive solution $p^\infty(s)$ on $\mathbb{R}$ satisfying that 
   $$(p^\infty)'(s)\geq\lambda p^\infty(s)\mbox{ for all }s\in\mathbb{R}.$$}
   
    By Step 1, there exist $j\in[n]$ and a real sequence $\{s_k\}$ such that 
    \begin{align}\label{lambdasequence}
        \lambda=\lim\limits_{k\to\infty}q_j(s_k).
    \end{align}
    Consider the sequence 
    $$\widehat{p}^{(k)}(s):=\dfrac{p(s+s_k)}{|p(s_k)|_1},\quad s\in\mathbb{R}.$$
    Note that each $\widehat{p}^{(k)}$ solves \eqref{linearsystem1} and $|\widehat p^{(k)}(0)|_1=1$. Then,
    by Corollary \ref{lmalinear} (b), 
    there exists $M>0$, independent of $k$ such that
    \begin{align}\label{p-bound}
|\widehat p^{(k)}(s)|_1\leq e^{M|s|}
\quad
\text{for all } s\in\mathbb{R} \text{ and }k\in\mathbb{N}.
\end{align}
Set $M_0:=\int_{\mathbb{R}} J(\eta) e^{M|\eta|}d\eta$. By (J1), we have 
$M_0<\infty$. Thanks to \eqref{p-bound}, we obtain
\begin{align*}
\left|
\int_{\mathbb{R}}J(\eta)\widehat p^{(k)}(s-\eta)\,d\eta
\right|_1
\leq
\int_{\mathbb{R}}J(\eta)e^{M|s-\eta|}\,d\eta
\leq
M_0e^{M|s|}.
\end{align*}
Also, the equation satisfied by $\widehat p^{(k)}$ yields that 
\begin{align}\label{p'-bound}
|(\widehat p^{(k)})'(s)|_1
\leq C_1e^{M|s|}\quad
\text{for all } s\in\mathbb{R} \text{ and }k\in\mathbb{N}.
\end{align}
for some $C_1>0$, independent of $k$.
Combining \eqref{p-bound} and \eqref{p'-bound}, both $\{\widehat{p}^{(k)}\}$ and $\{(\widehat{p}^{(k)})'\}$ are
    locally uniformly bounded on $\mathbb{R}$.
    Consequently $\{\widehat{p}^{(k)}\}$ is equicontinuous on every compact interval.

    We now show the equicontinuity of  $\{(\widehat{p}^{(k)})'\}$ on every compact interval. Fix $L>0$. For $s_1,s_2\in [-L,L]$, from \eqref{p'-bound} we have
    \begin{align*}
       &\Big|\int_{\mathbb{R}} J(\eta) \Big(\widehat p^{(k)}(s_1-\eta)-\widehat p^{(k)}(s_2-\eta)\Big)d\eta \Big|_1\leq \int_{\mathbb{R}} J(\eta)\Big|\int_{s_1}^{s_2}(\widehat p^{(k)})'(\xi-\eta)d\xi\Big|_1 d\eta   \\
       &\leq C_1 e^{ML}|s_1-s_2|\int_{\mathbb{R}}J(\eta)e^{M|\eta|}d\eta\leq C_1M_0e^{ML}|s_1-s_2|.
    \end{align*}
    Together with the equation satisfied by $\widehat p^{(k)}$, 
    we obtain the equicontinuity of  $\{(\widehat{p}^{(k)})'\}$.
    
    By the Arzel\'{a}-Ascoli theorem,
    $\{\widehat{p}^{(k)}\}$ converges, up to a diagonal extraction, to a nonnegative function $p^\infty$ in $\mathcal{C}^1_{\mathrm{loc}}(\mathbb{R})$ with $|p^\infty(0)|_1=1$.
    We still denote this subsequence by $\{\widehat{p}^{(k)}\}$.

    We now verify that this $p^\infty$ is a positive solution of \eqref{linearsystem1} on $\mathbb{R}$.
    By \eqref{p-bound}, we have
    \begin{align*}
        |J(s')\widehat{p}_i^{(k)}(s-s')|\leq J(s')e^{M|s'|}\cdot e^{M|s|}\quad\mbox{for all }i\in[n]\mbox{ and all }s,s'\in\mathbb{R}.
    \end{align*}
    In view of (J1), the Lebesgue dominated convergence theorem deduces that
    \begin{align*}
       \lim\limits_{k\to\infty} \int_{\mathbb{R}}J(s')\widehat{p}^{(k)}(s-s')ds'=\int_{\mathbb{R}}J(s')p^\infty(s-s')ds'.
    \end{align*}
    Thus, $p^\infty$ is a solution of \eqref{linearsystem1} on $\mathbb{R}.$
  Note that $|p^\infty(0)|_1>0$ and $P$ is an essentially nonnegative irreducible matrix, we can apply
  Lemma \ref{SMPlinearized} to conclude that $p^\infty$ is positive on $\mathbb{R}$.

     By linearity,  $p^{\infty}$ is smooth and its derivatives of all orders solve \eqref{linearsystem1}. 
     Put
    $$w^{(k)}:=(\widehat{p}^{(k)})'-\lambda \widehat{p}^{(k)},\quad w^\infty:=(p^\infty)'-\lambda p^\infty.$$
    Then $w^\infty$ is a solution of \eqref{linearsystem1} on $\mathbb{R}$. 
    Furthermore, 
    \begin{align}\label{expandwinfty}
        w^{(k)}(s)=\widehat{p}^{(k)}(s)\circ(q(s+s_k)-\lambda 1_n)\mbox{ for }s\in\mathbb{R}\mbox{ and }k\in\mathbb{N},
    \end{align} 
    where $\circ$ denotes the Hadamard product.
    
    Now, fix an $s\in\mathbb{R}.$
    By the definition of $\lambda$, we have
    \begin{align*}
        \liminf_{k\to\infty}q_i(s+s_k)\geq\lambda \mbox{ for all } i\in[n].
    \end{align*}
    Therefore, 
    for any $\varepsilon>0,$ there is a $K=K(s,\varepsilon)\in\mathbb{N}$ such that 
    \begin{align*}
        q_i(s+s_k)\geq\lambda-\varepsilon\mbox{ for any }i\in[n]\mbox{ and any }k\geq K
    \end{align*}
    and thus,
    \begin{align}\label{wgeqp}
        w^{(k)}_i(s)\geq-\varepsilon\widehat{p}^{(k)}_i(s)\geq-\varepsilon\sup_{k\in\mathbb{N}}\widehat{p}^{(k)}_i(s).
    \end{align}
    Since $\sup\limits_{k\in\mathbb{N}}\widehat{p}^{(k)}_i(s)<\infty,$ letting $k\to\infty$ and then $\varepsilon\to0^+$ in \eqref{wgeqp} yields that $w^\infty$ is nonnegative. This proves Step 2.

    {\bf Step 3}. {\it  The solution $p^\infty(s)$ of \eqref{linearsystem1} in Step 2 satisfies the requirement.}

    By \eqref{lambdasequence} and \eqref{expandwinfty},
    $$w_j^\infty(0)=p_j^\infty(0)\left(\lim\limits_{k\to\infty}q_j(s_k)-\lambda\right)=0.$$
    As $w^\infty$ is a nonnegative solution of \eqref{linearsystem1} on $\mathbb{R}$ and $P$ is an essentially nonnegative irreducible matrix,
    Lemma \ref{SMPlinearized} shows that $w^\infty\equiv 0_n$ on $\mathbb{R}$. Thus,
    $$(p^\infty)'=\lambda p^\infty\mbox{ on }\mathbb{R}.$$
    Since $p^\infty$ is positive on $\mathbb{R}$, we see that there is a positive vector $\psi\in\mathbb{R}^n$ such that $$p^\infty(s)=e^{\lambda s}\psi\mbox{ for all }s\in\mathbb{R}.$$
    Since $|p^\infty(0)|_1=1$, we have $|\psi|_1=1.$ Furthermore,
    \begin{align*}
        c\lambda e^{\lambda s}\psi=\mathcal{D}\left(\int_{\mathbb{R}} J(s')e^{\lambda (s-s')}\psi ds'- e^{\lambda s}\psi\right)+P(e^{\lambda s}\psi)\mbox{ for all }s\in\mathbb{R}.
    \end{align*}
    Canceling $e^{\lambda s},$ we have
    $(\mathcal{H}_{\lambda,c}+P)\psi=0_n$,
    which means that $\mathcal{H}_{\lambda,c}+P$ has zero eigenvalue with a positive eigenvector $\psi$.
   Since $\mathcal H_{\lambda,c}+P$ is essentially nonnegative and irreducible, and it has a positive eigenvector $\psi$ associated with the eigenvalue $0$, the Perron--Frobenius theorem implies
that $\Lambda_1(\mathcal H_{\lambda,c}+P)=0$. Hence, Proposition~\ref{linearexp} holds for $c>0$.

    The case $c<0$ can be done similarly by showing that 
    $\max\limits_{i\in[n]}\sup\limits_{s\in\mathbb{R}}q_i(s)\in\mathbb{R}.$ The proof is now complete.
\end{proof}

\begin{remark}
    A counterpart to Proposition~\ref{linearexp} in the random-diffusion setting was established in \cite[Lemma 6.1]{girardin2017non}, which relies on a strong componentwise Harnack-type inequality \cite{arapostathis1999harnack}. 
    In the nonlocal-diffusion setting, however, such strong estimates are generally more difficult to obtain. In our analysis, we demonstrate that an $\ell^1$-norm Harnack-type inequality is sufficient to achieve the desired exponential bound.  
    This observation simplifies the technical requirements on the nonlocal kernels.
\end{remark}
\medskip

\subsection{Characterization of the minimal wave speed}
The results in this subsection require assumptions (J0)--(J2) (where (J1) and (J2) are defined in Section \ref{sec:mainresult}).

\begin{lemma}\label{lmalambda1}
Assume that $J$ satisfies {\rm (J0)--(J2)}. Let $P$ be an $n\times n$ essentially nonnegative matrix.
For each $c\in\mathbb{R},$ denote
\begin{align}\label{defHlambdaclinear}
    m(\lambda)=\int_{\mathbb{R}}J(s)e^{-\lambda s}ds-1,\quad \mathcal{H}_{\lambda,c}=m(\lambda)\mathcal{D}-c\lambda I_n,
\end{align}
and let
$$\Lambda(c)=\{\lambda\in\mathbb{R}:\Lambda_1(\mathcal{H}_{\lambda,c}+P)=0\}.$$
Then the following hold.
\begin{enumerate}[(a)]
    \item If $\Lambda_1(P)<0,$ then for any $c\in\mathbb{R}$, $\Lambda(c)=\{\underline{\lambda},\overline{\lambda}\}$ for some $\underline{\lambda}<0<\overline{\lambda}.$
    \item Suppose that $\Lambda_1(P)>0$ and define
    \begin{align}\label{generalc*}
        c_0:=\inf_{\lambda\in(0,\infty)}\frac{\Lambda_1(m(\lambda)\mathcal{D}+P)}{\lambda}.
    \end{align}
    Then $c_0>0$ and  $\Lambda(c)$ is exactly
    \begin{enumerate}[(i)]
        \item $\{\underline{\lambda},\overline{\lambda}\}$ for some $\underline{\lambda}<\overline{\lambda}<0$ when $c<-c_0;$
        \item $\{\underline{\lambda}=\overline{\lambda}\}$ for some $\underline{\lambda}=\overline{\lambda}<0$ when $c=-c_0;$
        \item $\emptyset$ when $-c_0<c<c_0;$
        \item $\{\underline{\lambda}=\overline{\lambda}\}$ for some $\underline{\lambda}=\overline{\lambda}>0$ when $c=c_0;$
        \item $\{\underline{\lambda},\overline{\lambda}\}$ for some $0<\underline{\lambda}<\overline{\lambda}$ when $c>c_0;$
    \end{enumerate}
    \item If $\underline{\lambda}<\overline{\lambda},$ then $\Lambda_1(\mathcal{H}_{\lambda,c}+P)<0$ for all $\lambda\in(\underline{\lambda},\overline{\lambda}).$
\end{enumerate}
\end{lemma}
\begin{proof}
    Let 
    \begin{equation}\label{mu-def}
         \mu(\lambda):=\Lambda_1(\mathcal{H}_{\lambda,0}+P)=\Lambda_1(m(\lambda)\mathcal{D}+P),\quad \lambda\in\mathbb{R}.
    \end{equation}
     Since $J$ is even, we have $m(-\lambda)=m(\lambda)$. Furthermore, due to $\int_{\mathbb{R}}J(s)ds=1$, we have 
\begin{align}\label{m-form}
m(\lambda)=\int_{\mathbb{R}}J(s)\big(\cosh(\lambda s)-1\big)ds\geq0.   
\end{align}

Then, by (J1), we further have
\[
m''(\lambda)=\int_{\mathbb{R}}J(s)s^2e^{-\lambda s}ds>0.
\]
Therefore, $m$ is even and strictly convex on $\mathbb{R}$.

Set 
$$\widetilde{d}=\dfrac{1}{2}\min\{d_1,\dots,d_n\},\quad\widetilde{\Lambda}(\lambda)=\Lambda_1\Big(m(\lambda)\mathrm{diag}(d_1-\widetilde{d},\dots,d_n-\widetilde{d})+P\Big).$$
and let $F(t)=\Lambda_1\big(t\cdot \mathrm{diag}(d_1-\widetilde{d},\dots,d_n-\widetilde{d})+P\big)$, $t\in\mathbb{R}$. By the convexity property for the spectral bound of essentially nonnegative matrices \cite{cohen1981convexity}, $F(t)$ is convex.
In addition, $F(t)$ is nondecreasing in $t$.
Together with the fact that
$m$ is convex 
we see that  $\widetilde{\Lambda}$ is also convex.
Therefore, $\mu(\lambda)=\widetilde{\Lambda}(\lambda)+m(\lambda)\widetilde{d}$ is strictly convex.

In view of \eqref{m-form} and (J1), there exist $C>0$ and $\delta>0$ such that $m(\lambda)\geq C e^{\delta \lambda}$ for all large $\lambda$. Then 
$\lim_{\lambda\to\infty}\frac{m(\lambda)}{\lambda}=\infty$. Since $m(\lambda)\geq0$ and $\mathcal{D}\geq \widetilde{d} I_n$, the monotonicity of the spectral bound  gives
\[
\mu(\lambda)=
\Lambda_1(m(\lambda)\mathcal{D}+P)\geq
\Lambda_1(m(\lambda)\widetilde{d} I_n+P)
=\widetilde{d}m(\lambda)+\Lambda_1(P).
\]
This implies that 
\[
\mu(\lambda)\to\infty,\quad \frac{\mu(\lambda)}{\lambda}\to\infty\quad \text{ as } |\lambda|\to\infty.
\]
Now, for each fixed $c\in\mathbb{R}$, since we have  
$\Lambda_1(\mathcal{H}_{\lambda,c}+P)=\mu(\lambda)-c\lambda$, the function $\lambda\mapsto \Lambda_1(\mathcal{H}_{\lambda,c}+P)$
 is strictly convex, continuous, and satisfies
\[
\lim_{\lambda\to\pm\infty}\Lambda_1(\mathcal{H}_{\lambda,c}+P)=\infty.
\]
With these properties, the conclusion of Lemma~\ref{lmalambda1} can be derived easily. This completes the proof.
\end{proof}

\begin{remark}\label{rk:analytic}
If $P$ is irreducible in Lemma \ref{lmalambda1}, then $\mu$ is real analytic. Furthermore, we have $\mu''(\lambda)>0$ for all $\lambda\in\mathbb{R}$. Consequently, the unique real root corresponding to $c=\pm c_0$ has multiplicity two, and the two real roots corresponding to $|c|>c_0$ are simple.    
\end{remark}

The following proposition gives more qualitative properties of the solutions of the linear system \eqref{linearsystem1}, which will be used in the proof of Proposition \ref{lma:Persistence-Extinction}. 
This proposition is presented here, as its proof is independent of other sections.

\begin{proposition}\label{lma:increasing}
    Assume that $J$ satisfies {\rm (J0)--(J2)}.
    In the linear system \eqref{linearsystem1}, assume that $P$ is an essentially nonnegative irreducible matrix and $\Lambda_1(P)>0$.
    Also, define $c_0$ as in \eqref{generalc*}.    
    Fix $c\geq c_0$, and assume that $p(s)=(p_1(s),\dots,p_n(s))$ is a positive $\mathcal{C}^1$ solution of the linear system  \eqref{linearsystem1}.
    Then $p_i$ is strictly increasing on $\mathbb{R}$ for each $i\in[n]$ and $p(-\infty)=0_n$.
    Furthermore, we have
     \begin{align*}
        \lim\limits_{s\to-\infty}\int_{\mathbb{R}}J(s-s')p(s')ds'=0_n,\quad p'(-\infty)=0_n.
    \end{align*}
\end{proposition}
\begin{proof}
    Recall from the proof of Proposition \ref{linearexp} that the quantity
    \begin{align*}
        \lambda:=\min_{i\in[n]}\inf_{s\in\mathbb{R}}\frac{p_i'(s)}{p_i(s)}
    \end{align*}
    is a finite number satisfying $\Lambda_1(\mathcal{H}_{\lambda,c}+P)=0$ (where $\mathcal{H}_{\lambda,c}$ is defined as in \eqref{defHlambdaclinear}).
    Since $c\geq c_0$, Lemma \ref{lmalambda1} shows that $\lambda>0.$
    Then for each $i\in[n]$ and each $s\in\mathbb{R},$ we have $p_i'(s)\geq\lambda p_i(s)>0$. Thus, each $p_i$ is strictly increasing on $\mathbb{R}.$ Moreover, we have 
    $p_i(s)\leq p_i(0) e^{\lambda s}\mbox{ for }s\le 0 \mbox{ and }i\in[n].$
    Hence, 
    $p(-\infty)=0_n$.

    For the second assertion, Corollary \ref{lmalinear} (b) shows that there is a constant $M\in(0,\infty)$ such that
      \begin{align*}
          |p(s-s')|_1\leq e^{-Ms'}|p(s)|_1\mbox{ for all }s\in\mathbb{R}\mbox{ and }s'\leq 0.
      \end{align*}
      For any $\varepsilon>0$, choose, by (J1), an $L>0$ such that
      \begin{align*}
          \int_{-\infty}^{-L}J(s')e^{-Ms'}ds'<\varepsilon.
      \end{align*}
      Then for all $s< 0,$ we have
       \begin{align}\label{-inftytoL}
          \int_{-\infty}^{-L} J(s')|p(s-s')|_1ds'\leq \int_{-\infty}^{-L}J(s')e^{-Ms'}|p(s)|_1ds'< |p(0)|_1\varepsilon,
      \end{align}
      where we use $|p(s)|_1<|p(0)|_1$ for all $s<0$ in the last inequality.
      Since $p(-\infty)=0_n,$  we can assume that
     $|p(s)|_1<\varepsilon\mbox{ for all }s\leq -L.$
     This implies $|p(s-s')|_1<\varepsilon$ for all $s\le -2L$ and $s'\ge -L.$
     Then 
      \begin{align}\label{Ltoinfty}
          \int_{-L}^\infty J(s')|p(s-s')|_1ds'<\varepsilon\mbox{ for all }s\leq -2L.
      \end{align}
      Combining \eqref{-inftytoL} and \eqref{Ltoinfty},  we conclude that $\lim\limits_{s\to-\infty}\int_{\mathbb{R}}J(s')p(s-s')ds'=0_n.$

     Finally, we prove $p'(-\infty)=0_n$. To see this, note that $\Lambda_1(P)>0$ implies $c_0>0$. Hence $c\geq c_0>0$. Taking $s\to-\infty$ in \eqref{linearsystem1}, together with the facts that $c\neq 0$, 
     $\lim\limits_{s\to-\infty}\int_{\mathbb{R}}J(s')p(s-s')ds'=0_n$ and $Pp(-\infty)=0$,
     we immediately obtain  $p'(-\infty)=0_n$. This completes the proof.
\end{proof}
\section{Existence of traveling waves}\label{sec:existence and non}
In the rest of this paper, we will assume that {\rm (C1)--(C4)} hold without any further mention.
In this section, we are going to study properties of solutions of \eqref{epidemictws}, and in particular, we prove Theorem \ref{thm1}.

\subsection{Qualitative properties of the solutions}
First, by applying the maximum principles,  we have the following positivity result, whose proof is put in subsection \ref{sec:proofoflmapositive}.

\begin{lemma}\label{lmapositive}
 Assume that $J$ satisfies {\rm{(J0)}} and that $D_vg(E_0)$ is irreducible. 
 Let $c\in\mathbb{R}$. 
 Also, assume that $(U,V)$ is a bounded nonnegative solution of \eqref{epidemictws} on $\mathbb{R}$. Then the following assertions hold.
    \begin{enumerate}[(a)]
        \item $U(s)\leq K$ for all $s\in\mathbb{R}.$ 
        \item Either $(U,V)$ is positive on $\mathbb{R}$  or $(U,V)\equiv(K,0_n)$ on $\mathbb{R}.$
    \end{enumerate}
\end{lemma}

By utilizing our $\ell^1$-norm Harnack-type inequality (Theorem \ref{generalizedHarnack}) and the rescaling method, we have the following Lemmas \ref{nonlinearHarnack} and \ref{nonlinearexp}.

\begin{lemma}\label{nonlinearHarnack}
Assume that $J$ satisfies {\rm{(J0)}} and 
that $c$ is a nonzero constant. 
Let $(U,V)$ be a bounded positive solution of \eqref{epidemictws}.
Then there exists a constant $M>0$, depending only on
$J$, $c$, $D$, $D_vg(E_0)$ and $\mathbb{M}$, such that the following statements hold.
    \begin{enumerate}[(a)]
        \item For all $s\in\mathbb{R}$,
        \begin{align*}
            \frac{1}{|V(s)|_1}\left|\sum_{i=1}^nV_i'(s)\right|\leq M.
        \end{align*}
        \item For any $r>0,$ 
        \begin{align*}
            e^{-M r}\leq\frac{|V(\xi+s)|_1}{|V(\xi)|_1}\leq e^{M r}\mbox{ for all }\xi\in\mathbb{R}\mbox{ and }s\in[-r,r].
        \end{align*}
    \end{enumerate}  
\end{lemma}
\begin{proof}
   By  Lemma \ref{lmapositive} (a) and assumptions (C2) and (C4), $g(U(s),V(s))\leq PV(s)$ for all $s\in\mathbb{R},$ where $P=D_vg(E_0).$ Then
    \begin{align*}
        \mathbb{M}V(s)\leq cV'(s)-D\bigg(\int_{\mathbb{R}}J(s-s')V(s')ds'-V(s)\bigg)\leq PV(s),\quad s\in\mathbb{R}.
    \end{align*}
    Thus, the results follow directly from Theorem \ref{generalizedHarnack}.
\end{proof}

\begin{lemma}\label{nonlinearexp}
Assume that $J$ satisfies {\rm{(J0)}}  and {\rm (J1)}. 
Also, assume that $D_vg(E_0)$ is irreducible. Let $c$ be a nonzero constant.
   If \eqref{epidemictws} admits a bounded positive solution $(U,V)$  on $\mathbb{R}$ satisfying $(U,V)(-\infty)=E_0$, then there is a $\lambda\in\mathbb{R}$ such that $\Lambda_1(H_{\lambda,c}+D_vg(E_0))=0$, where $H_{\lambda,c}$ is defined as in \eqref{defofmlambda}.
\end{lemma}
\begin{proof}
    First, we assume that $c>0$. We shall claim that the quantity
    $$\lambda:=\min_{i\in[n]}\liminf_{s\to-\infty}\frac{V_i'(s)}{V_i(s)}$$
     is finite.
     To see this, note that for each $i\in[n]$, 
    \begin{align*}
    cV_i'(s)\geq d_i\int_{\R}J(s-s')V_i(s')ds'-d_iV_i(s)+\sum_{j=1}^nm_{ij}V_j(s)\geq -d_i V_i(s)+m_{ii} V_i(s),\quad s\in\R,  
    \end{align*}
    which implies that
    $V_i'\geq-\mu V_i$ on $\mathbb{R}$,
    where $\mu=\max\limits_{i\in[n]}\left(\frac{d_i+|m_{ii}|}{c}\right)$.
    This implies $\frac{V_i'}{V_i}\geq -\mu$ on $\mathbb{R}$ for each $i\in[n]$,  
    and thus $\lambda>-\infty.$ 
    Now, suppose that $\lambda=\infty.$ 
    Then there is an $s^{*}< -1$ such that 
    \begin{align*}
    \min_{i\in[n]}\frac{V_i'(s)}{V_i(s)}>M>0\quad \text{ for all } s\leq s^{*},
    \end{align*}
    where $M$ is the constant chosen in Lemma \ref{nonlinearHarnack} (a). This implies that 
    \begin{align*}
    \frac{1}{|V(s)|_1}\left|\sum_{i=1}^nV_i'(s)\right|=\frac{\sum_{i=1}^nV_i'(s)}{\sum_{i=1}^nV_i(s)}\geq \min_{i\in[n]}\frac{V_i'(s)}{V_i(s)}>M\quad \mbox{ for any }s \leq s^{*},
    \end{align*}
   a contradiction to Lemma \ref{nonlinearHarnack} (a).
     
    Thus, there are  a $j\in[n]$ and a sequence $s_k\to -\infty$ such that 
    $$\lambda=\lim_{k\to\infty}\frac{V_j'(s_k)}{V_j(s_k)}.$$
    Furthermore, in view of Lemma \ref{nonlinearHarnack} (b) (see also the proof of Proposition \ref{linearexp}), since $(U,V)(s_k)\to E_0$ as $k\to\infty$, the sequence 
    $$\widehat{p}^{(k)}(s):=\frac{V(s+s_k)}{|V(s_k)|_1},\quad s\in\mathbb{R},$$
    converges in $\mathcal{C}^1_{\mathrm{loc}}(\mathbb{R})$, up to a diagonal extraction, to a positive solution $p^\infty$ of the linear system
    \begin{align*}
        c(p^\infty)'(s)=D\left(\int_{\mathbb{R}}J(s-s')p^\infty(s')ds'-p^\infty(s)\right)+D_vg(E_0)p^\infty(s),\quad s\in\mathbb{R}.
    \end{align*}
    Proceeding as in Steps 2 and 3 of the proof of Proposition \ref{linearexp},
    we can show that this $p^\infty$ actually has the exact form $p^\infty(s)=e^{\lambda s}\psi,$ where $\psi$ is a positive unit vector. 
    Then we may deduce that $\Lambda_1(H_{\lambda,c}+D_vg(E_0))=0,$ as desired. The case $c<0$ can be done similarly by considering $\max\limits_{i\in[n]}\limsup\limits_{s\to-\infty}\frac{V_i'(s)}{V_i(s)}.$
\end{proof}

\begin{remark}
    Unlike the proof of Proposition \ref{linearexp}, we consider $\min\limits_{i\in[n]}\liminf\limits_{s\to-\infty}\frac{V_i'(s)}{V_i(s)}$ instead of $\min\limits_{i\in[n]}\inf\limits_{s\in\mathbb{R}}\frac{V_i'(s)}{V_i(s)}$.
    This is because for a positive solution $(U,V)$ of the nonlinear system \eqref{epidemictws} satisfying $(U,V)(-\infty)=E_0=(K,0_n)$, we need to extract a sequence $\{s_k\}$ such that $\frac{V(\cdot+s_k)}{|V(s_k)|_1}$ converges to a solution of the linearized system, which needs $(U,V)(s_k)\to E_0$ as $k\to\infty$.
\end{remark}

The following proposition gives a basic property of the propagation direction of traveling waves connecting $E_0$ at $-\infty$. In fact, when the linearized infected subsystem at $E_0$ is unstable, namely, $\Lambda_1(D_vg(E_0))>0$, a bounded positive wave cannot move to the right.
Therefore, any such wave must have $c>0$.

\begin{proposition}[Direction of propagation]\label{lmacpositive}
    Assume that $J$ satisfies {\rm (J0)--(J2)}.
    Let $c\in\mathbb{R}$ and assume that $D_vg(E_0)$ is irreducible with $\Lambda_1(D_vg(E_0))>0$. 
    If  \eqref{epidemictws} admits a bounded positive solution $(U,V)$ on $\mathbb{R}$ satisfying $(U,V)(-\infty)=E_0$, then $c>0.$
\end{proposition}
We put the proof of Proposition \ref{lmacpositive} in subsection \ref{sec:proofoflmacpositive}.

Next, we give a sufficient condition for the persistence of the solution of \eqref{epidemictws} at $\infty,$ whose proof also applies Theorem \ref{generalizedHarnack} and the rescaling method.
\begin{proposition}[Persistence]\label{prop:persist}
Assume that $J$ satisfies {\rm (J0)--(J2)}, and that
$D_vg(E_0)$ is irreducible with $\Lambda_1(D_vg(E_0))>0$.
If $c\geq c^*$ (where $c^*$ is defined as in \eqref{thm1c*}), then any bounded positive solution $(U,V)$ of \eqref{epidemictws}
is persistent at $\infty,$ i.e.,
    $$\liminf_{s\to\infty}U(s)>0\mbox{ and }\liminf_{s\to\infty}V_i(s)>0\mbox{ for each }i\in[n].$$
\end{proposition}
\begin{proof}
    We first deal with $U$. Assume to the contrary that $\liminf\limits_{s\to\infty}U(s)=0.$ Then there is a sequence $s_k\to\infty$ such that $U(s_k)\to 0$ as $k\to\infty.$ Since $\{(U,V)(\cdot+s_k)\}$ is uniformly bounded in $\mathcal{C}^{1,1}_{{\rm loc}}(\mathbb{R})$, and hence converges, up to a diagonal extraction, to some $(U^*,V^*)$ in $\mathcal{C}_{\mathrm{loc}}^1(\mathbb{R})$ with $U^*(0)=0$. This $(U^*,V^*)$ is again a bounded nonnegative  solution of \eqref{epidemictws}. By Lemma \ref{lmapositive} (b), $U^*$ is positive on $\mathbb{R},$ contradicting $U^*(0)=0.$ 
    Therefore, $\liminf\limits_{s\to\infty}U(s)>0$.

    Next, we deal with $V$. Assume to the contrary that $\liminf\limits_{s\to\infty}V_{i_0}(s)=0\mbox{ for some }i_0\in[n].$ Then there is a sequence $s_k\to\infty$ such that 
    \begin{align*}
        V_{i_0}(s_k)\to 0\mbox{ and }(V_{i_0})'(s_k)\leq 0\mbox{ for each }k.
    \end{align*}
    Similarly as in the previous paragraph, $\{(U,V)(\cdot+s_k)\}$ converges, up to a diagonal extraction, to some bounded nonnegative solution $(U^*,V^*)$ of \eqref{epidemictws} in $\mathcal{C}_{\mathrm{loc}}^1(\mathbb{R}).$ Since $V_{i_0}^*(0)=0,$ 
    by Lemma \ref{lmapositive} (b), we have $(U^*,V^*)=(K,0_n)$. Hence,
    $(U,V)(\cdot+s_k)\to (K, 0_n)$ in $\mathcal{C}_{\mathrm{loc}}^1(\mathbb{R})$, up to a diagonal extraction.
    Now, let
    $$\widehat{p}^{(k)}(s):=\frac{V(s+s_k)}{|V(s_k)|_1}.$$
    Proceeding as in Step 2 of the proof of Proposition \ref{linearexp} (replacing $-\infty$ by $\infty$; see also Lemma \ref{nonlinearexp}),
    we can show that $\{\widehat{p}^{(k)}\}$ converges, up to a diagonal extraction, to some function $p^\infty$ in $\mathcal{C}_{\mathrm{loc}}^1(\mathbb{R}),$ which is a positive solution of the linear system 
    \begin{align*}
        c(p^\infty)'(s)=D\left(\int_{\mathbb{R}}J(s-s')p^\infty(s')ds'-p^\infty(s)\right)+D_vg(E_0)p^\infty(s),\quad s\in\mathbb{R}.
    \end{align*}
    with $|p^\infty(0)|_1=1>0,\quad(p^\infty)_{i_0}'(0)=\lim\limits_{k\to\infty}\frac{V_{i_0}'(s_k)}{|V(s_k)|_1}\leq 0~(\mbox{up to a diagonal extraction}).$
    But $c\geq c^*$ gives $(p^\infty)_{i_0}'(0)>0$ by Proposition \ref{lma:increasing}.
    This contradiction completes the proof.
\end{proof}

Inspired by \cite[Lemma 6.12]{girardin2017non}, we are going to prove a  persistence-extinction dichotomy (Proposition \ref{lma:Persistence-Extinction}), which plays an important role in guaranteeing the boundary condition of the critical wave ($c=c^*$) at $-\infty$. 

\begin{proposition}[Persistence-Extinction Dichotomy]\label{lma:Persistence-Extinction}
Assume that $J$ satisfies {\rm (J0)--(J2)}, and that $D_vg(E_0)$ is irreducible with $\Lambda_1(D_vg(E_0))>0$. 
    Then for each $c\geq c^*$ (where $c^*$ is defined as in \eqref{thm1c*}), there is a positive constant $\eta_c>0$ such that any bounded nonnegative solution $(U,V)$ of \eqref{epidemictws}  
    (no behavior near $-\infty$ is assumed) 
    satisfies the following dichotomy:
    \begin{enumerate}[(i)]
        \item $\lim\limits_{s\to-\infty}V(s)=0_n$, or
        \item $\inf\limits_{s\in(-\infty,0)}|V(s)|_1\geq\eta_c>0$.
    \end{enumerate}
\end{proposition}
\begin{proof}
First, if $V_{i_0}(s_0)=0$ for some $i_0\in[n]$ and some $s_0\in\mathbb{R},$ then by Lemma \ref{lmapositive} (b), we have $V\equiv 0_n$ on $\mathbb{R}$ and so (i) must hold. So we may assume that $V$ is positive everywhere. The remainder of the proof is cut into three steps.

{\bf Step 1}. {\it If $\inf\limits_{s\in(-\infty,0)}|V(s)|_1=0$, then (i) holds.}\\
     Suppose the contrary, that is, $\inf\limits_{s\in(-\infty,0)}|V(s)|_1=0$ but $\limsup\limits_{s\to-\infty}V_{i_0}(s)>0$ for some $i_0\in[n]$. 
     Then there is a sequence $s_k\to-\infty$ such that 
     \begin{align*}
        V_{i_0}(s_k) \to 0\mbox{ and }V_{i_0}'(s_k)\leq 0\mbox{ for each }k.
     \end{align*}

     Note that $\{(U,V)(\cdot+s_k)\}$ is uniformly bounded in $\mathcal{C}^{1,1}_{{\rm loc}}(\mathbb{R})$, and hence converges, up to a diagonal extraction, to some nonnegative solution $(U^*,V^*)$ of \eqref{epidemictws} in $\mathcal{C}^1_{\mathrm{loc}}(\mathbb{R})$.
     Since $V_{i_0}^*(0)=0$, we see that $(U^*,V^*)\equiv(K,0_n)$ by Lemma \ref{lmapositive} (b). 
     We shall still denote this subsequence as $\{(U,V)(\cdot+s_k)\}$.

     Put
     \begin{align*}
         \widehat{p}^{(k)}(s):=\frac{V(s+s_k)}{|V(s_k)|_1},\quad k\in\mathbb{N},\quad s\in\mathbb{R}.
     \end{align*}
     In view of Lemma \ref{nonlinearHarnack} (b) (see also the proof of Lemma \ref{nonlinearexp}), since $(U,V)(s_k)\to (K,0_n)$  as $k\to\infty,$ the sequence $\{\widehat{p}^{(k)}\}$ converges in $\mathcal{C}^1_{\mathrm{loc}}(\mathbb{R})$, up to a diagonal extraction, to a positive (due to $|p(0)|_1=1$ and Lemma \ref{SMPlinearized}) solution $p^\infty(s)$ of the linear system
     \begin{align*}
         c(p^\infty)'(s)=D\left(\int_\mathbb{R}J(s')p^\infty(s-s')ds'-p^\infty(s)\right)+D_vg(E_0)p^\infty(s),\quad s\in\mathbb{R}.
     \end{align*}
     Since $c\geq c^*,$ Proposition \ref{lma:increasing} deduces that each $p_i^\infty$ is strictly increasing on $\mathbb{R}.$
     In particular, $(p_{i_0}^\infty)'(0)>0.$
     However, 
     $p_{i_0}'(0)=\lim\limits_{k\to\infty}\frac{V_{i_0}'(s_k)}{|V(s_k)|_1}\leq 0$ (up to a diagonal extraction), a contradiction.
     This proves Step 1.

{\bf Step 2}. {\it Let $\Sigma$ be the set of bounded nonnegative  solutions $(U,V)$ of \eqref{epidemictws} with $\inf\limits_{s\in(-\infty,0)}|V(s)|_1>0$. If $\Sigma$ is nonempty, then  }
$$\eta_c:=\inf\left\{\inf_{s\in(-\infty,0)}|V(s)|_1:(U,V)\in\Sigma\right\}>0.$$
     Suppose the contrary, i.e., there is a sequence $\{(U^{(k)},V^{(k)})\}$ of bounded nonnegative  solutions of \eqref{epidemictws} with $$\beta_k:=\inf\limits_{s\in(-\infty,0)}|V^{(k)}(s)|_1>0\mbox{ for each }k\in\mathbb{N}$$ such that $$\lim\limits_{k\to\infty} \beta_k=0.$$
     For each $k,$ choose an $s_k\in(-\infty,0)$ such that $\beta_k\leq|V^{(k)}(s_k)|_1<2\beta_k.$ Next, define 
     $$\widehat{q}^{(k)}(s):=\dfrac{V^{{(k)}}(s+s_k)}{\beta_k}\mbox{ for }s\in\mathbb{R}\mbox{ and }k\in\mathbb{N}.$$
     Recall from Lemma \ref{nonlinearHarnack} that 
     there is a positive constant $M>0$ such that for any $r>0,$
     $$\sup_{k\in\mathbb{N},s\in[-r,r]}\frac{|V^{{(k)}}(s+s_k)|_1}{|V^{{(k)}}(s_k)|_1}\leq e^{Mr}.$$
     Then for any $k\in\mathbb{N},$ $r>0$ and any $s\in[-r,r]$, $|V^{(k)}(s+s_k)|_1\leq e^{Mr}|V^{(k)}(s_k)|_1\leq 2e^{Mr}\beta_k$, 
     which means $V^{(k)}(\cdot+s_k)\to 0_n$, up to a diagonal extraction, in $\mathcal{C}_{\mathrm{loc}}^1(\mathbb{R})$ as $k\to\infty$. 
    Observing the equation satisfied by $U$, Lemma \ref{lmapositive} (b) deduces that $\{(U^{(k)},V^{(k)})(\cdot+s_k)\}$ converges, up to a diagonal extraction, to $(K,0_n)$ in $\mathcal{C}_{\mathrm{loc}}^1(\mathbb{R})$ as $k\to\infty$.
     On the other hand, $\{\widehat{q}^{(k)}\}$ is uniformly bounded in $\mathcal{C}^{1,1}_{{\rm loc}}(\mathbb{R})$,
     and thus converges, up to a diagonal extraction, to some nonnegative $q^\infty$ in $\mathcal{C}^1_{\mathrm{loc}}(\mathbb{R})$ which satisfies that
     \begin{align*}
         c(q^\infty)'(s)=D\left(\int_{\mathbb{R}}J(s')q^\infty(s-s')ds'-q^\infty(s)\right)+D_vg(E_0)q^\infty(s)\mbox{ for all }s\in\mathbb{R},
     \end{align*}
     and that 
     \begin{align}\label{contrad}
      |q^\infty(s)|_1\geq 1\quad \text{ for all } s\in(-\infty,0].   
     \end{align}
     In particular, $q^\infty\not\equiv0_n$.
     By Lemma \ref{SMPlinearized}, $q^\infty$ is positive on $\mathbb{R}.$
    By Proposition \ref{lma:increasing},
     $q^\infty(-\infty)=0_n$, which contradicts \eqref{contrad}. This completes the proof of Step 2.
     

{\bf Step 3}. {\it The number $\eta_c>0$ satisfies the requirement.}\\
     If (i) does not hold for some bounded positive solution $(U,V)$ of \eqref{epidemictws}, 
     then by Step 1, $\inf\limits_{s\in(-\infty,0)}|V(s)|_1>0$ and thus $(U,V)\in\Sigma$. Such $(U,V)$ must satisfy (ii), by the definition of $\eta_c$. This lemma is then proved.
\end{proof}

\medskip

\subsection{Nonexistence results}

Now, we are ready to establish the nonexistence result (Theorem \ref{thm1} (a) and (b)), which is formulated as Proposition~\ref{prop:thm1-a}.

\begin{proposition}\label{prop:thm1-a}
Assume that $J$ satisfies {\rm (J0)--(J2)}, and that
$D_vg(E_0)$ is irreducible.
Then the following hold:
\begin{enumerate}[(a)]
    \item If $\Lambda_1(D_vg(E_0))<0,$ then \eqref{epidemictws} admits no bounded positive solution for any wave speed $c\in\mathbb{R}$.
    \item If $\Lambda_1(D_vg(E_0))>0$, then \eqref{epidemictws} admits no bounded positive solution satisfying $(U,V)(-\infty)=E_0$ for any wave speed $c<c^*$, where $c^*$ is defined as in \eqref{thm1c*}.
\end{enumerate}
\end{proposition}
\begin{proof}
We denote 
    $P=D_vg(E_0), \Lambda=\Lambda_1(P)$.
    First suppose that  $\Lambda<0$. Also, suppose that \eqref{epidemictws} admits a bounded positive solution $(U,V)(s).$ 
    This means that \eqref{epidemicmodel} admits a bounded positive solution of the form 
    $(u,v)(x,t)=(U,V)(x+ct)$ for $x\in\mathbb{R},~t>0.$
    By Lemma \ref{lmapositive} (a) and properties (C2) and (C4), $$g(u,v)(x,t)\leq Pv(x,t)\mbox{ for }x\in\mathbb{R},t>0.$$
On the other hand, let $\eta$ be the Perron right positive eigenvector of $P$ and let $\widehat{v}(x,t):=\tau e^{\Lambda t}\eta$ for $x\in\mathbb{R}$ and $t\geq 0$, where $\tau>0$ is a positive constant such that $\tau\eta>v(x,0)$ for all $x\in\mathbb{R}.$ Then $w:=\widehat{v}-v$ satisfies
\begin{align*}
    w_i(x,0)&\geq 0\mbox{ for }x\in\mathbb{R};\\
    \frac{\partial w_i}{\partial t}(x,t)&\geq d_i\int_{\mathbb{R}} J(x-y)w_i(y,t)dy-d_iw_i(x,t)+\sum_{j=1}^nP_{ij}w_j(x,t)\mbox{ for }x\in\mathbb{R},t>0
\end{align*}
for each $i\in[n].$
By the weak maximum principle (Proposition \ref{maxprinciple}), $w(x,t)\geq0$ for $x\in\mathbb{R}$ and $t\geq 0.$ Thus, $0\leq v(x,t)\leq\tau e^{\Lambda t}\eta\mbox{ for }x\in\mathbb{R}\mbox{ and }t\geq 0.$ Since $\Lambda<0,$ we see that $v(x,t)\to 0$ as $t\to\infty$ uniformly in $x$ and hence it is impossible for \eqref{epidemictws} to have a bounded positive solution. Then part (a) is proved.

    Next, suppose that $\Lambda>0$. From Proposition \ref{lmacpositive}, we see that if $-\infty<c\leq 0,$ then \eqref{epidemictws} admits no bounded positive solutions satisfying $(U,V)(-\infty)=E_0$. 
    
    Finally, consider the case $\Lambda>0$ and $0<c<c^*.$ Suppose that \eqref{epidemictws} admits a bounded positive solution satisfying $(U,V)(-\infty)=E_0$.
    By Lemma \ref{nonlinearexp},
    there is a real number $\lambda$ such that $\Lambda_1(H_{\lambda,c}+P)=0,$ which contradicts Lemma \ref{lmalambda1} (b)(iii). Thus, \eqref{epidemictws} admits no bounded positive solution satisfying $(U,V)(-\infty)=E_0$. This completes the proof.
\end{proof}

\subsection{Existence results}
In this section, we always assume that $J$ satisfies (J0)--(J2), that {\rm(G1)} holds, and that 
$c\ge c^*$ ($c^*$ is defined as in \eqref{thm1c*}).
Furthermore, we shall assume
\begin{enumerate}
    \item[(H1)] $D_vg(E_0)$ is irreducible with $\Lambda_1(D_vg(E_0))>0$;
    \item[(H2)] $\mathbb{M}:=\begin{bmatrix}
        m_{ij}
    \end{bmatrix}$ satisfies $\Lambda_1(\mathbb{M})<0.$
\end{enumerate}
We state the existence result for traveling waves with a $\mathcal{C}^{1,1}$ estimate when $c>c^*$ as follows,
whose proof utilizes the method of super- and sub-solutions with Schauder’s fixed-point argument, and is put in Subsection \ref{sec:proofc>c*}.
\begin{proposition}\label{superexistence}
Assume that $J$ satisfies {\rm (J0)--(J2)}, and that {\rm(G1)}, {\rm (H1)} and {\rm (H2)} hold.
    If $c>c^*,$ then \eqref{epidemictws} admits a bounded positive solution $(U_c,V_c)(s)$ satisfying $(U_c,V_c)(-\infty)=E_0$. Furthermore,  given a bounded subset $E$ of $(c^*,\infty)$, there exists a constant $C>0$ independent of $c$ such that 
    \begin{align}\label{unibddwrtc}
        \|(U_c,V_c)\|_{\mathcal{C}^{1,1}(\mathbb{R},\mathbb{R}^{n+1})}\leq C\mbox{ for each }c\in E,
    \end{align}
    where the $\mathcal{C}^{1,1}$-norm of $p:=(U_c,V_c)$ on an interval $I$ is defined as
    $$\|p\|_{\mathcal{C}^{1,1}(I,\mathbb{R}^{n+1})}=\sup_{s\in I}|p(s)|_1+\sup_{s\in I}|p'(s)|_1+\sup_{s_1,s_2\in I,s_1\neq s_2}\frac{|p'(s_1)-p'(s_2)|_1}{|s_1-s_2|}.$$
\end{proposition}

With the estimate \eqref{unibddwrtc} at hand, we can prove the existence of a traveling wave solution when $c=c^*$ by using the persistence-extinction dichotomy (Proposition \ref{lma:Persistence-Extinction}) and again the rescaling method.

We first prove the following lemma,
analogous to \cite[Lemma 5.1]{lam2018traveling},
but considering the nonlocal diffusion. This
plays a crucial role in the proof of Proposition \ref{prop:c=c*}.

\begin{lemma}\label{limsuppositive}
    Assume that $J$ satisfies {\rm (J0)--(J2)}, and that {\rm(G1)}, {\rm(H1)} and {\rm (H2)} hold. Pick a decreasing sequence $\{c_k\}$ in $(c^*,c^*+1)$ that converges to $c^*$ and denote $(U_{c_k},V_{c_k})$ as a bounded positive solution of \eqref{epidemictws} with $(U_{c_k},V_{c_k})(-\infty)=(K,0_n)$ for $k\in\mathbb{N}$ when $c=c_k$, obtained in Proposition \ref{superexistence}. 
    Then 
    \begin{align*}
        \limsup_{k\to\infty}\sup_{s\in\mathbb{R}}|K-U_{c_k}(s)|>0\quad\mbox{or}\quad\limsup_{k\to\infty}\sup_{s\in\mathbb{R}}|V_{c_k}(s)|_1>0.
    \end{align*}
\end{lemma}
\begin{proof}
    Suppose the contrary, that is, as $k\to\infty,$
    \begin{align}\label{supto0}
        \sup_{s\in\mathbb{R}}|K-U_{c_k}(s)|\to 0\quad\mbox{and}\quad\sup_{s\in\mathbb{R}}|V_{c_k}(s)|_1\to 0.
    \end{align}
    We claim that under the hypothesis \eqref{supto0}, there is a $k_0\in\mathbb{N}$ such that
    \begin{align}\label{Vcknondecreasing}
        V_{c_k}'\geq 0_n\mbox{ on }\mathbb{R}\mbox{ for each }k\ge k_0.
    \end{align}
    Suppose that \eqref{Vcknondecreasing} does not hold. Then, by passing to a subsequence and relabeling if necessary, there exist an $i_0\in[n]$ and a sequence $\{\zeta_k\}$ in $\mathbb{R}$ such that $(V_{c_k})_{i_0}'(\zeta_k)<0$ for all $k\in\mathbb N$. Moreover, since $(V_{c_k})_{i_0}(-\infty)=0$ and $(V_{c_k})_{i_0}(\zeta_k)>0$, there 
    is a sequence $\{s_k\}$ such that
    $$(V_{c_k})_{i_0}'(s_k)=0\mbox{ for all }k\in\mathbb{N}.$$
    Consider the sequence
    \begin{align*}
        \widehat{p}^{(k)}(s):=\frac{V_{c_k}(s+s_k)}{|V_{c_k}(s_k)|_1}.
    \end{align*}
    In view of Remark \ref{rmk:uniforminc},
    the constant $M$ in Lemma \ref{nonlinearHarnack} (a) can be chosen to be independent of $c\in(c^*,c^*+1)$.
    Also, \eqref{supto0} implies $|V_{c_k}(s_k)|_1\to 0$ and $U_{c_k}(\cdot+s_k)\to K$ uniformly on $\mathbb{R}$ as $k\to\infty.$
    Then the sequence $\{\widehat{p}^{(k)}\}$ converges in $\mathcal{C}^1_{\mathrm{loc}}(\mathbb{R})$, up to a diagonal extraction, to a positive solution $p^\infty$ of the linear system
        $$c^*(p^\infty)'(s)=D\left(\int_{\mathbb{R}}J(s-s')p^\infty(s')ds'-p^\infty(s)\right)+D_vg(E_0)p^\infty(s),\quad s\in\mathbb{R},$$
        with $|p^\infty(0)|_1=1$ and $(p^\infty)_{i_0}'(0)=0.$
        However, Proposition \ref{lma:increasing} deduces that $(p^\infty)_{i_0}'(0)>0,$ which is a contradiction.


        Consequently, \eqref{Vcknondecreasing} holds. Namely,
        for each $k\geq k_0$ and each $i\in[n],$ the function $(V_{c_k})_i$ is nondecreasing on $\mathbb{R}.$ 
        Since $V_{c_k}$ is bounded on $\mathbb{R}$, $V_{c_k}(\infty)$ exists and is finite for each $k\ge k_0$.
        On the other hand, since $V_{c_k}'\geq 0_n\mbox{ on }\mathbb{R}$ and $V_{c_k}(-\infty)=0_n$ for each $k\ge k_0,$ we see that $(V_{c_k})'_i$ is Lebesgue integrable on $\mathbb{R}$ for each $k\ge k_0$ and each $i\in[n]$. 
        Furthermore, in view of \eqref{unibddwrtc}, each $V_{c_k}'$ is uniformly continuous on $\mathbb{R}$.
        Therefore,
        \begin{align}\label{Vck'infty=0}
            V_{c_k}'(\infty)=0_n\mbox{ for each }k\geq k_0.
        \end{align}
        On the other hand, since each $V_{c_k}$ is bounded on $\mathbb{R},$ it is easy to verify that
        \begin{align}\label{J*v-limit}
            \lim\limits_{s\to\infty}\int_{\mathbb{R}}J(s')V_{c_k}(s-s')ds'=V_{c_k}(\infty)\mbox{ for each }k\ge k_0.
         \end{align}
        Combining \eqref{Vck'infty=0}, \eqref{J*v-limit} and the equation of $V$, we obtain that for all $k\geq k_0$,
\begin{equation}\label{g-to-0}
    g\Big(U_{c_k}(s),V_{c_k}(s)\Big)\to0_n
    \quad\mbox{as }s\to\infty.
\end{equation}

Since $g(\cdot,0_n)\equiv 0_n$ and $g\in \mathcal{C}^1$, by
\eqref{supto0}, we have
\begin{align*}
  g\Big(U_{c_k}(s),V_{c_k}(s)\Big)=D_vg(E_0)V_{c_k}(s)+E_k(s),  
\end{align*}
where there exists $\rho_k\to0$ such that
\[
|E_k(s)|_1\leq \rho_k | V_{c_k}(s)|_1 \quad \text{ for all } s\in\mathbb{R}. 
\]
Together with \eqref{g-to-0}, we obtain
        \begin{align}\label{o(1)}
            D_vg(E_0)V_{c_k}(\infty)=o(|V_{c_k}(\infty)|_1) \quad \text{ as }  k\to\infty.
        \end{align}


        In view of Proposition \ref{prop:persist}, we can consider the bounded sequence of vectors
        \begin{align*}
            q_k:=\frac{V_{c_k}(\infty)}{|V_{c_k}(\infty)|_1},\quad k\geq k_0.
        \end{align*}
        Thanks to \eqref{o(1)}, we have 
        \[
        D_vg(E_0)q_k=o(1) \quad \text{ as }  k\to\infty.
        \]
        By passing to a subsequence, $q_k\to q_\infty$ as $k\to\infty$ for some nonnegative vector $q_\infty\in\mathbb{R}^n$ with
        $$D_vg(E_0)q_\infty=0_n,\quad|q_\infty|_1=1.$$
        Since $D_vg(E_0)$ is irreducible, we see that $q_\infty$ is a Perron eigenvector of $D_vg(E_0)$ corresponding to the principal eigenvalue $\Lambda_1(D_vg(E_0))=0$, which contradicts $\Lambda_1(D_vg(E_0))>0$.
        This completes the proof.
\end{proof}

The following result provides the existence of a critical wave.

\begin{proposition}\label{prop:c=c*}
Assume that $J$ satisfies {\rm (J0)--(J2)}, and that {\rm(G1)}, {\rm(H1)} and {\rm (H2)} hold. Then \eqref{epidemictws} admits a bounded positive solution $(U,V)$ satisfying $(U,V)(-\infty)=E_0$ when $c=c^*.$
\end{proposition}
\begin{proof}
Pick the sequence $\{(U_{c_k},V_{c_k})\}$ as in Lemma \ref{limsuppositive} and let
\[
\gamma_0:=\limsup_{k\to\infty} \max\Big\{\sup_{s\in\mathbb{R}}|K-U_{c_k}(s)|,\sup_{s\in\mathbb{R}}|V_{c_k}(s)|_1\Big\}.
\]
By Lemma~\ref{limsuppositive}, we see that $\gamma_0>0$.

Set $\varepsilon\in(0,\min\{\frac{\gamma_0}{2},\frac{\eta_{c^*}}{2},\frac{K}{2}\})$ small enough, where $\eta_{c^*}$ is chosen as in Proposition \ref{lma:Persistence-Extinction} for $c=c^*$.

Since $\gamma_0>2\varepsilon$, by passing to a subsequence, still denoted by $\{c_k\}$, we may assume that
\begin{align}\label{>2ep}
 \max\left\{\sup_{s\in\mathbb R}|K-U_{c_k}(s)|,\, \sup_{s\in\mathbb R}|V_{c_k}(s)|_1
\right\}
>2\varepsilon,\quad \text{ for all } k\in\mathbb{N}.   
\end{align}
For each $k\in\mathbb N$, set
\[
\mathcal{A}_k(s):=
\max\left\{ |K-U_{c_k}(s)|, |V_{c_k}(s)|_1 \right\},
\quad s\in\mathbb{R}.
\]

        Since $(U_{c_k},V_{c_k})(-\infty)=(K,0_n)$ for each $k\in\mathbb{N},$ we have $\mathcal{A}_k(-\infty)=0$. Moreover, by \eqref{>2ep}, we obtain $\sup_{s\in\R}\mathcal{A}_k(s)>2\varepsilon$. Therefore, we can define 
        \[
        t_k:=\inf\{s\in\mathbb{R}: \mathcal{A}_k(s)\geq \varepsilon\}.
        \]
        By the continuity of $\mathcal{A}_k(\cdot)$, we have
        \[
        \mathcal{A}_k(t_k)=\varepsilon,\quad  \mathcal{A}_k(t_k+s)<\varepsilon\quad \text{for all } s<0.
        \]
        Consequently,
        we have
        $$U_{c_k}(t_k+s)>K-\varepsilon\quad\mbox{and}\quad |V_{c_k}(t_k+s)|_1<\varepsilon\quad\mbox{for all }s<0$$
        and
        $$U_{c_k}(t_k)=K-\varepsilon\quad\mbox{or}\quad |V_{c_k}(t_k)|_1=\varepsilon.$$
        In view of \eqref{unibddwrtc}, the sequence  $\{(U_{c_k},V_{c_k})(\cdot+t_k)\}$ converges, up to a diagonal extraction, to some nonnegative solution $(U,V)$ of \eqref{epidemictws} with $c=c^*$ in $\mathcal{C}^1_{\mathrm{loc}}(\mathbb{R})$ such that
        \begin{align*}
            U(s)\geq K-\varepsilon\quad\mbox{and}\quad |V(s)|_1\leq \varepsilon\quad\mbox{for all }s<0
        \end{align*}
        and that 
        \begin{align*}
            U(0)=K-\varepsilon\quad\mbox{or}\quad |V(0)|_1=\varepsilon.
        \end{align*}
        By Lemma \ref{lmapositive}, $(U,V)$ is positive on $\mathbb{R}.$
        Furthermore, since $\inf\limits_{s\in(-\infty,0)}|V(s)|_1\leq\varepsilon<\eta_{c^*},$
        Proposition \ref{lma:Persistence-Extinction} deduces that $V(-\infty)=0_n$.

Finally, we claim that $U(-\infty)=K.$
This can be done similarly as in Step 4 of the proof of \cite[Theorem 3.2]{yang2018traveling}.
    Suppose that $\underline{U}_*:=\liminf\limits_{s\to-\infty}U(s)<K.$ Choose a sequence $r_k\to-\infty$ such that $U(r_k)\to\underline{U}_*$ as $k\to\infty.$ Let $(U^{(k)},V^{(k)})(s):=(U,V)(s+r_k).$ Then $(U^{(k)},V^{(k)})(s)$ converges to $(K,0_n)$, up to a diagonal extraction, in $\mathcal{C}^1_{\mathrm{loc}}(\mathbb{R})$ by \eqref{unibddwrtc} and Lemma \ref{lmapositive} (b). In particular, $U^{(k)}(0)=U(r_k)\to K=\underline{U}_*$ as $k\to\infty,$ a contradiction. Therefore, 
    $$\liminf\limits_{s\to-\infty}U(s)\geq K\geq\limsup_{s\to-\infty}U(s)$$
    and so $U(-\infty)=K$.
We then conclude that $(U,V)$ is a bounded positive solution satisfying $(U,V)(-\infty)=E_0$ when $c=c^*.$
\end{proof}

Finally, 
we close this section by proving Theorem~\ref{thm1} as follows.

\begin{proof}[Proof of Theorem~\ref{thm1}]
Parts (a) and (b) follow from Proposition \ref{prop:thm1-a}, while part (c) follows from Propositions \ref{prop:persist}, \ref{superexistence} and \ref{prop:c=c*}.
\end{proof}

\section{Asymptotic behaviors at $-\infty$}\label{sectionconvergencerate}

In this section, we establish the asymptotic behaviors of $(U,V)$ at $-\infty$ via Ikehara's theorem.
To make the Laplace transforms legitimate, we need the following results for exponential boundedness of $(U,V)$.

\begin{lemma}\label{lem:for-Laplace}
Assume that $J$ satisfies {\rm (J0)--(J2)} and that $D_vg(E_0)$ is irreducible with $\Lambda_1(D_vg(E_0))>0$. Let $c\geq c^*$ (where $c^*$ is defined as in \eqref{thm1c*}). Also, assume that \eqref{epidemictws} admits a bounded positive solution $(U,V)$ on $\mathbb{R}$ with $(U,V)(-\infty)=E_0$.
    Then for any $\varepsilon\in(0,\lambda_c),$
    we have
    \begin{align}\label{V-sup}
     \sup_{s\in\mathbb{R}}e^{-(\lambda_c-\varepsilon)s}|V(s)|_1<\infty \quad\mbox{and}\quad\sup_{s\in\mathbb{R}}e^{-(\lambda_c-\varepsilon)s}|V'(s)|_1<\infty,   
    \end{align}
    where $\lambda=\lambda_c$ is the smallest positive root of $\Lambda_1(H_{\lambda,c}+D_vg(E_0))=0.$ 
    Furthermore, for any $\mu_0\in(0,\min\{\nu_c,\lambda_c\})$, 
    \begin{align}\label{U-sup}
        \sup_{s\in\mathbb{R}}e^{-\mu_0s}(K-U(s))<\infty.
    \end{align}
\end{lemma}
\begin{proof}
    Recall from the proof of Lemma \ref{nonlinearexp} that the quantity
$$\lambda_*:=\min_{i\in[n]}\liminf_{s\to-\infty}\frac{V_i'(s)}{V_i(s)}$$
    is finite and satisfies $\Lambda_1(H_{\lambda_*,c}+D_vg(E_0))=0.$ Since $c\geq c^*,$ we have $\lambda_*\geq\lambda_c>0.$ However, by the definition of $\lambda_*,$ for any $\varepsilon\in(0,\lambda_c)$, there is an $M_\varepsilon>0$ such that $$\frac{V_i'(s)}{V_i(s)}>\lambda_*-\varepsilon\mbox{ for each }i\in[n]\mbox{ and each }s<-M_\varepsilon.$$
    Then
    $$\ln(V_i(-M_\varepsilon))>\ln(V_i(s))+(\lambda_*-\varepsilon)(-M_\varepsilon-s)\mbox{ for each }i\in[n]\mbox{ and each }s<-M_\varepsilon.$$
    Thus,
    $$e^{-(\lambda_*-\varepsilon)s}V_i(s)<e^{(\lambda_*-\varepsilon)M_\varepsilon}V_i(-M_\varepsilon)\mbox{ for each }i\in[n]\mbox{ and each }s<-M_\varepsilon.$$
    Since $V$ is bounded on $\mathbb{R},$ we see that 
    $\sup_{s\in\mathbb{R}}e^{-(\lambda_c-\varepsilon)s}|V(s)|_1<\infty.$
    Furthermore, by the equation satisfied by $V,$ we have 
    $\sup_{s\in\mathbb{R}}e^{-(\lambda_c-\varepsilon)s}|V'(s)|_1<\infty.$ This completes the proof of \eqref{V-sup}.


We now prove \eqref{U-sup}. Set
$W(s):=K-U(s)$. Then $W$ satisfies
\begin{align}\label{W-equation}
cW'(s)-d_0\left(
\int_{\mathbb R}J(s-s')W(s')ds'-W(s)\right)
+\delta W(s)=g_0(U(s),V(s))\geq 0,\quad s\in\mathbb{R}.
\end{align}
By Lemma \ref{lmapositive} (a), $W\geq0$ on $\mathbb R$. Also, by assumptions (C1), (C2) and (C4), together with $0<U\leq K$,
there exists $C_0\ge 0$ such that
\begin{align}\label{g0sublinear}
    0\leq g_0(U(s),V(s))\leq C_0|V(s)|_1\quad \text{for all } s\in\mathbb{R}.
\end{align}
Then \eqref{V-sup} and \eqref{g0sublinear} together give
\begin{align}\label{g0decay}
    g_0(U(s),V(s))=O(e^{\kappa s})\mbox{ as }s\to-\infty \quad \text{for any} \ \kappa\in(0,\lambda_c). 
\end{align}
We now choose any $\mu_0$ such that
$0<\mu_0<\min\{\lambda_c,\nu_c\}$.
Since $\Delta_c(\lambda)>0$ for $\lambda\in(0,\nu_c)$, we have
\[
\Delta_c(\mu_0)
=
c\mu_0-d_0m(\mu_0)+\delta>0.
\]
Furthermore, $g_0(U,V)(\cdot)$ is bounded on $\mathbb R$, which together with $\kappa>\mu_0$ and \eqref{g0decay} implies 
\[
\sup_{s\in\mathbb R}e^{-\mu_0s}g_0(U(s),V(s))<\infty.
\]

We now choose $A>0$ such that
\[
g_0(U(s),V(s))\leq A\Delta_c(\mu_0)e^{\mu_0s}
\quad\text{for all }s\in\mathbb R.
\]
Let
$\overline W(s):=Ae^{\mu_0s}$.
A direct calculation yields 
\[
c\overline W'(s)-d_0\Big(\int_{\mathbb R}J(s-s')\overline  W(s')ds'-\overline W(s)\Big)
+\delta\overline W(s)
=
A\Delta_c(\mu_0)e^{\mu_0s}
\geq g_0(U(s),V(s)).
\]
Together with $\overline W(-\infty)=0$ and 
$\overline W(\infty)=\infty$, by a standard comparison argument, we obtain
\[
0\leq K-U(s)=W(s)\leq Ae^{\mu_0s}=\overline{W}(s)
\quad\text{for all}\quad s\in\mathbb R.
\]
This completes the proof.
\end{proof}

\begin{remark}
 In the study of traveling waves for non-monotone nonlocal diffusion systems, establishing exponential boundedness such as 
 \eqref{V-sup} typically requires intricate estimates (see, e.g. \cite[Lemma 4.5]{diekmann1978bounded}).
In contrast, our proof utilizes the $\ell^1$-norm Harnack-type inequality to provide a 
short proof, achieving a 
sharper exponential bound.
\end{remark}

We state Ikehara's theorem as follows, which can be easily obtained by performing a change of variables $y=-x$ from the version in \cite[Proposition 2.3]{carr2004uniqueness}.
\begin{lemma}[Ikehara's Theorem]
     Let $f:(-\infty,0)\to(0,\infty)$ be a strictly increasing function satisfying 
that there is a $\lambda_*\in(0,\infty),k\in(-1,\infty)$ and an analytic function $h$ in the strip $0<\mathrm{Re}~\lambda\leq\lambda_*$
such that 
    \begin{align*}
        \int_{-\infty}^0e^{-\lambda s}f(s)ds=\frac{h(\lambda)}{(\lambda_*-\lambda)^{k+1}}\mbox{ for }0<\mathrm{Re}~\lambda<\lambda_*.
    \end{align*}
    Then we have 
    \begin{align*}
        \lim\limits_{s\to-\infty}\frac{f(s)}{e^{\lambda_* s}|s|^k}=\frac{h(\lambda_*)}{\Gamma(k+1)}.
    \end{align*}
\end{lemma}
The following lemma (see \cite[Chapter 5, Lemma 4.1]{volpert1994traveling}) will also be needed:

\begin{lemma}\label{lmavolpert}
    Let $A$ be an essentially nonnegative matrix and let $z\in\mathbb{C}^n.$ If 
   all eigenvalues of $A$ have negative real parts
    and $\mathrm{Re}~z_k\leq 0$ for each $k\in[n],$ then
    all eigenvalues of $A+\mathrm{diag}(z)$ have negative real parts.
\end{lemma}

We are ready to prove Theorem~\ref{thm2}.
\begin{proof}[Proof of Theorem~\ref{thm2}]
We still denote $c^*$ as in \eqref{thm1c*}.
Before beginning the proof, we do some preparation.
First, in view of Lemma \ref{lmalambda1} (b) and Remark~\ref{rk:analytic}, $\Lambda_1(H_{\lambda,c}+D_vg(E_0))=0$ has a smallest positive root $\lambda=\lambda_c$ with multiplicity $p_c+1$ whenever $c\geq c^*$, where $p_c=0$ if $c>c^*$; $p_c=1$ if $c=c^*$.
Also, by Lemma \ref{lmalambda1} (a) (taking $D=\begin{bmatrix}
    d_0
\end{bmatrix},P=\begin{bmatrix}
    -\delta
\end{bmatrix}$), $\Delta_c(\lambda):=c\lambda-d_0m(\lambda)+\delta=0$ has a unique positive real simple root $\nu_c>0$. Furthermore, Lemma \ref{lmalambda1} (c) asserts that $\Delta_c(\lambda)>0$ for all $\lambda\in(0,\nu_c)$, and thus $\Delta_c'(\nu_c)<0.$

    Since $g_i(U,0_n)=0$, Taylor's theorem gives
     \begin{align*}
         g_i(U,V)(s)=\sum_{j=1}^n\frac{\partial g_i}{\partial v_j}(U(s),0_n)V_j(s)+\frac{1}{2}\sum_{j,k=1}^n\frac{\partial^2 g_i}{\partial v_k\partial v_j}(P_{0,i}(s))V_j(s)V_k(s),
     \end{align*}
     where $P_{0,i}(s)=(U(s),\xi_{0,i}(s)V(s))$ for some $\xi_{0,i}(s)\in[0,1].$
     Furthermore, for each $j\in[n],$
     \begin{align*}
         \frac{\partial g_i}{\partial v_j}(U(s),0_n)=\frac{\partial g_i}{\partial v_j}(E_0)+\frac{\partial^2 g_i}{\partial u\partial v_j}(P_{ij}(s))(U(s)-K),
     \end{align*}
     where $P_{ij}(s)=(K+\xi_{ij}(s)(U(s)-K),0_n)$ for some $\xi_{ij}(s)\in[0,1].$
     Now, let
     \begin{align*}
         G_i(s)=\sum_{j=1}^n\frac{\partial^2 g_i}{\partial u\partial v_j}(P_{ij}(s))(U(s)-K)V_j(s)+\frac{1}{2}\sum_{j,k=1}^n\frac{\partial^2 g_i}{\partial v_k\partial v_j}(P_{0,i}(s))V_j(s)V_k(s)
     \end{align*}
     and $G(s):=(G_1(s),\dots,G_n(s)).$ Then
     \begin{align}\label{expandgatuv}
         g_i(U,V)(s)=\sum_{j=1}^n\frac{\partial g_i}{\partial v_j}(E_0)V_j(s)+ G_i(s)=\sum_{j=1}^n\Big(K(g_{ij}^0)'(0)+m_{ij}\Big)V_j(s)+ G_i(s).
     \end{align}
Since $U,V\in L^{\infty}(\mathbb{R})$ and $g_{ij}^0\in \mathcal{C}^2(\mathbb{R})$, there exists $C_1>0$ such that
\begin{align}\label{G-est}
|G(s)|_1\leq C_1\left(|K-U(s)|\cdot|V(s)|_1+|V(s)|_1^2\right),
\quad s\in\mathbb R.
\end{align}
By Lemma \ref{lem:for-Laplace} with small $\varepsilon\in(0,\min\{\frac{\lambda_c}{2},\mu_0\})$ and \eqref{G-est},
we obtain
\begin{align}\label{G-bound}
|G(s)|_1=O(e^{\widehat\mu_1s})
\quad\text{as }s\to-\infty,
\end{align}
where
\[
\widehat\mu_1=
\min\left\{
2(\lambda_c-\varepsilon),
\lambda_c-\varepsilon+\mu_0
\right\}
>\lambda_c.
\]

Thanks to Lemma~\ref{lem:for-Laplace}, we can define the bilateral Laplace transform of $V$
$$\mathcal{L}(V)(\lambda):=\int_{\mathbb{R}}e^{-\lambda s}V(s)ds,\quad 0<\mathrm{Re}~\lambda<\lambda_c.$$
     Also, denote
     \begin{align*}
         \mathcal{L}_+(V)(\lambda):=\int_0^\infty e^{-\lambda s}V(s)ds,\quad \mathcal{L}_-(V)(\lambda):=\int_{-\infty}^0e^{-\lambda s}V(s)ds,\quad \mathcal{L}_G(\lambda):=-\int_{\mathbb{R}}e^{-\lambda s}G(s)ds.
     \end{align*}
Then $\mathcal{L}_+(V)$ is analytic in $\mathrm{Re}~\lambda>0$, and 
$\mathcal{L}_G(\lambda)$ is analytic in $0<\mathrm{Re}~\lambda<\widehat{\mu}_1$, by \eqref{G-bound}.
     
     Applying the bilateral Laplace transform on both sides of the equation satisfied by $V,$ we obtain
     \begin{align}\label{laplacetran}
         Q_c(\lambda)\Big(\mathcal{L}_+(V)(\lambda)+\mathcal{L}_-(V)(\lambda)\Big)=\mathcal{L}_G(\lambda)
     \end{align}
     once they are defined, where 
      \[
     Q_c(\lambda):=H_{\lambda,c}+P,\quad P:=D_vg(E_0).
     \]
     Define $\mathbb{D}(\lambda):=\det(Q_c(\lambda))$.
      Then for $0<\mathrm{Re}~\lambda<\lambda_c$, if $Q_c(\lambda)^{-1}$ exists, then \eqref{laplacetran} gives
      \begin{align}\label{separate}
           \mathcal{L}_-(V)(\lambda)=Q_c(\lambda)^{-1}\mathcal{L}_G(\lambda)-\mathcal{L}_+(V)(\lambda).
      \end{align}

     
     The remainder of this proof is cut into three steps.
     
{\bf Step 1}. {\it The function $h:\{\lambda\in\mathbb{C}:0<\mathrm{Re}~\lambda<\lambda_c\}\to\mathbb{C}^n$ by
     \begin{align*}
         h(\lambda):=(\lambda_c-\lambda)^{p_c+1}\mathcal{L}_-(V)(\lambda).
     \end{align*}
     can be analytically extended to $\{\lambda\in\mathbb{C}:0<\mathrm{Re}~\lambda\leq\lambda_c\}.$\\
     }
     Note that the function $\mathbb{D}(\lambda)=\det(Q_c(\lambda))$
     is an analytic nonconstant function in $0<\mathrm{Re}~\lambda<\infty$. Then the zero locus $Z:=\{\lambda\in\mathbb{C}:0<\mathrm{Re}~\lambda<\infty\mbox{ and }\mathbb{D}(\lambda)=0\}$ is discrete. 
    Moreover, by the preparation above, $\lambda_c\in Z$ is a zero of
$\mathbb{D}$ of exact order $p_c+1$.
     Then for all $\lambda\in\{\lambda\in\mathbb{C}:0<\mathrm{Re}~\lambda<\lambda_c\}\setminus Z$, by \eqref{separate},
     \begin{align*}
         h(\lambda)&=(\lambda_c-\lambda)^{p_c+1}\left[(Q_c(\lambda))^{-1}\mathcal{L}_G(\lambda)-\mathcal{L}_+(V)(\lambda)\right]\\&=\frac{(\lambda_c-\lambda)^{p_c+1}}{\mathbb{D}(\lambda)}\mathrm{adj}(Q_c(\lambda))\mathcal{L}_G(\lambda)-(\lambda_c-\lambda)^{p_c+1}\mathcal{L}_+(V)(\lambda),
     \end{align*}
     where $\mathrm{adj}(Q_c(\lambda))$ denotes the adjugate matrix of $Q_c(\lambda).$
     Thus, this $h$ can be analytically extended to $\{\lambda\in\mathbb{C}:0<\mathrm{Re}~\lambda\leq\lambda_c\}$ if
     the function
     $$g(\lambda):=\frac{(\lambda_c-\lambda)^{p_c+1}}{\mathbb{D}(\lambda)}$$ has no pole in $\{\lambda\in\mathbb{C}:\mathrm{Re}~\lambda=\lambda_c\}.$ 
     To this end, fix a $\theta\in\mathbb{R}\setminus\{0\}.$ Then
     \begin{align*}
         Q_c(\lambda_c+\hat{\imath}\theta)
         &=Q_c(\lambda_c)
         +\bigg(\int_{\mathbb{R}}J(s)e^{-\lambda_cs}(e^{-\hat{\imath}\theta s}-1)ds\bigg)D-\hat{\imath}\theta cI\\
         &=Q_c(\lambda_c)
         +\widetilde{\theta}_{\mathrm{Re}}D+\hat{\imath}(\widetilde{\theta}_{\mathrm{Im}}D-\theta cI),
     \end{align*}
     where
     $$\widetilde{\theta}_{\mathrm{Re}}=\int_{\mathbb{R}}J(s)e^{-\lambda_cs}(\cos(\theta s)-1)ds<0,\quad\widetilde{\theta}_{\mathrm{Im}}=-\int_{\mathbb{R}}J(s)e^{-\lambda_cs}\sin(\theta s)ds.$$
     Rewrite
     \begin{align*}
         Q_c(\lambda_c+\hat{\imath}\theta)
         =\left(Q_c(\lambda_c)
         +\dfrac{1}{2}\widetilde{\theta}_{\mathrm{Re}}D\right)+\left(\dfrac{1}{2}\widetilde{\theta}_{\mathrm{Re}}D+\hat{\imath}(\widetilde{\theta}_{\mathrm{Im}}D-\theta cI)\right).
     \end{align*}
     Also note that $$\Lambda_1\left(Q_c(\lambda_c)+\dfrac{1}{2}\widetilde{\theta}_{\mathrm{Re}}D\right)\leq\Lambda_1(Q_c(\lambda_c))+\dfrac{1}{2}\widetilde{\theta}_{\mathrm{Re}}d_{\min}=\dfrac{1}{2}\widetilde{\theta}_{\mathrm{Re}}d_{\min}<0.$$ 
     By Lemma \ref{lmavolpert}, all eigenvalues of $Q_c(\lambda_c+\hat{\imath}\theta)$ have negative real parts. In particular, 
     $\mathbb{D}(\lambda_c+\hat{\imath}\theta)\neq 0$ for all $\theta\neq0$.
     Hence, $g$ has no pole in $\{\lambda\in\mathbb{C}:\mathrm{Re}~\lambda=\lambda_c,~\mathrm{Im}~\lambda\neq 0\}.$ Moreover, 
     since $\lambda_c$ is a zero of $\mathbb{D}$ of exact order $p_c+1$,
     we see that $\lambda_c$ is a removable singularity of $g$, which completes the proof of that $g$ has no pole in $\{\lambda\in\mathbb{C}:\mathrm{Re}~\lambda=\lambda_c\}$. Consequently, $h:\{\lambda\in\mathbb{C}:0<\mathrm{Re}~\lambda\leq\lambda_c\}\to\mathbb{C}^n$ is an analytic function such that
     \begin{align}\label{L_V1}
     \mathcal{L}_-(V)(\lambda)=\dfrac{h(\lambda)}{(\lambda_c-\lambda)^{p_c+1}}\mbox{ for }0<\mathrm{Re}~\lambda<\lambda_c.
     \end{align}
This completes Step 1.

{\bf Step 2}. {\it $h(\lambda_c)>0_n$ and thus the conclusion of part (a) is asserted.} 
    

Since $\Lambda_1(Q_c(\lambda_c))=0$
is an algebraically simple and isolated eigenvalue of $Q_c(\lambda_c)$, 
    we can choose $\Gamma$ to be a circle centered at $\Lambda_1(Q_c(\lambda_c))=0$, lying entirely within the resolvent set of $Q_c(\lambda_c)$ and enclosing no other eigenvalues of $Q_c(\lambda_c).$
    We consider the Riesz projector (see Kato \cite[Chapter 1]{kato2013perturbation})
    \begin{align*}
    \mathcal{P}(\lambda):=\frac{1}{2\pi \hat{\imath}}\oint_{\Gamma}(\zeta I_n-Q_c(\lambda))^{-1}\,d\zeta\mbox{ for }\lambda\mbox{ near }\lambda_c.
     \end{align*} 
   Since $0$ is an algebraically simple eigenvalue of $Q_c(\lambda_c)$, we have ${\rm rank}(\mathcal P(\lambda_c))=1$. By perturbation theory, ${\rm rank}(\mathcal P(\lambda))=1$ for all $\lambda$ near $\lambda_c$.
   Furthermore, since $\mathcal P(\lambda)$ commutes with $Q_c(\lambda)$ and has rank $1$, the range ${\rm R}(\mathcal P(\lambda))$ is a one-dimensional $Q_c(\lambda)$-invariant subspace.
   Therefore, there exists a unique $\kappa_c(\lambda)\in \mathbb{C}$ such that 
   \[
   Q_c(\lambda)\mathcal P(\lambda) =\kappa_c(\lambda)\mathcal P(\lambda).
\]
   Since ${\rm tr}(\mathcal P(\lambda))=1$, we have
\[
\kappa_c(\lambda)
=
{\rm tr}\big(Q_c(\lambda)\mathcal P(\lambda)
\big),
\]
and hence $\kappa_c(\cdot)$ is analytic for $\lambda$ near $\lambda_c$.

For $\lambda\in\mathbb{R}$ sufficiently close to $\lambda_c$,
$Q_c(\lambda)$ is an irreducible essentially nonnegative matrix, and thus we have
\[
\kappa_c(\lambda)=\Lambda_1(Q_c(\lambda))
=\mu(\lambda)-c\lambda,
\]
where
\begin{align}\label{defofmu}
    \mu(\lambda)=\Lambda_1(m(\lambda)D+P).
\end{align}

Let $\psi_c$ and $\varphi_c^\top$ be the right and left positive eigenvectors of $Q_c(\lambda_c)$ corresponding to
the eigenvalue $0$, respectively. We normalize $\varphi_c^\top\psi_c=1$.

Set
$\widetilde\psi(\lambda):=
\mathcal P(\lambda)\psi_c$
and  $\widetilde\varphi(\lambda)^\top
:=\varphi_c^\top\mathcal P(\lambda)$.
Since
$\widetilde\psi(\lambda_c)=\psi_c$
and 
$\widetilde\varphi(\lambda_c)^\top=\varphi_c^\top$,
we have 
\[
q(\lambda)
:=
\widetilde\varphi(\lambda)^\top
\widetilde\psi(\lambda)\neq0
\quad \mbox{ for $\lambda$ near $\lambda_c$.}
\]
Define
\[
\psi(\lambda):=\widetilde\psi(\lambda), \quad
\varphi(\lambda)^\top:=\frac{\widetilde\varphi(\lambda)^\top}{q(\lambda)}.
\]
Then we obtain
$Q_c(\lambda)\psi(\lambda)=
\kappa_c(\lambda)\psi(\lambda)$, 
$\varphi(\lambda)^\top Q_c(\lambda)=
\kappa_c(\lambda)\varphi(\lambda)^\top$
and
$\varphi(\lambda)^\top\psi(\lambda)=1$.
Since $\mathcal P(\lambda)$ has rank one, it follows that
\begin{align}\label{P-form}
\mathcal P(\lambda)
=
\psi(\lambda)\varphi(\lambda)^\top\quad \mbox{for all $\lambda$ near $\lambda_c$.}   
\end{align}
Furthermore, both $\psi(\cdot)$ and $\varphi(\cdot)$ are analytic near $\lambda_c.$

    By spectral decomposition, we 
    have
    \begin{align*}
    Q_c(\lambda)^{-1}=\frac{1}{\kappa_c (\lambda)}\mathcal{P}(\lambda)+M(\lambda),
    \end{align*}
    where $M$ is analytic
    for all $\lambda$ near $\lambda_c$.  Then from \eqref{separate} and \eqref{P-form},  we have 
\begin{align}\label{L_V2}
\mathcal{L}_-(V)(\lambda)=\frac{1}{\kappa_c(\lambda)} \psi(\lambda)\big(\varphi(\lambda)^\top \mathcal{L}_G(\lambda)\big)+\Tilde{M}(\lambda)
\end{align}
where $\Tilde{M}(\lambda)=M(\lambda)\mathcal{L}_G(\lambda)-\mathcal{L}_+(V)(\lambda)$ is analytic  
for all $\lambda$ near $\lambda_c$. Combining \eqref{L_V1} and \eqref{L_V2}, we have
\begin{align}\label{h-form}
    h(\lambda)=\frac{(\lambda_c-\lambda)^{p_c+1}}{\kappa_c(\lambda)} \psi(\lambda)\big(\varphi(\lambda)^\top \mathcal{L}_G(\lambda)\big)+(\lambda_c-\lambda)^{p_c+1}\Tilde{M}(\lambda)\mbox{ for all }\lambda\in(0,\lambda_c).
\end{align}
Note that since $\Tilde{M}(\cdot)$ is analytic,
$(\lambda_c-\lambda)^{p_c+1}\Tilde{M}(\lambda)\to 0_n$ as $\lambda\to\lambda_c$. We will focus on the first term in \eqref{h-form}.

Let us claim that 
\begin{align}\label{LG>0}
\mathcal L_G(\lambda_c)
\geq 0_n,\quad  \mathcal L_G(\lambda_c)\neq 0_n.   
\end{align}
To see this, from the definition of $G_i$ and assumption (C4), we see that $G_{i}(s)\leq 0$ for all $i$ and $s\in\mathbb{R}$. This implies $ \mathcal L_G(\lambda_c)\geq 0_n$. 
In addition, by (G2), we have 
$(g_{i_0j_0}^0)'(0)>0$ for some $i_0,j_0\in[n]$. 
Next, we show that $U\not\equiv K$. Indeed, if $U\equiv K$, by the equation of $U$, we obtain $g_0(K,V)\equiv 0$,
which is impossible because $V>0_n$ and
$g_{i_0j_0}^0(V_{j_0})>0$.
Therefore, there exists $s_0\in\mathbb{R}$ such that $U(s_0)<K$ and thus
\[
G_{i_0}(s_0)\leq (U(s_0)-K)(g_{i_0j_0}^0)'(0) V_{j_0}(s_0)<0,
\]
which yields that \eqref{LG>0}.
Furthermore,
since the left Perron eigenvector $\varphi_c^\top:=\varphi(\lambda_c)^\top$ 
is positive, together with \eqref{LG>0}, we conclude that 
\begin{align}\label{phiLG>0}
\varphi_c^\top\mathcal L_G(\lambda_c)=
-\int_{\mathbb R}e^{-\lambda_cs}\varphi_c^\top G(s)ds>0.
\end{align}
Moreover, $\psi(\lambda)$ satisfies 
\[
(m(\lambda)D+P)\psi(\lambda)=\mu(\lambda) \psi(\lambda),
\]
where $\mu(\cdot)$ is defined as in \eqref{defofmu}.
By differentiating in $\lambda\in\mathbb{R}$, one has 
\begin{align}\label{diff-1}
m'(\lambda)D\psi(\lambda)+(m(\lambda)D+P)\psi'(\lambda)=\mu'(\lambda) \psi(\lambda)+\mu(\lambda) \psi'(\lambda).
\end{align}
Next, multiplying both sides of \eqref{diff-1} by $\varphi(\lambda)^\top$ on the left and using $\varphi(\lambda)^\top (m(\lambda)D+P)=\mu(\lambda)\varphi(\lambda)^\top $ and $\varphi(\lambda)^\top \psi(\lambda) =1$, 
\begin{align*}
m'(\lambda)\varphi(\lambda)^\top D\psi(\lambda)+\mu(\lambda)\varphi(\lambda)^\top\psi'(\lambda)=\mu'(\lambda)+\mu(\lambda) \varphi(\lambda)^\top \psi'(\lambda),
\end{align*}
which implies that
\begin{align*}
\mu'(\lambda)=m'(\lambda)\varphi(\lambda)^\top D\psi(\lambda).
\end{align*}
Using $J(s)=J(-s),$ we have
\begin{align*}
    m'(\lambda)=&-\int_{\mathbb{R}}J(s)se^{-\lambda s}ds
    =-\left(\int_{-\infty}^0J(s)se^{-\lambda s}ds+\int_{0}^\infty J(s)se^{-\lambda s}ds\right)\\
    =&\int_0^\infty sJ(s)(e^{\lambda s}-e^{-\lambda s})ds>0\mbox{ for all }\lambda>0.
\end{align*}
Together with $\varphi(\lambda)^\top D \psi(\lambda)>0$, we have 
$$\mu'(\lambda)>0.$$

Recall $\psi_c=\psi(\lambda_c)$, $\varphi_c^\top=\varphi(\lambda_c)^\top,$ and $\kappa_c(\lambda)=\mu(\lambda)-c\lambda.$
We now divide our discussion into two cases:
\begin{itemize}
    \item[(i)] $c>c^*$ ($p_c=0$). In this case, $\kappa_c(\lambda_c)=0$ and $\kappa'(\lambda_c)=\mu'(\lambda_c)-c<0$,
    where
    the strict inequality follows from the fact that $\lambda_c$ is the smaller positive root of the strictly convex function $\lambda\mapsto\mu(\lambda)-c\lambda$ (Lemma~\ref{lmalambda1}). 
    Hence,  
    \begin{align*}
    \frac{\lambda_c-\lambda}{\kappa_c(\lambda)}=\frac{\lambda_c-\lambda}{\kappa'(\lambda_c)(\lambda-\lambda_c)+o(|\lambda-\lambda_c|)}\to \frac{1}{c-\mu'(\lambda_c)}>0\quad \text{ as } \lambda\to\lambda_c.
    \end{align*}
    Moreover, note that 
    \begin{align*}
    \psi(\lambda)\to \psi_c,\quad \varphi(\lambda)^\top\mathcal{L}_G(\lambda) \to \varphi_c^\top\mathcal{L}_G(\lambda_c) \quad \text{ as } \lambda\to\lambda_c.
    \end{align*}
    From \eqref{h-form} we obtain that $h(\lambda_c)=A_c \psi_c$, where 
    \begin{align*}
     A_c=\frac{\varphi_c^\top\mathcal{L}_G(\lambda_c)}{c-\mu'(\lambda_c)}>0
    \end{align*}
    and we have used \eqref{LG>0}
    and $c-\mu'(\lambda_c)>0$. 

      \item[(ii)] $c=c^*$ ($p_c=1$). In this case, $\kappa_c(\lambda_c)=\kappa'(\lambda_c)=0< \kappa''(\lambda_c)=\mu''(\lambda_c)$. Hence, 
    \begin{align*}
    \frac{(\lambda_c-\lambda)^2}{\kappa_c(\lambda)}=\frac{(\lambda_c-\lambda)^2}{(\kappa''(\lambda_c)/2)(\lambda_c-\lambda)^2+o(|\lambda-\lambda_c|^2)}\to \frac{2}{\mu''(\lambda_c)}>0 \quad \text{ as } \lambda\to\lambda_c.
    \end{align*}
    This implies that $h(\lambda_c)=A_c \psi_c$, where 
    \begin{align*}
     A_c=\frac{2\varphi_c^\top\mathcal{L}_G(\lambda_c)}{\mu''(\lambda_c)}>0
    \end{align*}
   because of \eqref{LG>0} and the fact that
   $\mu''(\lambda_c)>0$ (Lemma~\ref{lmalambda1}). 
\end{itemize}
Combining (i) and (ii), we obtain from Ikehara's theorem that there is an $A_c>0$ such that
    \begin{align*}
     \lim\limits_{s\to-\infty}\dfrac{V(s)}{|s|^{p_c}e^{\lambda_cs}}=\frac{A_c\psi_c}{\Gamma(p_c+1)}>0_n
    \end{align*}
    when each $V_i$ is strictly increasing on $(-\infty,0)$.
    
    If some $V_i$ is not strictly increasing on $(-\infty,0)$, 
    then we may consider a translation $V(s)e^{\rho s},$ where $\rho>
\max_{i\in[n]}
\frac{d_i+|m_{ii}|}{c}$ is chosen such that $V_i'+\rho V_i>0$ on $\mathbb{R}$ for each $i\in[n]$. Then $V_i(s)e^{\rho s}$ is strictly increasing on $\mathbb{R}$. 
Next, we can apply the above argument to $V(s)e^{\rho s}$, instead of $V(s)$. More precisely, since 
    \[\mathcal L_-(V(s)e^{\rho s})(\lambda)=\mathcal L_-(V)(\lambda-\rho),\] 
    we see that $\mathcal L_-(V(s)e^{\rho s})$ is well-defined in $
0<{\rm Re}\lambda<\lambda_c+\rho$. Following the above process, one can derive
\[
\lim_{s\to-\infty}
\frac{V(s)e^{\rho s}}
{|s|^{p_c}e^{(\lambda_c+\rho)s}}
=
\frac{\tilde{A_c}\psi_c}{\Gamma(p_c+1)}>0_n
\]
for some $\tilde{A_c}>0$. This concludes the desired result for $V(s)$.

{\bf Step 3}. {\it Conclude the asymptotic behavior of $U$ in part (b).}

     Recall that $\int_{\mathbb{R}}e^{-\lambda s}V(s)ds$
is defined for $0<\mathrm{Re}~\lambda<\lambda_c.$ Let $$W(s)=K-U(s)\mbox{ for }s\in\mathbb{R}.$$  Likewise as before, we can assume that $W$ is strictly increasing on $(-\infty,0)$ 
in order to apply Ikehara's theorem.
Otherwise, we may consider 
$e^{\rho s} W(s)$ for some large $\rho$ so that $e^{\rho s} W(s)$ is strictly increasing.

By the equation satisfied by $U,$ we have
\begin{align*}
        cW'(s)=d_0\int_{\mathbb{R}}J(s')W(s-s')ds'-d_0W(s)-\delta W(s)+(K-W(s))\sum_{i,j=1}^n g_{ij}^0(V_j(s))\mbox{ for }s\in\mathbb{R}.
    \end{align*}

    Furthermore, denote
    \begin{align*}
        \mathcal{L}_+(W)(\lambda):=\int_0^\infty e^{-\lambda s}W(s)ds,\quad\mathcal{L}_-(W)(\lambda):=\int_{-\infty}^0 e^{-\lambda s}W(s)ds.
    \end{align*}
Clearly, $\mathcal{L}_+(W)$ is analytic in $\mathrm{Re}~\lambda>0$.
By Lemma~\ref{lem:for-Laplace}, $\mathcal{L}_-(W)$ is analytic in $0<\mathrm{Re}~\lambda<\alpha_c,$ where $\alpha_c=\min\{\nu_c,\lambda_c\}$. 
    
    Using the bilateral Laplace transform and the relation \eqref{expandgatuv}, we have
\begin{align}\label{LaplaceU}
         \Delta_c(\lambda)\left(\mathcal{L}_+(W)(\lambda)+\mathcal{L}_-(W)(\lambda)\right)
         =I_1(\lambda)+I_2(\lambda)
     \end{align}
     as long as both $\mathcal{L}_-(W)(\lambda)$ and $\mathcal{L}_-(V)(\lambda)$ are defined,
     where
     \begin{align*}
         I_1(\lambda):=\sum_{i,j=1}^nK(g_{ij}^0)'(0)\int_{\mathbb{R}}e^{-\lambda s}V_j(s)ds,\quad I_2(\lambda):=\sum_{i=1}^n\int_{\mathbb{R}}e^{-\lambda s}G_i(s)ds.
     \end{align*} 
     
     Also, recall from 
    Steps 1 and 2 that there exists a function 
     $h_V$  which has an analytic continuation across
$\mathrm{Re}~\lambda=\lambda_c$,
    with $h_V(\lambda_c)=A_c\psi_c>0_n$ such that
\[
\mathcal L_-(V)(\lambda)= \frac{h_V(\lambda)}
     {(\lambda_c-\lambda)^{p_c+1}},
\quad
0<\mathrm{Re}~\lambda<\lambda_c,
\]
It follows that
$$I_1(\lambda)=\frac{h_1(\lambda)}{(\lambda_c-\lambda)^{p_c+1}}\mbox{ for }0<\mathrm{Re}~\lambda<\lambda_c$$
for some function $h_1$ which has an analytic continuation across
$\mathrm{Re}~\lambda=\lambda_c$, with
$h_1(\lambda_c)>0$.
      Furthermore, as we have done in the beginning of the proof, we can verify that $I_2(\lambda)$ is defined and analytic in $0<\mathrm{Re}~\lambda<\min\{2\lambda_c,\lambda_c+\mu_0\}.$ Then by \eqref{LaplaceU},
      \begin{align}\label{Ilambdapole}
          \Delta_c(\lambda)\mathcal{L}_-(W)(\lambda)=\frac{\widehat{h}(\lambda)}{(\lambda_c-\lambda)^{p_c+1}}\mbox{ for }0<\mathrm{Re}~\lambda<\lambda_c,
      \end{align}
      where $\widehat{h}(\lambda):=h_1(\lambda)+(\lambda_c-\lambda)^{p_c+1}(I_2(\lambda)-\Delta_c(\lambda)\mathcal{L}_+(W)(\lambda))$ is an analytic function defined on $0<\mathrm{Re}~\lambda\leq\lambda_c$  with $\widehat{h}(\lambda_c)= h_1(\lambda_c)>0$. 

     Since $\Delta_c(\lambda)=0$ is a scalar equation, one can follow the argument in \cite[p.2437]{carr2004uniqueness} to conclude that  $\Delta_c$ has no complex zero in
 $\{0<\mathrm{Re}~\lambda<\nu_c\}$ 
and that $\nu_c$ is its only zero on
$\mathrm{Re}~\lambda=\nu_c$.
    
 We now divide our discussion into three cases.
     \begin{enumerate}[(i)]
         \item $\nu_c>\lambda_c.$ In this case, 
         $\Delta_c(\lambda)\neq 0$ for $0<\mathrm{Re}~\lambda\leq\lambda_c$, 
          by \eqref{LaplaceU} and \eqref{Ilambdapole}, we have
     \[
       \mathcal{L}_-(W)(\lambda)=\frac{\widehat{h}(\lambda)/\Delta_c(\lambda)}{(\lambda_c-\lambda)^{p_c+1}}\mbox{ for }0<\mathrm{Re}~\lambda<\lambda_c,    
     \]
     where $\frac{\widehat{h}(\lambda)}{\Delta_c(\lambda)}$ is analytic in $0<\mathrm{Re}~\lambda\leq\lambda_c$ and is positive at $\lambda=\lambda_c$.
 Thus, we can apply Ikehara's Theorem to get
         \begin{align*}
             \lim_{s\to-\infty}\frac{W(s)}{|s|^{p_c}e^{\lambda_c s}}=B_c\quad \mbox{for some $B_c>0$.}
         \end{align*}
         \item $\nu_c=\lambda_c$. 
         Since $\lambda_c$ is a simple root of $\Delta_c(\lambda)=0$, we can express $\Delta_c(\lambda)$ as 
         \begin{align*}
            \Delta_c(\lambda)=(\lambda-\lambda_c)\Delta_c'(\lambda_c)+o(|\lambda-\lambda_c|)\quad\mbox{as }\lambda\to \lambda_c,\quad \Delta_c'(\lambda_c)< 0. 
         \end{align*}
         Hence, there exists an 
         analytic function $\widetilde{h}:\{\lambda\in\mathbb{C}:0<\mathrm{Re}~\lambda\leq\lambda_c\}\to\mathbb{C}$ with $\widetilde{h}(\lambda_c)>0$   such that
         $$\mathcal{L}_-(W)(\lambda)=\frac{\widetilde{h}(\lambda)}{(\lambda_c-\lambda)^{(p_c+1)+1}}\mbox{ for }0<\mathrm{Re}~\lambda<\lambda_c.$$ Then  Ikehara's theorem asserts that
         \begin{align*}
             \lim_{s\to-\infty}\frac{W(s)}{|s|^{p_c+1}e^{\lambda_c s}}=B_c\quad \mbox{for some $B_c>0$.}
         \end{align*}
         \item $\nu_c<\lambda_c.$
        Note that $\Delta_c(\lambda)\neq 0$ for $0<\mathrm{Re}~\lambda<\nu_c$ and $\Delta_c(\nu_c)= 0$. Since $\lambda=\nu_c$ is a simple root (and is the unique root in $\mathrm{Re}~\lambda=\nu_c$) with $\Delta_c'(\nu_c)< 0$, we have 
         \begin{align*}
             \mathcal{L}_-(W)(\lambda)=\frac{h_0(\lambda)}{\nu_c-\lambda}\mbox{ for }0<\mathrm{Re}~\lambda<\nu_c,
         \end{align*}
         where 
         \[
h_0(\lambda):=\frac{I(\lambda)(\nu_c-\lambda)}{\Delta_c(\lambda)}
-(\nu_c-\lambda)\mathcal L_+(W)(\lambda),\quad I(\lambda):=I_1(\lambda)+I_2(\lambda).
\]
Note that
$$\Delta_c(\lambda)=(\lambda-\nu_c)\Delta_c'(\nu_c)+o(|\lambda-\nu_c|)\quad\mbox{ as }\lambda\to\nu_c.$$
Then      
 $h_0(\lambda)$ is analytic in the strip $0<\mathrm{Re}~\lambda\leq\nu_c$ with $h_0(\nu_c)=-I(\nu_c)/\Delta_c'(\nu_c)>0.$
         By  Ikehara's theorem,
         \begin{align*}
             \lim_{s\to-\infty}\frac{W(s)}{e^{\nu_c s}}=B_c \quad \mbox{for some $B_c>0$.}
         \end{align*}
     \end{enumerate}
     From the above discussion,
    the proof of Theorem~\ref{thm2} is now complete.
\end{proof}

\appendix

\section{Appendix}
\setcounter{equation}{0}

\subsection{Proof of Lemma \ref{lmapositive}}\label{sec:proofoflmapositive}
\begin{proof}[Proof of Lemma \ref{lmapositive}]
For part (a), set $M:=\sup\limits_{s\in\mathbb{R}}U(s)$.
Since $U\in\mathcal{C}(\mathbb{R})\cap L^\infty(\mathbb{R})$, one can apply the fluctuation lemma (see, e.g., \cite[p.154]{smith2011introduction}) to conclude that there exists a sequence $\{s_k\}$ such that 
\[
U(s_k)\to M,\quad  U'(s_k)\to 0 \quad \text{ as }  k\to\infty. 
\]

Since $U\leq M$, $\int_{\R}J(s)ds=1$, and $g_0(U,V)\geq 0$, the equation of $U$ in \eqref{epidemictws} gives
\begin{align*}
       cU'(s_k)&=d_0\int_\mathbb{R} J(s_k-s')[U(s')-U(s_k)]ds'+\delta(K-U(s_k))-g_0(U(s_k),V(s_k))\\&\leq d_0(M-U(s_k))+\delta(K-U(s_k)).
    \end{align*}
Letting $k\to\infty$, we obtain $0\leq \delta(K-M)$ because $cU'(s_k)\to 0$. This implies $U(s)\leq M\leq K$ for all $s\in\mathbb{R}$ and the proof of part (a) is complete.




    
To prove part (b), we first deal with $U.$
    Let $$P=\delta+\sup_{s\in\mathbb{R}}\left(\sum_{i,j=1}^ng_{ij}^0(V_j(s))\right).$$ Then 
    $$cU'(s)\geq d_0\int_{\mathbb{R}}J(s-s')U(s')ds'-d_0U(s)-PU(s),\quad s\in\mathbb{R}.$$
    By Lemma \ref{SMPlinearized}, $U$ is either positive or identically zero on $\mathbb{R}.$ If $U\equiv 0$ on $\mathbb{R},$ then we have $\delta K=0,$ a contradiction. Thus, $U$ is positive on $\mathbb{R}.$

We next deal with $V$. 
    Note that for each $i\in[n],$ $V_i$ satisfies 
    \begin{align*}
        cV_i'(s)=  d_i\int_{\mathbb{R}}J(s-s')V_i(s')ds'-d_iV_i(s)+\sum_{j=1}^nQ_{ij}(s)V_j(s)\mbox{ for all }s\in\mathbb{R},
    \end{align*}
where $Q_{ij}(s)=U(s)G_{ij}(s)+m_{ij}$ and
\[
G_{ij}(s)=
\begin{cases}
\frac{g_{ij}^0(V_j(s))}{V_j(s)},\quad &  V_j(s)>0,\\
(g_{ij}^0)'(0),\quad &V_j(s)=0.
\end{cases}
\]
Clearly, $Q_{ij}\in\mathcal{C}(\mathbb{R})\cap L^\infty(\mathbb{R})$ for each $i,j\in[n].$
Recall from assumption (C3) that $g_{ij}^0(V_j(s))>0$ whenever $V_j(s)>0$ and $(g_{ij}^0)'(0)>0$. Then the matrix $\begin{bmatrix}
      Q_{ij}(s)
  \end{bmatrix}$ is irreducible for each $s\in\mathbb{R}$ since  $D_vg(E_0)=\begin{bmatrix}
           K(g_{ij}^0)'(0)+m_{ij}
       \end{bmatrix}$ is irreducible.
     By Lemma \ref{SMPlinearized}, we see that $V$ is either positive on $\mathbb{R}$ or identically zero on $\mathbb{R}$.
     

  

Now, suppose that $V\equiv 0_n$ on $\mathbb{R}.$ 
Set $m:=\inf\limits_{s\in\mathbb{R}}U(s)$. 
By the fluctuation lemma, there exists a sequence $\{t_k\}$ such that $U(t_k)\to m$ and $U'(t_k)\to 0$ as $k\to\infty$. 
Using $V\equiv 0_n$ and $g_0(U,0_n)\equiv 0$, we obtain
\begin{align*}
        cU'(t_k)&=d_0\int_\mathbb{R} J(t_k-s')[U(s')-U(t_k)]ds'+\delta(K-U(t_k))\\&\geq d_0(m-U(t_k))+\delta(K-U(t_k)).
    \end{align*}
Letting $k\to\infty$, we have $0\geq \delta(K-m)$. Together with $U\leq K$ from part (a), we obtain $U\equiv K$. This completes the proof of part (b).
\end{proof}

\subsection{Proof of Proposition \ref{lmacpositive}}\label{sec:proofoflmacpositive}
\begin{proof}[Proof of Proposition \ref{lmacpositive}]
    Denote 
    $\Lambda=\Lambda_1(D_vg(E_0))>0$ and $\kappa^\top$ as the left positive Perron eigenvector of $D_vg(E_0)$.
    Since $(U,V)(-\infty)=E_0,$ there is an $s^*<-1$ such that $\kappa^\top g(U,V)(s)>\dfrac{1}{2}\Lambda \kappa^\top V(s)$ for all $s\leq s^*$. This implies that 
\begin{align}\label{Vineq1}
         c\kappa^\top V'(s)>\kappa^\top D\int_\mathbb{R}J(s-s')(V(s')-V(s))ds'+\dfrac{1}{2}\Lambda \kappa^\top V(s)\mbox{ for all }s\leq s^*.
 \end{align}
We claim that for any $c\in\mathbb{R}$, $V$ is integrable on $(-\infty,s^*].$
To prove this, note that for $t<s\leq s^*,$
    \begin{align}\label{Vineq2}
        c\kappa^\top(V(s)-V(t))>\kappa^\top D\int_t^s\int_\mathbb{R}J(\tau-s')(V(s')-V(\tau))ds'd\tau+\dfrac{1}{2}\Lambda \kappa^\top \int_t^s V(\tau)d\tau.
    \end{align}
    Note that
    \begin{align*}
        &\int_t^s\int_\mathbb{R}J(\tau-s')(V(s')-V(\tau))ds'd\tau
        =\int_\mathbb{R}J(s')\int_t^s(V(\tau-s')-V(\tau))d\tau ds'\\
        =&\int_\mathbb{R}J(s')(-s')\int_t^s\int_0^1V'(\tau-\theta s')d\theta d\tau ds'
        =\int_\mathbb{R}J(s')(-s')\int_0^1\int_t^sV'(\tau-\theta s') d\tau d\theta ds'\\
        =&\int_\mathbb{R}J(s')(-s')\int_0^1(V(s-\theta s')-V(t-\theta s')) d\theta ds'.
    \end{align*}
    Thus,
    \begin{align*}
        \bigg|\int_t^s\int_\mathbb{R}J(\tau-s')(V(s')-V(\tau))ds'd\tau\bigg|_1\leq 2\sup_{s\in\mathbb{R}}|V(s)|_1\int_\mathbb{R}J(s')|s'|ds'=:I.
    \end{align*}
    Here, $I<\infty$ by (J1).
    Thanks to \eqref{Vineq2}, we have
    \begin{align}\label{Vineq3}
        c\kappa^\top(V(s)-V(t))>-I\theta+\dfrac{1}{2}\Lambda \kappa^\top \int_t^s V(\tau)d\tau,\quad t<s\leq s^*,
    \end{align}
    where $\theta$ is the maximal entry of $\kappa^\top D.$
    Taking $t\to-\infty$ in \eqref{Vineq3},
    since each component of $V$ is positive, we see that $V$ is integrable on $(-\infty,s^*],$ using $V(-\infty)=0_n.$

    Now, let $z(s):=\kappa^\top DV(s)$. Then for $s<s^*$,
    \begin{align*}
        &\int^{s}_{-\infty}\int_{\mathbb{R}}J(s')(z(\tau-s')-z(\tau))ds'd\tau
        =\int_{\mathbb{R}}J(s')\int_{-\infty}^{s}(z(\tau-s')-z(\tau))d\tau ds'\\
        =&\int_{\mathbb{R}}J(s')\bigg(\int_{-\infty}^{s-s'}z(\tau)d\tau-\int_{-\infty}^{s}z(\tau)d\tau\bigg)ds'
        =\int_{\mathbb{R}}J(s')\int_{s}^{s-s'}z(\tau)d\tau ds'\\
        =&\int_{-\infty}^0J(s')\int_{s}^{s-s'}z(\tau)d\tau ds'-\int^{\infty}_0J(s')\int_{s-s'}^{s}z(\tau)d\tau ds'\\
        =&\int^{\infty}_0J(s')\bigg(\int_{s}^{s+s'}z(\tau)d\tau-\int_{s-s'}^{s}z(\tau)d\tau\bigg)ds',
    \end{align*}
    where we have used (J2) in the last step.
    Hence,
    \begin{align*}
        c\kappa^\top V(s)>\int^{\infty}_0J(s')\bigg(\int_{s}^{s+s'}z(\tau)d\tau-\int_{s-s'}^{s}z(\tau)d\tau\bigg)ds'\mbox{ for any }s\leq s^*,
    \end{align*}
    and thus,
    \begin{align}\label{Vineq4}
        c\kappa^\top \int_{-\infty}^{s^*}V(s)ds>\int_{-\infty}^{s^*}\int^{\infty}_0J(s')\bigg(\int_{s}^{s+s'}z(\tau)d\tau-\int_{s-s'}^{s}z(\tau)d\tau\bigg)ds'ds.
    \end{align}
    After some elementary calculations using Fubini's theorem, we can see that
    the right-hand side of \eqref{Vineq4} is equal to
    \begin{align*}
        \int_0^\infty J(s')\bigg(\int_{s^*}^{s^*+s'}(s^*-\tau+s')z(\tau)d\tau +\int_{s^*-s'}^{s^*}(\tau-s^*+s')z(\tau)d\tau \bigg) ds',
    \end{align*}
    which is a positive number.
    Since $\kappa^\top \int_{-\infty}^{s^*}V(s)ds>0,$ we then deduce from \eqref{Vineq4} that $c>0.$
    This completes the proof of Proposition \ref{lmacpositive}.
\end{proof}

\subsection{Proof of Proposition \ref{superexistence}}\label{sec:proofc>c*}
To prove the existence of supercritical traveling waves, we first construct super- and sub-solutions.

By (H1), choose the constants $\underline{\lambda},\overline{\lambda}$ by Lemma \ref{lmalambda1} with $P=D_vg(E_0)$ and $\mathcal{D}=D$ satisfying $0<\underline{\lambda}<\overline{\lambda}<\infty$. 
Also, let $\kappa=(\kappa_1,\dots,\kappa_n)^\top$ be the right positive Perron eigenvector of $H_{\underline{\lambda},c}+D_vg(E_0),$ i.e., 
\begin{align*}
    (H_{\underline{\lambda},c}+D_vg(E_0))\kappa=0_n,
\end{align*}
where $H_{\lambda,c}$ is defined as in \eqref{defofmlambda}.
In view of assumption (G2), the quantity
$$\chi_i:=\sum_{j=1}^n\sup_{\xi\in[0,\infty)}g_{ij}^0(\xi)+1$$
is a finite positive constant for each $i\in[n].$
Then we can choose, by (H2), a positive vector $q=(q_1,\dots,q_n)^\top$ such that 
\begin{align}\label{Mq<chi}
    \mathbb{M}q<-K\chi^\top,
\end{align}
where $\chi=(\chi_1,\dots,\chi_n)$.
As we will see, a pair of super- and subsolutions can be given by
\begin{align*}
    \overline{U}(s)&\equiv K,\quad\overline{V}_i(s)=\min\{\kappa_i e^{\underline{\lambda}s},q_i\};\\
    \underline{U}(s)&=\max\{K-\sigma_0e^{\alpha s},0\},\quad\underline{V}_i(s)=\max\{\kappa_i e^{\underline{\lambda}s}(1-\sigma_ie^{\varepsilon s}),0\},\quad i\in[n],
\end{align*}
where $\alpha,\varepsilon,\sigma_i,i\in\{0\}\cup[n],$  are positive constants to be determined. Also, we denote $\overline{V}=(\overline{V}_1,\dots,\overline{V}_n)$ and 
    $\underline{V}=(\underline{V}_1,\dots,\underline{V}_n).$

\begin{lemma}\label{lmaoverlinev}
    For $i\in[n],$
    $$c\overline{V}_i'(s)\geq d_i\int_{\mathbb{R}}J(s-s')(\overline{V}_i(s')-\overline{V}_i(s))ds'+g_i(K,\overline{V}(s))\mbox{ for all }s\neq \overline{s}_i:=\frac{1}{\underline{\lambda}}\ln\frac{q_i}{\kappa_i}.$$ 
\end{lemma}
\begin{proof}
First assume that $s<\overline{s}_i.$
    A simple calculation yields
    \begin{align*}
        \int_{\mathbb{R}}J(s-s')(\overline{V}_i(s')-\overline{V}_i(s))ds' \leq \kappa_ie^{\underline{\lambda}s}\bigg(\int_{\mathbb{R}}J(s')e^{-\underline{\lambda}s'}ds'-1\bigg),
    \end{align*}
    where we used $\overline{V}_i(s')\leq \kappa_ie^{\underline{\lambda}s'} $.
    By Taylor's theorem, by assumptions (C2) and (C4),
    \begin{align*}
        g_i(K,\overline{V}(s))\leq\sum_{j=1}^n\dfrac{\partial g_i}{\partial v_j}(E_0)\overline{V}_j(s).
    \end{align*}
    Hence, for $i\in[n],$
    \begin{align*}
        &e^{-\underline{\lambda}s}\left[-c\overline{V}_i'(s)+d_i\int_{\mathbb{R}}J(s-s')(\overline{V}_i(s')-\overline{V}_i(s))ds'+g_i(K,\overline{V}(s))\right]\\
        \leq&-c\underline{\lambda}\kappa_i+d_i\kappa_i\bigg(\int_{\mathbb{R}}J(s')e^{-\underline{\lambda}s'}ds'-1\bigg)+\sum_{j=1}^n\dfrac{\partial g_i}{\partial v_j}(E_0)\kappa_j
        =((H_{\underline{\lambda},c}+D_vg(E_0))\kappa)_i=0.
    \end{align*}
    Now, assume that $s>\overline{s}_i.$ Then $\overline{V}_i(s)=q_i$. 
    Since $\overline{V}_j(s)\leq q_j$ for $j\neq i$, assumption (C2) implies that
    $$g_i(K,\overline{V}(s))\leq g_i(K,q)\mbox{ for all }s\in\mathbb{R}.$$
    Also, since $\overline V_i(s')\le q_i$ for all $s'\in\mathbb{R}$, we have
    $\int_{\mathbb{R}}J(s-s')
    (\overline V_i(s')-\overline V_i(s)) ds'\le0$.
    Therefore,
    \begin{align*}
        \int_{\mathbb{R}}J(s-s')(\overline{V}_i(s')-\overline{V}_i(s))ds'+g_i(K,\overline{V}(s))\leq g_i(K,q)\leq K\chi_i+\sum_{j=1}^n m_{ij}q_j\leq0=\overline{V}_i'(s),
    \end{align*}
    where we have used \eqref{Mq<chi} in the last inequality.
    This proves the desired inequality. 
\end{proof}

\begin{lemma}\label{underlineu}
    There exist $\sigma_0>0,\alpha\in(0,\min\{\overline{\lambda}-\underline{\lambda},\underline{\lambda}\})$ such that the function $\underline{U}(s)$ satisfies
    $$c\underline{U}'(s)\leq d_0\int_{\mathbb{R}}J(s-s')(\underline{U}(s')-\underline{U}(s))ds'+f(\underline{U}(s))-g_0(\underline{U}(s),\overline{V}(s))\mbox{ for all }s\neq\underline{s}_0:=\dfrac{1}{\alpha}\ln\dfrac{K}{\sigma_0}.$$    
\end{lemma}
\begin{proof}
    If $s>\underline{s}_0,$ then $\underline{U}(s)=0$ and so the inequality holds obviously. Suppose that $s<\underline{s}_0.$ Then
    \begin{align*}
        &d_0\int_{\mathbb{R}}J(s-s')(\underline{U}(s')-\underline{U}(s))ds'-c\underline{U}'(s)+f(\underline{U}(s))-g_0(\underline{U},\overline{V})(s)\\
        \geq&d_0\sigma_0\int_{\mathbb{R}}J(s-s')(e^{\alpha s}-e^{\alpha s'})ds'+c\sigma_0\alpha e^{\alpha s}-g_0(K,\overline{V}(s))\\
        \geq &d_0\sigma_0e^{\alpha s}\bigg(1-\int_{\mathbb{R}}J(s')e^{-\alpha s'}ds'\bigg)+c\sigma_0\alpha e^{\alpha s}-\sum_{j=1}^n\dfrac{\partial g_0}{\partial v_j}(E_0)\kappa_je^{\underline{\lambda}s}\\
        =&e^{\alpha s}\left[d_0\sigma_0\alpha\bigg(\frac{c}{d_0}-\int_{\mathbb{R}}J(s')\frac{e^{-\alpha s'}-1}{\alpha}ds'\bigg)-\sum_{j=1}^n\dfrac{\partial g_0}{\partial v_j}(E_0)\kappa_je^{(\underline{\lambda}-\alpha)s}\right].
    \end{align*}
    Note that 
    $$\lim\limits_{\alpha\to 0^+}\frac{e^{-\alpha s'}-1}{\alpha}=-s'\mbox{ for each }s'\in\mathbb{R}.$$
    By the Lebesgue dominated convergence theorem, and using (J1) and the symmetry condition (J2),
    we have
    $$\lim\limits_{\alpha\to 0^+}\int_{\mathbb{R}}J(s')\frac{e^{-\alpha s'}-1}{\alpha}ds'=\int_{\mathbb{R}}J(s')(-s')ds'=0.$$
    Therefore, we may choose an $\alpha\in(0,\min\{\overline{\lambda}-\underline{\lambda},\underline{\lambda}\})$ such that
    $$\frac{c}{d_0}-\int_{\mathbb{R}}J(s')\frac{e^{-\alpha s'}-1}{\alpha}ds'>0.$$
    After choosing $\alpha,$ we may choose a sufficiently large $\sigma_0>K$ such that
    \begin{align*}
        d_0\sigma_0\alpha\bigg(\frac{c}{d_0}-\int_{\mathbb{R}}J(s')\frac{e^{-\alpha s'}-1}{\alpha}ds'\bigg)>\sum_{j=1}^n\dfrac{\partial g_0}{\partial v_j}(E_0)\kappa_j.
    \end{align*}
    Also, for $s<\underline{s}_0,$ $e^{(\underline{\lambda}-\alpha)s}<e^{[(\underline{\lambda}-\alpha)/\alpha]\ln(K/\sigma_0)}=(\frac{K}{\sigma_0})^{(\underline{\lambda}-\alpha)/\alpha}<1.$
    Since $\dfrac{\partial g_0}{\partial v_j}(E_0)\geq 0$ (by assumption (C2)),
    \begin{align*}
        \sum_{j=1}^n\dfrac{\partial g_0}{\partial v_j}(E_0)\kappa_j\geq\sum_{j=1}^n\dfrac{\partial g_0}{\partial v_j}(E_0)\kappa_je^{(\underline{\lambda}-\alpha)s}\mbox{ for }s<\underline{s}_0.
    \end{align*}
    With such a choice of $\alpha$ and $\sigma_0,$ we obtain the desired inequality.
\end{proof}

\begin{lemma}\label{lmaunderlinev}
    Suppose that the constants $\alpha$ and $\sigma_0$ are chosen as in Lemma \ref{underlineu}. Then there are constants $\varepsilon>0$ and $\sigma_i>0,~i\in[n],$ such that for each $i\in[n],$ we have
    $$c\underline{V}_i'(s)\leq d_i\int_{\mathbb{R}}J(s-s')(\underline{V}_i(s')-\underline{V}_i(s))ds'+g_i(\underline{U}(s),\underline{V}(s))\mbox{ for all }s\neq \underline{s}_i:=-\dfrac{1}{\varepsilon}\ln\sigma_i.$$
\end{lemma}
\begin{proof} 
First, take 
$\varepsilon=\frac{1}{2}\min\{\alpha,\overline{\lambda}-\underline{\lambda},\underline{\lambda}\},$
    where
    $\overline{\lambda},\underline{\lambda}$ are determined by Lemma \ref{lmalambda1} with $P=D_vg(E_0).$
Fix an $i\in[n]$ and assume $s<\underline{s}_i,$ as the result is trivial when $s>\underline{s}_i$.
    Similarly as in \eqref{expandgatuv},
    \begin{align*}
        g_i(\underline{U}(s),\underline{V}(s))=&\sum_{j=1}^n\dfrac{\partial g_i}{\partial v_j}(E_0)\underline{V}_j(s)+G_i(s),
    \end{align*}
    where 
    \begin{align*}
        G_i(s):=\sum_{j=1}^n\dfrac{\partial^2 g_i}{\partial u\partial v_j}(P_{ij}(s))(\underline{U}(s)-K)\underline{V}_j(s)
        +\dfrac{1}{2}\sum_{j,k=1}^n\dfrac{\partial^2 g_i}{\partial v_k\partial v_j}(P_{0,i}(s))\underline{V}_j(s)\underline{V}_k(s)
    \end{align*}
    with $P_{0,i}(s)=(\underline{U}(s),\xi_{0,i}(s)\underline{V}(s))$, $P_{ij}(s)=(K+\xi_{ij}(s)(\underline{U}(s)-K),0_n)$ for some $\xi_{0,i}(s),\xi_{ij}(s)\in[0,1]$ for $j\in[n].$ Then
    \begin{align*}
        g_i(\underline{U}(s),\underline{V}(s))\geq&e^{\underline{\lambda}s}\bigg(\sum_{j=1}^n\dfrac{\partial g_i}{\partial v_j}(E_0)\kappa_j(1-\sigma_je^{\varepsilon s})+\sum_{j=1}^n\dfrac{\partial g_i}{\partial u\partial v_j}(P_{ij}(s))(\underline{U}(s)-K)\kappa_j(1-\sigma_je^{\varepsilon s})_+\\&+\frac{1}{2}e^{\underline{\lambda}s}\sum_{j,k=1}^n\dfrac{\partial g_i}{\partial v_k\partial v_j}(P_{0,i}(s))\kappa_j\kappa_k(1-\sigma_je^{\varepsilon s})_+(1-\sigma_ke^{\varepsilon s})_+\bigg).
    \end{align*}
    On the other hand, 
    \begin{align*}
        \int_{\mathbb{R}}J(s-s')(\underline{V}_i(s')-\underline{V}_i(s))ds'
        \geq
        &\kappa_i\int_{\mathbb{R}} J(s-s')[(e^{\underline{\lambda}s'}-\sigma_ie^{(\underline{\lambda}+\varepsilon)s'})-(e^{\underline{\lambda}s}-\sigma_ie^{(\underline{\lambda}+\varepsilon)s})]ds'\\
        =&\kappa_i\bigg[e^{\underline{\lambda}s}\int_{\mathbb{R}}J(s')(e^{-\underline{\lambda}s'}-1)ds'-\sigma_ie^{(\underline{\lambda}+\varepsilon)s}\int_{\mathbb{R}}J(s')(e^{-(\underline{\lambda}+\varepsilon)s'}-1)ds'\bigg].
    \end{align*}
    Let
    \begin{align*}
        K_i(s):=e^{-\underline{\lambda}s}\bigg(d_i\int_{\mathbb{R}}J(s-s')(\underline{V}_i(s')-\underline{V}_i(s))ds'-c\underline{V}_i'(s) +g_i(\underline{U}(s),\underline{V}(s))\bigg).
    \end{align*}
    Then
    \begin{align*}
        K_i(s)
        \geq&\kappa_id_i\bigg[\int_{\mathbb{R}}J(s')(e^{-\underline{\lambda}s'}-1)ds'-\sigma_ie^{\varepsilon s}\int_{\mathbb{R}}J(s')(e^{-(\underline{\lambda}+\varepsilon)s'}-1)ds'\bigg]+c\kappa_i\sigma_i(\underline{\lambda}+\varepsilon)e^{\varepsilon s}-c\kappa_i\underline{\lambda}\\
        &+\sum_{j=1}^n\dfrac{\partial g_i}{\partial v_j}(E_0)\kappa_j(1-\sigma_je^{\varepsilon s})+\sum_{j=1}^n\dfrac{\partial g_i}{\partial u\partial v_j}(P_{ij}(s))(\underline{U}(s)-K)\kappa_j(1-\sigma_je^{\varepsilon s})_+\\
        &+e^{\underline{\lambda}s}\sum_{j,k=1}^n\dfrac{\partial g_i}{\partial v_k\partial v_j}(P_{0,i}(s))\kappa_j\kappa_k(1-\sigma_je^{\varepsilon s})_+(1-\sigma_ke^{\varepsilon s})_+\\
        \geq&-e^{\varepsilon s}I_i+\sum_{j=1}^n\dfrac{\partial g_i}{\partial u\partial v_j}(P_{ij}(s))(\underline{U}(s)-K)\kappa_j(1-\sigma_je^{\varepsilon s})_+\\&+e^{\underline{\lambda}s}\sum_{j,k=1}^n\dfrac{\partial g_i}{\partial v_k\partial v_j}(P_{0,i}(s))\kappa_j\kappa_k(1-\sigma_je^{\varepsilon s})_+(1-\sigma_ke^{\varepsilon s})_+,
    \end{align*}
    where
    $$I_i=\kappa_id_i\sigma_i\int_{\mathbb{R}}J(s')(e^{-(\underline{\lambda}+\varepsilon)s'}-1)ds'-c\kappa_i\sigma_i(\underline{\lambda}+\varepsilon)+\sum_{j=1}^n\dfrac{\partial g_i}{\partial v_j}(E_0)\kappa_j\sigma_j,$$ and we used the fact in the last inequality that 
    $$\kappa_id_i\int_{\mathbb{R}}J(s')(e^{-\underline{\lambda}s'}-1)ds'-c\kappa_i\underline{\lambda}+\sum_{j=1}^n\dfrac{\partial g_i}{\partial v_j}(E_0)\kappa_j=0,$$
    asserted by Lemma \ref{lmalambda1}. 
    Since $\underline{\lambda}+\varepsilon\in(\underline{\lambda},\overline{\lambda}),$ we see that $\Lambda_1(H_{\underline{\lambda}+\varepsilon,c}+D_vg(E_0))<0$ by Lemma \ref{lmalambda1} (c). Let $\eta=(\eta_1,\dots,\eta_n)^\top$ be the unit positive eigenvector of $H_{\underline{\lambda}+\varepsilon,c}+D_vg(E_0)$ corresponding to $\Lambda_1(H_{\underline{\lambda}+\varepsilon,c}+D_vg(E_0)).$ Then 
for each $j\in[n],$ we have    
$$((H_{\underline{\lambda}+\varepsilon,c}+D_vg(E_0))\eta)_j=\Lambda_1(H_{\underline{\lambda}+\varepsilon,c}+D_vg(E_0))\eta_j<0.$$

Write $$R_{i1}(s):=\sum_{j=1}^n\dfrac{\partial g_i}{\partial u\partial v_j}(P_{ij}(s))\kappa_j(1-\sigma_je^{\varepsilon s})_+\mbox{ and }R_{i2}(s):=\sum_{j,k=1}^n\dfrac{\partial g_i}{\partial v_k\partial v_j}(P_{0,i}(s))\kappa_j\kappa_k(1-\sigma_je^{\varepsilon s})_+(1-\sigma_ke^{\varepsilon s})_+.$$
 Since $g$ is of $\mathcal{C}^2$, there is a constant $M>0$ independent of $s$ such that $|R_{i\ell}(s)|\leq M$ for all $s$ and all $i\in[n],\ell=1,2$. Denote $\ell_j:=-((H_{\underline{\lambda}+\varepsilon,c}+D_vg(E_0))\eta)_j>0$ for $j\in[n].$ Also, take
 $$\eta_0=\max\left\{\frac{(\sigma_0+1)M}{\min\{\ell_1,\dots,\ell_n\}},\frac{\max\{\kappa_1,\dots,\kappa_n\}}{\min\{\eta_1,\dots,\eta_n\}}\max\left\{\left(\frac{K}{\sigma_0}\right)^{-\varepsilon/\alpha},\max_{i\in[n]}\left(\frac{q_i}{\kappa_i}\right)^{-\varepsilon/\underline{\lambda}},2\right\}+1\right\}.$$
 Take $\sigma_j=\frac{\eta_0\eta_j}{\kappa_j}$ for $j\in[n].$ Then for each $j\in[n],$ we have
 \begin{enumerate}[(i)]
     \item $\underline{s}_j=-\frac{1}{\varepsilon}\ln\sigma_j=-\frac{1}{\varepsilon}\ln\frac{\eta_0\eta_j}{\kappa_j}<-\frac{1}{\varepsilon}\ln\left(\frac{K}{\sigma_0}\right)^{-\varepsilon/\alpha}=\frac{1}{\alpha}\ln\frac{K}{\sigma_0}=s_0;$
     \item $\underline{s}_j<-\frac{1}{\varepsilon}\ln\left(\frac{q_i}{\kappa_i}\right)^{-\varepsilon/\underline{\lambda}}=\overline{s}_j;$
     \item $\underline{s}_j<-\frac{1}{\varepsilon}\ln 2<0.$
 \end{enumerate}
 Thus, for each $s<\underline{s}_i,$ we have $\underline{U}(s)-K=-\sigma_0 e^{\alpha s}$ and hence,
 \begin{align*}
     e^{-\varepsilon s}K_i(s)\geq\eta_0\ell_i-\sigma_0e^{(\alpha-\varepsilon)s}R_{i1}(s)+e^{(\underline{\lambda}-\varepsilon)s}R_{i2}(s)\geq\eta_0\ell_i-\sigma_0M-M\geq0.
 \end{align*}
 Then $K_i(s)\geq 0$ for each $s<\underline{s}_i,$  and so we obtain the desired inequality.
 \end{proof}

For $a>\max\limits_{i\in\{0\}\cup[n]}\max\{|\underline{s}_i|,|\overline{s}_i|\},$ denote
\begin{align*}
    \Gamma_a=\{(U,V)\in\mathcal{C}([-a,a],\mathbb{R}^{n+1}):&(U,V)(-a)=(\underline{U},\underline{V})(-a),\underline{U}(s)\leq U(s)\leq K\\&\mbox{ and }\underline{V}_i(s)\leq V_i(s)\leq\overline{V}_i(s)\mbox{ for all }s\in[-a,a],i\in[n]\}.
\end{align*}
Given $(U^0,V^0)\in\Gamma_a,$ extend it into $\mathbb{R}$ by defining
\begin{align*}
    \widetilde{U^0}(s)=\begin{cases}
        U^0(a)\mbox{ if }s>a,\\
        U^0(s)\mbox{ if }-a\leq s\leq a,\\
        \underline{U}(s)\mbox{ if }s<-a;
    \end{cases}\quad \widetilde{V^0_i}(s)=\begin{cases}
        V_i^0(a)\mbox{ if }s>a,\\
         V^0_i(s)\mbox{ if }-a\leq s\leq a,\\
        \underline{V}_i(s)\mbox{ if }s<-a,
    \end{cases}
    \quad i\in[n].
\end{align*}
It is easy to verify that 
    \begin{align*}
        \underline{U}(s)\leq\widetilde{U^0}(s)\leq K\mbox{ and }\underline{V}_i(s)\leq\widetilde{V_i^0}(s)\leq\overline{V}_i(s)\mbox{ for all }s\in\mathbb{R}.
    \end{align*}
 Define $T:\Gamma_a\to\mathcal{C}([-a,a],\mathbb{R}^{n+1})$ by
$T(U^0,V^0)=(U,V),$ where $(U,V)\in\mathcal{C}^1([-a,a],\mathbb{R}^{n+1})$ is the unique solution of the initial value problem
\begin{align}\label{IVP}
\begin{cases}
    cU'(s)=d_0\dis\int_{\mathbb{R}}J(s-s')\widetilde{U^0}(s')ds'-d_0U(s)+f(U(s))-g_0(U(s),V^0(s))=:F_0(s,U(s)),\\
    cV_i'(s)=d_i\dis\int_{\mathbb{R}}J(s-s')\widetilde{V_i^0}(s')ds'-d_iV_i(s)+g_i^0(U^0(s),V^0(s))\\
    \qquad\qquad\qquad   +\sum_{j\in[n]\setminus\{i\}}m_{ij}V_j^0(s)+m_{ii}V_i(s)=:F_i(s,V_i(s)),\\
    (U,V)(-a)=(\underline{U},\underline{V})(-a).
\end{cases}
\end{align}  
The following lemma establishes the existence result in $[-a,a]$ with a $\mathcal{C}^{1,1}$ estimate.

\begin{lemma}\label{c1estimates}
For each $a>\max\limits_{i\in\{0\}\cup[n]}\max\{|\underline{s}_i|,|\overline{s}_i|\},$ there is a  positive $(U_a,V_a)\in\Gamma_a\cap\mathcal{C}^1([-a,a],\mathbb{R}^{n+1})$ such that $$T(U_a,V_a)=(U_a,V_a).$$
  Furthermore, there exists a constant $C>0$ independent of $a$ such that 
    $$\|(U_a,V_a)\|_{\mathcal{C}^{1,1}([-a,a],\mathbb{R}^{n+1})}\leq C\mbox{ for all }a>\max\limits_{i\in\{0\}\cup[n]}\max\{|\underline{s}_i|,|\overline{s}_i|\}.$$
\end{lemma}
\begin{proof}
Let $s^*:=\max\limits_{i\in\{0\}\cup[n]}\max\{|\underline{s}_i|,|\overline{s}_i|\}$ and fix $a>s^*.$
To apply Schauder's fixed point theorem, we need to show that $T$ is completely continuous and $T(\Gamma_a)\subseteq\Gamma_a$.

By applying the variation of constant formula to each equation in \eqref{IVP}, we may easily see that $T$ is Lipschitz continuous on $\Gamma_a$. Let $\{(U^0_k,V^0_k)\}$ be a uniformly bounded sequence in $\Gamma_a$ and $(U_k,V_k)=T(U^0_k,V^0_k).$ Then it is obvious that both $\{(U_k,V_k)\}$ and $\{(U_k',V_k')\}$ are uniformly bounded. This means that $\{(U_k,V_k)\}$ is equicontinuous on $[-a,a]$ and so the Arzel\'{a}-Ascoli theorem deduces that $\{(U_k,V_k)\}$ has a uniformly convergent subsequence on $[-a,a]$. This shows that $T$ is completely continuous.

    Next, we prove the inclusion. Let $(U^0,V^0)\in\Gamma_a$ and $(U,V)=T(U^0,V^0).$ 
    Recall $F_0$ from \eqref{IVP}.
    Clearly, $F_0(s,K)\leq 0$ and so $U(s)\leq K$ for all $s\in[-a,a].$ On the other hand, Lemma \ref{underlineu} asserts that for all $s\in[-a,\underline{s}_0),$ we have
    \begin{align*}
        c\underline{U}'(s)\leq d_0\int_{\mathbb{R}}J(s-s')(\underline{U}(s')-\underline{U}(s))ds'+f(\underline{U}(s))-g_0(\underline{U}(s),\overline{V}(s)).
    \end{align*}  
    Since $V^0_i\leq\overline{V}_i$ for each $i\in[n],$ we have $g_0(\underline{U},\overline{V})\geq g_0(\underline{U},V^0)$ by assumption (C2).
    This means that
    $c\underline{U}'(s)\leq F_0(s,\underline{U}(s))$ and hence $\underline{U}(s)\leq U(s)$ for all $s\in[-a,\underline{s}_0).$ However, $\underline{U}(s)=0$ if $s\in[\underline{s}_0,a]$ and so $c\underline{U}'(s)\leq F_0(s,\underline{U}(s))$ holds for $s\in[\underline{s}_0,a]$. So far we have proved that
    $$\underline{U}(s)\leq U(s)\leq\overline{U}(s)\mbox{ for all }s\in[-a,a].$$

    Let us deal with $V$. 
    By Lemma \ref{lmaoverlinev}, for all $s\in[-a,\overline{s}_i)$,
    \begin{align*}
        c\overline{V}_i'(s)&\geq d_i\int_{\mathbb{R}}J(s-s')(\overline{V}_i(s')-\overline{V}_i(s))ds'+g_i(K,\overline{V}(s))\notag\\
        &\geq d_i\int_{\mathbb{R}}J(s-s')(\widetilde{V^0_i}(s')-\overline{V}_i(s))ds'+g_i^0(U^0(s),V^0(s))+\sum_{j\in[n]\setminus\{i\}}m_{ij}V_j^0(s)+m_{ii}\overline{V}_i(s).
    \end{align*}
    That is, 
    \begin{align}\label{overlineVinequality}
        c\overline{V}_i'(s)\geq F_i(s,\overline{V}_i(s))\mbox{ for all }s\in[-a,\overline{s}_i).
    \end{align}
    By the comparison principle for ODEs, $V_i(s)\leq\overline{V}_i(s)$ for all $s\in[-a,\overline{s}_i].$ But \eqref{overlineVinequality} also holds for $s\in(\overline{s}_i,a],$ meaning that  $V_i(s)\leq\overline{V}_i(s)$ for all $s\in[-a,a].$
    Likewise, Lemma \ref{lmaunderlinev} asserts that for all $s\in[-a,\underline{s}_i),$ we have 
    \begin{align*}
        c\underline{V}_i'(s)&\leq d_i\int_{\mathbb{R}}J(s-s')(\underline{V}_i(s')-\underline{V}_i(s))ds'+g_i(\underline{U}(s),\underline{V}(s))\\
        &\leq d_i\int_{\mathbb{R}}J(s-s')(\widetilde{V_i^0}(s')-\underline{V}_i(s))ds'+g_i^0(U^0(s),V^0(s))+\sum_{j\in[n]\setminus\{i\}}m_{ij}V_j^0(s)+m_{ii}\underline{V}_i(s),
    \end{align*}
    which means that
        $c\underline{V}_i'(s)\leq F_i(s,\underline{V}_i(s))\mbox{ for all }s\in[-a,\underline{s}_i).$
    Thus, $\underline{V}_i(s)\leq V_i(s)$ for all $s\in[-a,\underline{s}_i].$ Finally, since $\underline{V}_i(s)=0$ if $s\in[\underline{s}_i,a],$ we may conclude that
$$\underline{V}_i(s)\leq V_i(s)\leq\overline{V}_i(s)\mbox{ for all }s\in[-a,a].$$
Therefore, we conclude that
$(U,V)\in\Gamma_a.$

 Clearly, $\Gamma_a$ is nonempty, closed and convex in $\mathcal{C}([-a,a],\mathbb{R}^{n+1})$ with respect to the uniform norm. Thus, Schauder's fixed point theorem asserts that there is an $(U_a,V_a)\in\Gamma_a$ such that $T(U_a,V_a)=(U_a,V_a)$. By the equations satisfied by $(U_a,V_a),$ we see that $(U_a,V_a)\in\mathcal{C}^1([-a,a],\mathbb{R}^{n+1}).$ Also, we have $U_a(-a)\geq\underline{U}(-a)>0$ and $(V_a)_i(-a)\geq\underline{V}_i(-a)>0$ for each $i\in[n],$ which implies that $(U_a,V_a)$ is positive for $s\in[-a,a]$ by the uniqueness.

    For the $\mathcal{C}^{1,1}$ estimate, note that $0<U_a(s)\leq K$ and $V_a(s)>0$ for all $s\in[-a,a].$ Next, we have 
    \begin{align}\label{UaLip}
        |U_a'(s)|\leq \dfrac{2K}{c^*}\left(2d_0+2\delta +\sum_{i=1}^n\chi_i\right)=:L_1.
    \end{align}
    Then we have
    \begin{align}\label{ULip}
        |U_a(s)-U_a(t)|\leq L_1|s-t|\mbox{ for all }s,t\in[-a,a].
    \end{align}

To give the $\mathcal{C}^{1,1}$ estimate for $U_a$, recall that $(V_a)_i(s)\leq q_i$ for each $s\in[-a,a],i\in[n],$ by our construction of the supersolution. Then there is an $M_1>0$ independent of $a$ such that
    $$\|V_a\|_{\mathcal{C}([-a,a],\mathbb{R}^n)}\leq M_1.$$
    By the equations satisfied by $V_a,$ we see that
    $$|(V_a)_i'(s)|\leq\dfrac{2}{c^*}\Big(2d_iM_1+\sup_{u\in[0,K],v\in[0,M_1]^n}|g_i(u,v)|\Big)=:L_{2,i}.$$
    Then 
    \begin{align}\label{VLip}
        |(V_a)_i(s)-(V_a)_i(t)|\leq L_{2,i}|s-t|\mbox{ for each }s,t\in[-a,a].
    \end{align}

To establish global Lipschitz bounds, we define
\begin{align*}
    \widetilde{U_a}(s)=\begin{cases}
        U_a(a)\mbox{ if }s>a,\\
        U_a(s)\mbox{ if }-a\leq s\leq a,\\
        \underline{U}(s)\mbox{ if }s<-a;
    \end{cases}\quad (\widetilde{V_a})_i(s)=\begin{cases}
        (V_a)_i(a)\mbox{ if }s>a,\\
         (V_a)_i(s)\mbox{ if }-a\leq s\leq a,\\
        \underline{V}_i(s)\mbox{ if }s<-a,
        \end{cases}
        \quad i\in[n].
\end{align*}
Recall that $(U_a,V_a)(-a)=(\underline{U},\underline{V})(-a)$. Then $(\widetilde{U_a},\widetilde{V_a})$ is Lipschitz continuous on $\mathbb{R},$ where $\widetilde{V_a}=((\widetilde{V_a})_1,\dots,(\widetilde{V_a})_n).$

Now, we are ready to complete the estimate for $U_a$.
    Note that for $s,t\in[-a,a]$,
    \begin{align*}
         c|U_a'(s)-U_a'(t)|\leq& d_0\bigg|\int_{\mathbb{R}}J(s')(\widetilde{U_a}(s-s')-\widetilde{U_a}(t-s'))ds'\bigg|+(d_0+\delta)|U_a(s)-U_a(t)|\\&+|g_0(U_a,V_a)(s)-g_0(U_a,V_a)(t)|.
     \end{align*}
     By \eqref{UaLip} and the Lipschitz continuity of $\widetilde{U_a}$, we can easily pick an $\widetilde{L_1}>0$, independent of $a,$ such that  
     \begin{align*}
         |\widetilde{U_a}(s)-\widetilde{U_a}(t)|\leq \widetilde{L_1}|s-t|\mbox{ for all } s,t \in\mathbb{R}.
     \end{align*}
     Then
     \begin{align}\label{complicated}
         \bigg|\int_{\mathbb{R}}J(s')(\widetilde{U_a}(s-s')-\widetilde{U_a}(t-s'))ds'\bigg|\leq\|J\|_{L^1(\mathbb{R})}\widetilde{L_1}|s-t|\mbox{ for all } s,t \in\mathbb{R}.
     \end{align}

    Now, let
    \begin{align*}
        L_g:=\sum_{i,j=1}^nK(g_{ij}^0)'(0),\quad L_2:=\sum_{i=1}^nL_{2,i},\quad\chi:=\sum_{i=1}^n\chi_i.
    \end{align*}
    By assumptions (C2) and (C4), each $(g_{ij}^0)'$ is nonnegative and nonincreasing on $[0,\infty)$ and thus,
    \begin{align}\label{Lgbound}
        |D_vg_i^0(u,v)|_1\leq L_g\mbox{ for all }u\in[0,K],\,v\geq 0_n\mbox{ and }i\in[n].
    \end{align}
    Then by \eqref{Lgbound}, \eqref{ULip} and \eqref{VLip}, we have 
    \begin{align*}
        |g_0(U_a,V_a)(s)-g_0(U_a,V_a)(t)|\leq&\chi\cdot|U_a(s)-U_a(t)|+U_a(t)\left|\sum_{i,j=1}^n\left(g_{ij}^0((V_a)_j(s))-g_{ij}^0((V_a)_j(t))\right)\right|\\
        \leq&\chi\cdot|U_a(s)-U_a(t)|+L_g\cdot|V_a(s)-V_a(t)|_1\\
        \leq &(\chi L_1+L_g L_2)|s-t|.
    \end{align*}
    Let
    \begin{align*}
         L_1':=\dfrac{2}{c^*}\left(d_0\|J\|_{L^1(\mathbb{R})}\widetilde{L_1}+(d_0+\delta) L_1+(\chi L_1+L_g L_2)\right).
    \end{align*}
    Then
    \begin{align*}
        |U_a'(s)-U_a'(t)|\leq L_1'|s-t|\mbox{ for each }s,t\in[-a,a].
    \end{align*}
    
    Finally, we complete the estimate for $V_a'$. 
    Note that 
    $$|\widetilde{V_a}(s)|_1
    \leq\sum_{i=1}^n\overline V_i(s)
    \leq Q\mbox{ for each }s\in\mathbb{R},$$
    where 
    $Q:=\sum_{i=1}^nq_i.$
    
    As we have done in \eqref{complicated}, we can calculate
    \begin{align*}
        \bigg|\int_{\mathbb{R}}J(s')(\widetilde{V_a}(s-s')-\widetilde{V_a}(t-s'))ds'\bigg|_1\leq\|J\|_{L^1(\mathbb{R})}\widetilde{L_2}|s-t|,
    \end{align*}
    where $\widetilde{L_2}>0$ is a constant independent of $a.$
Similarly as $L_1',$ we let
    \begin{align*}
    L_2':=
    \dfrac{2}{c^*}\left((\max_{i\in[n]}d_i)\big(\|J\|_{L^1(\mathbb{R})}\widetilde{L_2}+L_2\big)+(\chi L_1+L_g L_2)+L_2\max_{j\in[n]}\sum_{i=1}^n|m_{ij}|\right). 
\end{align*}
Then
\begin{align*}
        |V_a'(s)-V_a'(t)|_1\leq L_2'|s-t|\mbox{ for each }s,t\in[-a,a].
    \end{align*}
Therefore, the desired estimate follows.
\end{proof}

Finally, we are ready to prove Proposition \ref{superexistence}.
\begin{proof}[Proof of Proposition \ref{superexistence}]
    For all large $k\in\mathbb{N},$ pick a  positive $(U_k,V_k)\in \Gamma_k\cap\mathcal{C}^1([-k,k],\mathbb{R}^{n+1})$ such that $T(U_k,V_k)=(U_k,V_k).$ By Lemma \ref{c1estimates}, we can conclude, by passing to a subsequence, that
    \eqref{epidemictws} admits a nonnegative bounded solution $(U_c,V_c)$ on $\mathbb{R}$ with 
    $$\underline{U}\leq U_c\leq K\mbox{ and }\underline{V}_i\leq (V_c)_i\leq\overline{V}_i\mbox{ on }\mathbb{R}\mbox{ for each }i\in[n].$$
    Since $(\underline{U},\underline{V})(-\infty)=(\overline{U},\overline{V})(-\infty)=E_0,$ we see that $(U_c,V_c)(-\infty)=E_0.$ 
     Finally, since $(V_c)_i(\underline{s}_i-1)\geq\underline{V}_i(\underline{s}_i-1)>0$ for each $i\in[n]$, we conclude that $(U_c,V_c)$ is positive on $\mathbb{R}$ by Lemma \ref{lmapositive} (b).
     Moreover,
     \eqref{unibddwrtc} is straightforward from the proof of Lemma \ref{c1estimates}, as the constant $C$ can be chosen uniformly for $c$ in any bounded subset of $(c^*,\infty)$.
\end{proof}

\section*{Acknowledgements}
CHW is supported by the National Science and Technology Council of Taiwan (NSTC 113-2628-M-A49-004-MY4) and the National Center for Theoretical Sciences. 

\printbibliography
\end{document}